\documentclass[12pt, a4paper, twoside, openright]{book}

\usepackage{vuwthesis} 

\usepackage{palatino} 

\usepackage{url} 

\usepackage{caption}
\usepackage{subcaption}

\usepackage{float}

\makeatletter
 \def\@textbottom{\vskip \z@ \@plus 1pt}
 \let\@texttop\relax
\makeatother

\usepackage{etoolbox}
\apptocmd{\sloppy}{\hbadness 10000\relax}{}{}

\usepackage{amsfonts}
\usepackage{amsthm}
\usepackage{mathtools}
\usepackage{amssymb} 
\usepackage{ mathrsfs } 
\allowdisplaybreaks 

\DeclareMathOperator{\tr}{\text{tr}}

\let\emptyset\varnothing

\usepackage{makeidx}
\makeindex	
\newcommand{\newterm}[1]{\textit{#1}\index{#1}}

\usepackage[shortlabels]{enumitem}

\setlist{itemsep=.01em}

\newtheorem{thm}{T}
\usepackage{adjustbox} 

\usepackage{xcolor} 
\usepackage[normalem]{ulem}
\newcommand{\oldc}[1]{}
\newcommand{\newc}[1]{#1}

\numberwithin{thm}{chapter}

\newtheorem{theorem}[thm]{Theorem}
\newtheorem{lemma}[thm]{Lemma}
\newtheorem{corollary}[thm]{Corollary}

\newtheorem{proposition}[thm]{Proposition}

\theoremstyle{remark} 
\newtheorem{remark}[thm]{Remark}
\newtheorem{example}[thm]{Example}
\newtheorem{notation}[thm]{Notation}
\theoremstyle{plain} 

\theoremstyle{definition} 
\newtheorem{definition}[thm]{Definition}
\theoremstyle{plain} 

\newtheorem*{theorem*}{Theorem}
\newtheorem*{corollary*}{Corollary}
\newtheorem*{proposition*}{Proposition}
\newtheorem*{lemma*}{Lemma}
\newtheorem*{conjecture*}{Conjecture}
\newtheorem*{problem*}{Problem}

\theoremstyle{remark} 
\newtheorem*{remark*}{Remark}
\newtheorem*{notation*}{Notation}
\newtheorem*{example*}{Example}
\theoremstyle{plain} 

\theoremstyle{definition} 
\newtheorem*{definition*}{Definition}
\theoremstyle{plain} 

\makeatletter
\newtheorem*{rep@theorem}{\rep@title}
\newcommand{\newreptheorem}[2]{%
\newenvironment{rep#1}[1]{%
 \def\rep@title{#2 \ref{##1}}%
 \begin{rep@theorem}}%
 {\end{rep@theorem}}}
\makeatother

\newreptheorem{theorem}{Theorem}
\newreptheorem{proposition}{Proposition}

\newtheorem{innercustomthm}{Theorem}
\newenvironment{customthm}[1]
  {\renewcommand\theinnercustomthm{#1}\innercustomthm}
  {\endinnercustomthm}

\usepackage{tikz}
\usepackage{tikz-cd} 
\usetikzlibrary{arrows}
\usetikzlibrary{arrows.meta}

\newcommand*\circled[1]{\tikz[baseline=(char.base)]{
   \node[shape=circle,draw,inner sep=1pt, thick] (char) {#1};}}
\newcommand*\circledRed[1]{\text{\normalsize $\color{rgb,255:red,180;green,51;blue,51}\circled{#1}$}}

\begin{document}

\frontmatter


\title{AF Embeddability of the C*-Algebra of a Deaconu-Renault Groupoid}
\author{Rafael Pereira Lima}

\subject{Mathematics}
\abstract{We study Deaconu-Renault groupoids corresponding to surjective local homeomorphisms on  locally compact, Hausdorff, second countable, totally disconnected spaces, and we characterise  when the C*-algebras of these groupoids are AF embeddable. Our main result generalises theorems in the literature for graphs and for crossed products of commutative C*-algebras by the integers. We give a condition on the surjective local homeomorphism that characterises the AF embeddability of the C*-algebra of the associated Deaconu-Renault groupoid. In order to prove our main result, we analyse homology groups for AF groupoids, and we prove a theorem that gives an explicit formula for the isomorphism of these groups and the corresponding K-theory. This isomorphism generalises \cite{FKPS, Matui}, since we give an explicit formula for the isomorphism and we show that it preserves positive elements.}

\phd


\maketitle

\chapter*{Acknowledgments}\label{C:ack} 

The completion of this thesis is not the result of an independent work. So I would like to recognise the work and help that I have got during these three years. My supervisors, Astrid an Huef and Lisa Orloff Clark, were present during the process. Thank you for meeting regularly with me, for reading carefully the text, for the patience, and for the support to attend conferences.

Camila Sehnem has contributed during her stay in New Zealand by attending my weekly meetings with the supervisors, and giving some helpful suggestions. 

I would like also to thank Rodrigo Bissacot, Paulo Cordaro and Cristi\'an Ortiz for the reference letters when I was applying for the PhD. There are other people who contributed to my career, like Steven Archer and Mark McGuinness. Iain Raeburn offered an interesting course on C*-algebras during my first year.

I would like to thank the colleagues who have been nice to me. I am grateful that I still keep in touch with the research group from my master's degree Working with Rodrigo Bissacot, Bruno Kimura and Thiago Raszeja motivates me to stay in academia. My family has been very supportive, as always. 

\tableofcontents


\mainmatter



\chapter{Introduction}

A problem that has received much interest during the last decades is the classification of C*-algebras. This problem consists of finding an invariant that identifies all C*-algebras of a certain class up to isomorphism. One of the first results in the area is Glimm's \cite{Glimm} classification of some uniformly hyperfinite C*-algebras. Glimm's result was interesting because uniformly hyperfinite algebras are the C*-algebraic version of the $\text{II}_1$ factors for von Neumann algebras, but unlike the von Neumann case \cite{MurrayvonNeumann} (see also \cite[Theorem V.1.13]{Connes}), uniformly hyperfinite C*-algebras are not all isomorphic.

Dixmier \cite{Dixmier-classification} then generalised this classification by studying a class of C*-algebras called matroid C*-algebras. In \cite{BratteliAF}, Bratteli defined a more general class \oldc{the class} of C*-algebras, called approximately finite dimensional C*-algebras (or AF algebras). An AF algebra is the completion of an increasing union of finite dimensional C*-algebras. Elliott \cite[Theorem 4.3]{ElliottAF} (see \cite[Theorem IV.4.3]{Davidson} for a more modern formulation) classified all AF algebras by their K-theory.

Since Elliott's classification of AF algebras, different types of C*-algebras have been classified by invariants. Several stably finite C*-algebras \cite{GLN1, GLN2} are classifiable. Kirchberg \cite{Kirchberg} and Phillips \cite{Phillips} classified all unital, simple, separable, nuclear and purely infinite C*-algebras (called Kirchberg algebras) satisfying the Universal Coefficient Theorem by their K-theory. See Strung's book on the classification of C*-algebras \cite{Strung-classification} for more details.

In certain types of C*-algebras, a dichotomy exists between purely infinite and stably finite. In other words, for certain classes of C*-algebras, a C*-algebra is either stably finite or purely infinite. This dichotomy holds for reduced C*-algebras of certain minimal, topological principal, ample groupoids \cite[Theorem D]{BonickeLi}, \cite[Theorem 7.4]{RainoneSims}, \cite{Ma}.  Additionally to the classification programme, the possibility of finding a dichotomy for large classes of C*-algebras motivates us to study purely infinite and stably finite C*-algebras, and related properties.

For different C*-algebras, the notions of stably finite, quasidiagonal and AF embeddable C*-algebras are equivalent. This equivalence was given by Pimsner \cite[Theorem 9]{Pimsner} for crossed products of commutative C*-algebras by the integers. Brown \cite{Brown-AFE} generalised this result to crossed products of AF algebras to the integers, and Schafhauser \cite[Theorem 6.7]{Schafhauser-topologicalgraphs} proved the equivalence for topological graph algebras with no sinks. Topological graph algebras are very general and cover the C*-algebras of Deaconu-Renault groupoids, that we study in this thesis. So, the equivalence of these notions hold for the C*-algebras we are interested.

In this thesis we study C*-algebras that are AF embeddable \newc{(or AFE algebras)}, i.e., C*-algebras that are subalgebras of AF algebras. 

There are several results in the literature studying C*-algebras that are AF, AF embeddable or purely infinite. For example, graph loops characterises when a graph algebra is approximately finite dimensional \cite{KPR}. Similarly, Schafhauser \cite{Schafhauser-AFEgraph} describes when a graph algebra is AF embeddable, while \cite[Proposition 5.3]{BPRS} gives a graph condition that makes the corresponding C*-algebra purely infinite. Szyma\'{n}ski \cite{Szymanski} showed that every stable, purely infinite, simple and classifiable C*-algebra with $K_1$ torsion-free is a graph algebra. 

Kumjian and Pask \cite{KP-kgraph} study higher rank graph algebras and give sufficient conditions that make the C*-algebras AF \oldc{dimensional} or purely infinite. \cite{CaHS}, \cite[Corollary 6.2]{Schafhauser-AFEsimple} characterise when certain higher rank graph algebras are AF embeddable.

 In this thesis we study C*-algebras of Deaconu-Renault groupoids. These algebras generalise graph algebras \cite{KPRR} and are examples of topological graph algebras \cite[Proposition 10.9]{KatsuraII}. In Appendix \ref{appendix:Deaconu-Renault}, we describe this class of groupoids with more detail.

 The main result of this thesis is the characterisation of the AF embeddability for C*-algebras of Deaconu-Renault groupoids, given some topological conditions, as we state below.

\begin{customthm}{1}
\label{thm:AFEDR}
Let $\sigma$ be a surjective local homeomorphism on a locally compact, Hausdorff, second countable, totally disconnected space $X$. Denote by $\mathcal{G}$ be the Deaconu-Renault groupoid corresponding to $\sigma$. Then the following are equivalent:
\begin{enumerate} [(i)]
\item $C^*(\mathcal{G})$ is AF embeddable,
\item $C^*(\mathcal{G})$ is quasidiagonal,
\item $C^*(\mathcal{G})$ is stably finite,
\item $\mathrm{Im}(\sigma_\ast - \mathrm{id}) \cap C_c(X, \mathbb{N}) = \lbrace 0 \rbrace$,
\end{enumerate}
where the map $\sigma_\ast: C_c(X, \mathbb{Z}) \rightarrow C_c(X, \mathbb{Z})$ is defined by
\begin{align*}
\sigma_\ast(f)(x) = \sum_{y: \sigma(y) = x} f(y)
\hspace{20pt}
\text{for $x \in X$, $f \in C_c(X, \mathbb{Z})$.}
\end{align*}
\end{customthm}

This theorem generalises similar results for graph algebras  \cite{Schafhauser-AFEgraph}, and for crossed products of unital commutative C*-algebras by the integers \cite[Theorem 11.5]{Brown-quasidiagonal}, \cite[Theorem 9]{Pimsner}. It is analogous to some results for higher-rank graphs \cite{CaHS}, \cite[Corollary 6.2]{Schafhauser-AFEsimple}. C*-algebras of Deaconu-Renault groupoids are examples of topological graph algebras. When $X$ is compact, then the corresponding topological graph algebra is unital and, in this case, Theorem \ref{thm:AFEDR} is an example of \cite[Theorem 6.7]{Schafhauser-topologicalgraphs}, that characterises when a unital topological graph algebra is AF embeddable.

Note that the construction of the Deaconu-Renault groupoid $\mathcal{G}$ depends on the surjective local homeomorphism $\sigma$. The contribution of Theorem \ref{thm:AFEDR} is that it provides a condition on $\sigma$ that is equivalent to the AF embeddability of the C*-algebra of the corresponding Deaconu-Renault groupoid.

In order to prove Theorem \ref{thm:AFEDR}, we apply the following result by Brown \cite{Brown-AFE}, that gives an equivalent condition to AF embeddability of crossed products of AF algebras by the integers:

\begin{definition*}
If $B$ is an AF algebra and $\beta \in \mathrm{Aut}(B)$, then we denote by $H_\beta$ the subgroup of $K_0(B)$ given by all elements of the form $K_0(\beta)(x) - x$ for $x \in K_0(B)$.
\end{definition*}

\begin{theorem} (Brown) \cite[Theorem 0.2]{Brown-AFE}
\label{thm:brown}
If $B$ is an AF algebra and \oldc{$\alpha \in \mathrm{Aut}(B)$}\newc{$\beta \in \mathrm{Aut}(B)$}, then the following are equivalent:
\begin{enumerate} [(a)]
\item $B \rtimes_\beta \mathbb{Z}$ is AF embeddable,
\item $B \rtimes_\beta \mathbb{Z}$ is quasidiagonal,
\item $B \rtimes_\beta \mathbb{Z}$ is stably finite,
\item $H_\beta \cap K_0(B)^+ = \lbrace 0 \rbrace$.
\end{enumerate}
\end{theorem}

Note that AF embeddability, quasidiagonality and the stable finite properties are preserved by stable isomorphisms. If we choose an AF algebra $B$ such that $B \rtimes_\beta \mathbb{Z}$ is stably isomorphic to $C^*(\mathcal{G})$, then we can apply Brown's theorem to get a characterisation of the AF embeddability of $C^*(\mathcal{G})$. Indeed, we let $c: \mathcal{G} \rightarrow \mathbb{Z}$ be the canonical cocycle, and we let $B = C^*(\mathcal{G}(c))$ be the C*-algebra of the skew product $\mathcal{G}(c)$. This C*-algebra is AF by \cite[Lemma 6.1]{FKPS}. Theorem III.5.7 of \cite{Renault} gives the isomorphism
\begin{align*}
B \cong C^*(\mathcal{G}) \rtimes_{\alpha^c} \mathbb{T},
\end{align*}
where $\alpha^c$ is the action of $\mathbb{T}$ on $C^*(\mathcal{G})$ corresponding to the cocycle $c$. We let $\beta$ be an automorphism on $B$ corresponding to the dual map $\widehat{\alpha}^c$ with respect to the isomorphism above. By applying the crossed products on both sides of the equivalence above with the respective maps, we get the isomorphism $B \rtimes_\beta \mathbb{Z} \cong C^*(\mathcal{G}) \rtimes_{\alpha^c} \mathbb{T} \rtimes_{\widehat{\alpha}^c} \mathbb{Z}$, and by Takai duality, we have that $B \rtimes_\beta \mathbb{Z}$ is stably isomorphic to $C^*(\mathcal{G})$. So, we apply Brown's theorem to $C^*(\mathcal{G})$ and we get the equivalence of items (i)--(iii) of Theorem \ref{thm:AFEDR} with item (d) of Brown's theorem.

This equivalence gives a characterisation of when the C*-algebra is AF embeddable. Note that item (d) depends on the K-theory, which is very abstract. We use the theory of homology groups of groupoids in order to show that (d) is equivalent to condition (iv) of Theorem \ref{thm:AFEDR}. By applying results from \cite{FKPS, Matui} on homology groups, we get the commutative diagram below.

\begin{equation}
\label{eqn:diagram}
 \centering
\begin{tikzcd}
      K_0(B) \arrow[r, "K_0(\beta)"] \arrow{d}[swap]{\cong} 
      & K_0(B) \arrow[d, "\cong"]\\
      H_0(\mathcal{G}(c))  \arrow[r, "{[\widetilde{\beta}]}_{\mathcal{G}(c)}"] \arrow{d}[swap]{\cong} & H_0(\mathcal{G}(c))  \arrow[d,"\cong"]\\
      H_0(c^{-1}(0))  \arrow{r}{[\sigma_\ast] }[swap]{} & H_0(c^{-1}(0))  
\end{tikzcd}
\end{equation}

The existence of the isomorphism $K_0(B) \rightarrow H_0(\mathcal{G}(c))$ is given by \cite[Corollary 5.2]{FKPS} and is such that $[1_V]_0 \mapsto [1_V]$ for every compact open $V \subset \mathcal{G}^{(0)}$.  The groupoid $c^{-1}(0)$ is the kernel of the cocycle $c:\mathcal{G} \rightarrow \mathbb{Z}$ defined by $c(x,k,y) = k$. We define the isomorphism $H_0(\rho_\varphi): H_0(\mathcal{G})(c)) \rightarrow H_0(c^{-1}(0))$ in Chapter \ref{section:diagram}. The homomorphisms $K_0(\beta)$ and $[ \widetilde{\beta} ]_{\mathcal{G}(c)}$ are group homomorphisms induced by $\beta$, and the group homomorphism $[\sigma_\ast]$ is induced by the map $\sigma_\ast$ of Theorem \ref{thm:AFEDR}.

By studying the diagram and by applying properties of homology groups of groupoids, we can show that items (d) and (iv) are equivalent, proving Theorem \ref{thm:AFEDR}.

Note, that we need to know the positive elements of $K_0(B)$ in order to study item (d). So, when we apply the commutative diagram, we need to guarantee that the isomorphisms above takes positive elements from one group to the another. Although \cite[Corollary 5.2]{FKPS} gives the existence of the group isomorphism $K_0(B) \cong H_0(\mathcal{G}(c))$, we need this isomorphism to be an ordered group isomorphism. We prove that this isomorphism is indeed a ordered group isomorphism using different techniques, and we give an explicit formula for the isomorphism, stated as follows:
\begin{definition*}
Let $\sigma: X \rightarrow Y$ be a surjective local homeomorphism on locally compact, Hausdorff, second countable, totally disconnected spaces, and let $\varphi: Y \rightarrow X$ be a continuous section of $\sigma$. We define the map $\mathrm{tr}_\varphi: \mathcal{P}_\infty(C^*(R(\sigma))) \rightarrow C_c(X, \mathbb{N})$ by
\begin{itemize}
\item $\mathrm{tr}_\varphi(p)(x) = 1_{X_\varphi}(x) \displaystyle\sum_{y: \sigma(x) = \sigma(y)} p(y)$, \hspace{15pt} for \oldc{$p \in \mathcal{P}(C^*(R(\sigma)))$}\newc{$p \in C^*(R(\sigma))$}, $x \in X$, and
\item $\mathrm{tr}_\varphi(p)(x) = \displaystyle\sum_{i=1}^n \mathrm{tr}_\varphi(p_{ii})(x)$, \hspace{52pt} for $p \in \mathcal{P}_n(C^*(R(\sigma)))$, $n \geq 1$,
\end{itemize}
where we identify $C^*(R(\sigma))$ as a subset of $C_0(R(\sigma))$ by applying the $j$ map from \cite[Proposition II.4.2]{Renault} and \cite[Proposition 9.3.3]{sims2017hausdorff}.
\end{definition*}

\begin{customthm}{2}
\label{thm:HK}
Let $G$ be an AF groupoid, written as the inductive limit $G = \displaystyle\lim_\rightarrow G_n$ of elementary groupoids. Then, for every $n$, there exists \newc{a locally compact, Hausdorff, second countable, totally disconnected space $Y_n$, and there is} a surjective local homeomorphism $\sigma_n: G_n^{(0)} \rightarrow Y_n$ such that \oldc{$G_n = R(\sigma)$}\newc{$G_n = R(\sigma_n)$}. Moreover,
\begin{enumerate}
\item there exists a unique isomorphism $\nu: K_0(C^*(G)) \rightarrow H_0(G)$ defined by
\begin{align*}
\nu([p]_0) = [\mathrm{tr}_{\varphi_n}(p)]
\end{align*}
for $p \in \mathcal{P}_n(C^*(G_n))$, and such that $\varphi_n$ is an arbitrary continuous section of $\sigma_n$;
\item the map $\nu$ is an isomorphism of ordered groups, i.e., $\nu(K_0(C^*(G))^+) = H_0(G)^+$; and
\item the isomorphism $\nu$ is precisely the map given by \cite[Corollary 5.2]{FKPS}. In other words, $\nu([1_V]_0) = [1_V]$ for all $V \subset G^{(0)}$ compact open.
\end{enumerate}
\end{customthm}

This theorem also generalises \cite[Theorem 4.10]{Matui} to the AF groupoids with non-compact unit spaces.

We organise this thesis as follows: In Chapter \ref{section:crossedproducts}, we apply results from the theory of crossed products to give a more modern proof for the isomorphism \cite[Proposition II.5.1]{Renault} by Renault. More specifically, we show that, for a given continuous cocycle $c: G \rightarrow \mathbb{Z}$ on a certain groupoid $G$, the C*-algebra of the skew-product groupoid $G(c)$ is \newc{isomorphic to a crossed product} of $C^*(G)$ by the unit circle.

In Chapter \ref{section:homology}, we study the theory of homology groups of groupoids. In Chapter \ref{section:K0H0}, we prove Theorem \ref{thm:HK}, giving a isomorphism between the homology groups for AF groupoids and the K-theory of the C*-algebras of AF groupoids. Similar results have been proved \oldc{by}\newc{in} \cite[Corollary 5.2]{FKPS}, \cite[Theorem 4.10]{Matui}. Our innovation is that we give an explicit formula for this isomorphism, and we show that it preserves positive elements from one group to another.

In Chapter \ref{section:diagram}, we apply the results from Chapters 3 and 4 to prove the commutative diagram \eqref{eqn:diagram}.

In Chapter \ref{section:result}, we prove Theorem \ref{thm:AFEDR}, the main theorem of our thesis, that characterises when the C*-algebra of certain Deaconu-Renault \newc{groupoids} is AF embeddable. In Chapter \ref{section:examples}, we show how our main theorem generalises some known results in the literature.

The appendices introduce some of the definitions and known results that we apply in the thesis.

\chapter{Skew-product groupoids and crossed products}
\chaptermark{skew-products and Crossed products} 
\label{section:crossedproducts}

Let $G$ be a locally compact, Hausdorff, second countable, \'etale groupoid, \newc{and} let $c: G \rightarrow \mathbb{Z}$ be a continuous cocycle. In this chapter, we study the isomorphism between $C^*(G) \rtimes_{\alpha^c} \mathbb{T}$ and $C^*(G(c))$, for a groupoid with certain topological properties. Although this result is given by Theorem III.5.7 of Renault's book \cite{Renault}, we will give a more modern proof to this isomorphism.

As explained in the introduction, we use this isomorphism to understand when the C*-algebra of a Deaconu-Renault groupoid $\mathcal{G}$ is AFE by applying Brown's theorem (Theorem \ref{thm:brown}) to the AF algebra $B = C^*(\mathcal{G}(c))$, for the canonical cocycle $c: \mathcal{G} \rightarrow \mathbb{Z}$ given by $c(x,k,y) = k$.

Our isomorphism depends on results from Williams's book \cite{Williams-crossedproducts}. In particular, we apply \cite[Theorem 2.39]{Williams-crossedproducts} to show that a certain covariant homomorphism $(\rho, u)$ induces the homomorphism $\rho \rtimes u: C^*(G) \rtimes_{\alpha^c} \mathbb{T} \rightarrow M(C^*(G(c)))$. Then we show that this homomorphism is identified with the isomorphism $C^*(G) \rtimes_{\alpha^c} \mathbb{T} \cong C^*(G(c))$ that we are looking for.

We say that our proof is more modern than Renault's because we use techniques of crossed products that can be found in Williams's book, that was published in 2007. In contrast, Renault's book was \oldc{publish}\newc{published} in 1980, but it is also more general. He proves the isomorphism $C^*(G, \sigma) \rtimes_\alpha \widehat{A} \cong C^*(G(c), \sigma)$, where the groupoid is not necessarily \'etale, and must be equipped with a 2-cocyle $\sigma$. Also, $A$ is an arbitrary locally compact abelian group. His proof uses a completion of $C_c(G(c))$ with respect to a certain norm to find the isomorphism. 

For our purposes, it is sufficient to prove the case $A = \mathbb{Z}$.


After proving the isomorphism, we also find a formula for the action $\beta$ of $\mathbb{Z}$ on $C^*(G(c))$, induced by $\widehat{\alpha}^c$. This formula is used in Chapter \ref{section:diagram} to study the commutative diagram \eqref{eqn:diagram} of page \pageref{eqn:diagram}.


In this chapter, we need many definitions and technical results. Since we are focusing on the proof of the isomorphism, the results are stated without much explanation. See the references for more details.

In Section \ref{section:crossedproducts:skewproduct}, we define the skew product groupoid $G(c)$ as an example of \cite[Definition I.1.6]{Renault} by Renault. Then, in Section \ref{section:crossedproducts:crossedproducts}, we state some results on crossed products from Williams's book \cite[Chapters 1, 2 and 7]{Williams-crossedproducts}. Finally, in Section \ref{subsection:crossedproducts:isomorphism}, we combine the results in this chapter to prove the isomorphism $C^*(G) \rtimes_{\alpha^c} \mathbb{T} \cong C^*(G(c))$.


\par\nobreak\section{The skew-product groupoid}
\label{section:crossedproducts:skewproduct}


Here we define the skew-product groupoid using \cite[Definition I.1.6]{Renault} from Renault's book. Our definition is more specific because here we consider a cocycle on the integers, while Renault considers more general groups.

\begin{definition}
\label{def:Gc}
Let $G$ be a groupoid and $c: G \rightarrow \mathbb{Z}$ a homomorphism. Define the sets
\begin{align*}
G(c) = \lbrace (g,a) : g \in G, a \in \mathbb{Z} \rbrace
\hspace{10pt}\text{and}\hspace{10pt}
G(c)^{(0)} = G^{(0)} \times \mathbb{Z}.
\end{align*}
Define the maps $r, s: G(c) \rightarrow G^{(0)} \times \mathbb{Z}$ by $r(g,a) = (r(g), a)$ and $s(g,a) = (s(g), a + c(g))$. Consider the set
\begin{align}
\label{eqn:Gc2}
G(c)^{(2)} = \lbrace ((g,a), (h,a + c(g))): (g,h) \in G^{(2)}, a \in \mathbb{Z} \rbrace,
\end{align}
and define the product $G(c)^{(2)} \rightarrow G(c)$ and the inverse $G(c) \rightarrow G(c)$ by
\begin{align}
(g,a)(h, a + c(g)) = (gh,a)
\hspace{15pt}\text{and}\hspace{15pt}
(g,a)^{-1} = (g^{-1},a+c(g))
\label{eqn:Gcprod},
\end{align}
respectively. Then $G(c)$ becomes a groupoid, called a \newterm{skew-product}.
\end{definition}

Now we show that, if $c$ is a continuous cocycle and we equip $G(c)$ with the induced topology from $G \times \mathbb{Z}$, then $G(c)$ inherits the usual topogical properties of $G$. Recall that we assume that $G$ is locally compact, Hausdorff, second countable and \'etale.

\begin{lemma}
\label{lemma:Gctopology}
Equip $G(c)$ with the topology induced from $G \times \mathbb{Z}$. Then $G(c)$ is also a locally compact, Hausdorff, second countable, \'etale groupoid. Moreover, if $G$ is totally disconnected, then $G(c)$ is totally disconnected, too. If $G$ is measurewise amenable, then $G(c)$ is also measurewise amenable.
\end{lemma}
\begin{proof}
First we show that $G(c)$ is a topological groupoid. Since $c$ is continuous, then \eqref{eqn:Gcprod} shows that the operations on $G(c)$ are compositions of continuous functions. Hence, the groupoid operations are continuous, and thus $G(c)$ is a topological groupoid.

Now we show that $G(c)$ is \'etale. Fix $(g,a) \in G(c)$. Let $\mathcal{U} \subset G$ be an open bisection containing $g$ such that $c$ is constant on $\mathcal{U}$. Then $r$ is injective on $\mathcal{U} \times \lbrace a \rbrace$ and has image $r(\mathcal{U} \times \lbrace a \rbrace) = r(\mathcal{U}) \times \lbrace a \rbrace$, which is open. Similarly, $s$ is injective on $\mathcal{U} \times \lbrace a \rbrace$ and has image $s(\mathcal{U} \times \lbrace a \rbrace) = \mathcal{U} \times \lbrace a + c(g) \rbrace$ is also open. Thus $G(c)$ is \'etale.

Since $\mathbb{Z}$ is discrete, and $G$ is Hausdorff and second countable, then $G(c)$ is also Hausdorff and second countable. By the same argument, if $G$ is totally disconnected, then $G(c)$ has also a basis of compact open sets, which implies that $G(c)$ is totally disconnected.

If $G$ is measurewise amenable, then $G(c)$ is measurewise amenable by \cite[Proposition II.3.8]{Renault}.
\end{proof}

Note that, for $f, f_1, f_2 \in C_c(G(c))$ and $(g, a) \in G \times \mathbb{Z}$, the operations on $C_c(G(c))$ are given by
\begin{align*}
(f_1 \cdot f_2)(g,a)
&= \sum_{g_1 g_2 = g} f_1(g_1,a) f_2(g_2, a + c(g_1)), \text{ and }\\
f^*(g,a) &= \overline{f(g^{-1}, a + c(g))}.
\end{align*}


It follows from \cite[page 92]{Renault} that if $G$ is measurewise amenable, the full and the reduced C*-algebras of $G$ coincide.

\par\nobreak\section{Crossed products}
\label{section:crossedproducts:crossedproducts}

The definition of a crossed product $A \rtimes_\alpha Z$ depends on the C*-dynamical system $(A, Z, \alpha)$, where $A$ is a C*-algebra, $Z$ is a locally compact group and $\alpha$ decribes how $Z$ `acts' on $A$. We define $A \rtimes_\alpha Z$ as the completion of $C_c(Z, A)$, where the norm and operations on this space of functions are connected with $\alpha$.

Throughout this section we assume that the locally compact group $Z$ is either $\mathbb{Z}$ or $\mathbb{T}$. These are the only groups we use in this thesis when we apply crossed products. Moreover, this assumption allows us to skip some technical details from the construction of more abstract crossed products.

Here we use Chapters 1 and 2 of Dana Williams's book \cite{Williams-crossedproducts}.

\begin{definition}
A \newterm{C*-dynamical system} (or just a \newterm{dynamical system}) is a triple $(A, Z, \alpha)$ consisting of a C*-algebra $A$, a locally compact group $Z$ and a continuous homomorphism $\alpha: Z \rightarrow \mathrm{Aut\hphantom{.}}A$ with $z \mapsto \alpha_z$. We say that $(A, Z, \alpha)$ is \newterm{separable} if $A$ is separable.
\end{definition}

Sometimes we write $\alpha: Z \rightarrow \mathrm{Aut}\hphantom{.}A$ is a dynamical system to denote that $(A, Z, \alpha)$ is a dynamical system. If $Z = \mathbb{Z}$, we identify the action $\alpha$ with the automorphism $\alpha_1$. This identification makes sense because, for each $n \in \mathbb{Z}$, $\alpha_n = (\alpha_1)^n$.

\begin{example*}
Let $A$ be a unital C*-algebra, and let $u \in A$ be unitary. Define $\alpha \in \mathrm{Aut}\hphantom{.}A$ by $\alpha(a) = uau^*$. Then $\alpha$ is a dynamical system.
\end{example*}

The lemma below defines the operations on $C_c(Z,A)$. The theory of crossed products depends on Haar measures on locally compact groups. As we are assuming that $Z$ is either $\mathbb{Z}$ or $\mathbb{T}$, we fix the Haar measures on $\mathbb{Z}$ and $\mathbb{T}$ to be the counting measure and the Lebesgue measure, respectively.

\begin{lemma}
\cite[page 48]{Williams-crossedproducts}
\label{lemma:operationscrossed}
Let $(A, Z, \alpha)$ be a dynamical system and fix a Haar measure $\mu$ on $Z$, with $Z  = \mathbb{Z}, \mathbb{T}$. Given \oldc{$f, f_1, f_2 \in C(Z, A)$}\newc{$f, f_1, f_2 \in C_c(Z, A)$}, we define the \newterm{convolution} $f_1 \cdot f_2 \in C_c(Z, A)$ and the \newterm{involution} $f^* \in C_c(Z, A)$ by
\begin{align*}
f_1 \cdot f_2(z) = \int_Z f_1(w) \alpha_w(f_2(w^{-1}z)) d\mu(w)
\end{align*}
and
\begin{align*}
f^*(z) = \alpha_z(f(z^{-1})^*).
\end{align*}
With these operations, $C_c(Z, A)$ is a $\ast$-algebra.
\end{lemma}

Let us now study the notion of covariant representations. Later on we use these representations to equip $C_c(Z, A)$ with a norm such that the crossed product $A \rtimes_\alpha Z$ is defined as the completion of $C_c(Z, A)$.


\begin{definition}
\label{def:covariantrepresentation}
Let $(A, Z, \alpha)$ be a dynamical system and let $B$ be a C*-algebra. Then a \newterm{covariant homomorphism} of $(A, Z, \alpha)$ into $M(B)$ is a pair $(\pi, u)$ consisting of a homomorphism $\pi: A \rightarrow M(B)$ and a strongly continuous homomorphism $u: Z \rightarrow U M(B)$ such that
\begin{align*}
\pi(\alpha_z(a)) = u_z \pi(a) u_z^*
\hspace{20pt}
\text{$z \in Z$, $a \in A$}.
\end{align*}
If $B = B(H)$ is the space of bounded operators on a Hilbert space $H$, we say that $u$ is a \newterm{unitary representation} of $Z$ and that $(\pi, u)$ is a \newterm{covariant representation} of $(A, Z, \alpha)$.
\end{definition}


Note that \cite[Definition 2.37]{Williams-crossedproducts} of covariant homomorphisms is more general is considers $\mathcal{L}(X)$ instead of $M(B)$, where $X$ is a Hilbert $B$-module. See \cite[Remark 2.38]{Williams-crossedproducts} for more details.

We use the representations of the proposition below to equip $C_c(Z, A)$ with a norm and then define the crossed product as the completion of $C_c(Z, A)$. Throughout this section, we assume that $(A, Z, \alpha)$ is a dynamical system and that $\mu$ is a Haar measure on $Z$.

\begin{proposition}
\label{prop:integratedform}
\cite[Proposition 2.23]{Williams-crossedproducts}
Suppose that $(\pi,u)$ is a covariant representation of $(A, Z, \alpha)$ on $H$. Equip $C_c(Z, A)$ with the operations Lemma \ref{lemma:operationscrossed}. Then
\begin{align*}
\pi \rtimes u(f) = \int_Z \pi(f(z)) u_z d\mu(s)
\hspace{20pt}
\text{for $f \in C_c(Z,A)$}
\end{align*}defines a $L^1$-norm decreasing $\ast$-representation $\pi: C_c(Z,A) \rightarrow B(H)$ called the \newterm{integrated form} of $(\pi, u)$.
\end{proposition}

Now we define the crossed product $A \rtimes_\alpha Z$.

\begin{lemma}
\label{lemma:crossedproducts}
\cite[Lemma 2.27]{Williams-crossedproducts}
Equip $C_c(Z, A)$ with the operations Lemma \ref{lemma:operationscrossed}. Define for each $f \in C_c(Z, A)$,
\begin{align*}
\Vert f \Vert
= \sup \lbrace \Vert \pi \rtimes u(f) \Vert: (\pi, u)
\text{ is a covariant representation of }
(A, Z, \alpha) \rbrace.
\end{align*}
Then $\Vert \cdot \Vert$ is a norm on $C_c(Z, A)$ called the \newterm{universal norm}. The universal norm is dominated by the $\Vert \cdot \Vert_1$-norm, and the completion of $C_c(Z, A)$ with respect to $\Vert \cdot \Vert$ is a C*-algebra $A \rtimes_\alpha Z$, called the \newterm{crossed product} of $A$ by $Z$.
\end{lemma}

It follows from \cite[Theorem 1.57]{Williams-crossedproducts} that, up to isomorphism, the definition of the crossed products does not depend on the choice of the Haar measure $\mu$.

\begin{example}
Let $X$ be a locally compact, Hausdorff, second countable space, and let $\sigma: X \rightarrow X$ be a homeomorphism. We use the notation $C_0(X) \rtimes_\sigma \mathbb{Z}$ to denote the crossed product $C_0(X) \rtimes_{\alpha} \mathbb{Z}$ with respect to the automorphism $\alpha: C_0(X) \rightarrow C_0(X)$ given by
\begin{align*}
\alpha_\sigma(f)(x) = f(\sigma^{-1}(x))
\hspace{20pt}
\text{for $f \in C(X), x \in X$.}
\end{align*}
\end{example}


We use the proposition below in Section \ref{subsection:crossedproducts:isomorphism} to find the isomorphism $C^*(G) \rtimes_{\alpha^c} \mathbb{Z} \cong C^*(G(c))$.

\begin{proposition}
\cite[Proposition 2.39]{Williams-crossedproducts}
\label{prop:covariant}
Let $B$ be a C*-algebra. If $(\pi,u)$ is a covariant homomorphism of $(A, Z, \alpha)$ into $M(B)$, then there exists a unique homomorphism $\pi \rtimes u: A \rtimes_\alpha Z \rightarrow M(B)$ such that
\begin{align*}
\pi \rtimes u(f) = \int_Z \pi(f(z))u_z d\mu(z)
\hspace{20pt}\text{$f \in C_c(Z, A)$}.
\end{align*}
\end{proposition}

Now we study a condition that allows us to extend an isomorphism of two C*-algebras $A$, $B$ to an isomorphism of crossed products $A \rtimes Z \cong B \rtimes Z$.

\begin{definition}
Two dynamical systems $(A,Z,\alpha)$ and $(B, Z, \beta)$ are \newterm{equivariantly isomorphic} if there is an isomorphism $\varphi: A \rightarrow B$ such that $\varphi(\alpha_z(a)) = \oldc{\beta(\varphi(a))}\newc{\beta_z(\varphi(a))}$ for all $z \in Z$ and $a \in A$. We call $\varphi$ an \newterm{equivariant isomorphism}.
\end{definition}

\begin{lemma} \cite[Lemma 2.65]{Williams-crossedproducts}
\label{lemma:crossedproductiso}
Suppose that $\varphi$ is an equivariant isomorphism of $(A, Z, \alpha)$ onto $(B, Z, \beta)$. Then the map $\varphi \rtimes \mathrm{id}: C_c(Z, A) \rightarrow C_c(Z, B)$ defined by
\begin{align*}
\varphi \rtimes \mathrm{id}(f)(z) = \varphi(f(z))
\end{align*}
extends to an isomorphism of $A \rtimes_\alpha Z$ onto $B \rtimes_\alpha Z$.
\end{lemma}


\par\nobreak\section{The isomorphism $C^*(G(c)) \cong C^*(G) \rtimes_{\alpha^c} \mathbb{T}$}
\label{subsection:crossedproducts:isomorphism}

Here we find the isomorphism $\rho \rtimes u: C^*(G) \rtimes_{\alpha^c} \mathbb{T} \rightarrow C^*(G(c))$ by applying Proposition \ref{prop:covariant} to a specific covariant homomorphism $(\rho, u)$ of $(\oldc{C^*(G(c))}\newc{C^*(G)}, \mathbb{T}, \alpha^c)$ into $M(C^*(G(c)))$.

Throughout this section, we assume $G$ to be a locally compact, Hausdorff, second countable, \'etale, measurewise amenable groupoid, and we suppose that $c: G \rightarrow \mathbb{Z}$ is a continuous cocycle. We require $G$ to be measurewise amenable because, when we study the continuity of $\rho$ and $u$, we assume that the reduced and full norms on $C_c(G(c))$ coincide. Note that $G$ and $G(c)$ have the same topological properties by Lemma \ref{lemma:Gctopology}.

Let us begin by defining the dynamical system $\alpha^c: \mathbb{T} \rightarrow \mathrm{Aut}(C^*(G))$ induced by the continuous cocycle $c$. For every $z \in \mathbb{T}$, $f \in C_c(G)$ and $g \in G$, we set
\begin{align*}
\alpha_z^c(f)(g) = z^{c(g)} f(g).
\end{align*}
By Proposition II.5.1 of Renault's book \cite{Renault}, there exists a unique dynamical system defined by the equation above. We make an abuse of notation by denoting this dynamical system $\alpha^c: \mathbb{T} \rightarrow \mathrm{Aut}(C^*(G))$.

Now we define the $\ast$-homomorphism $\rho$ of the covariant representation $(\rho, u)$.

\begin{lemma}
\label{lemma:rhomultiplier}
There exists a unique $\ast$-homomorphism $\rho$ from $C^*(G)$ to the multiplier algebra $M(C^*(G(c)))$ such that, for $h \in C_c(G)$, $\rho(h)$ is defined by
\begin{align}
(\rho(h) \cdot F)(g,a) &= \sum_{g_1 g_2 = g} h(g_1) F(g_2, a + c(g_1)), \label{eqn:rhomultiplierL}\\
(F \cdot \rho(h))(g,a) &= \sum_{g_1 g_2 = g} F(g_1, a)h(g_2), \label{eqn:rhomultiplierR}
\end{align}
for $F \in C_c(G(c))$ and $(g, a) \in G(c)$.
Moreover, $\Vert \rho(h) \Vert = \Vert h \Vert$.
\end{lemma}
\begin{proof}
We fix $h \in C_c(G(c))$ and we let $L, R$ be the linear operators on $C_c(G(c))$ \oldc{defined by $L(F) = \rho(h) \cdot F$ and $R(F) = F \cdot \rho(h)$ as in equations \eqref{eqn:rhomultiplierL} and \eqref{eqn:rhomultiplierR}.}\newc{by letting $L(F)$ and $R(F)$ be the formulae on the right-hand sides of \eqref{eqn:rhomultiplierL} and \eqref{eqn:rhomultiplierR}, respectively.}

This proof is extensive and has several calculations. So we divide this proof into five parts:
\begin{enumerate}[(a)]
\item For $(x,a) \in G^{(0)} \times \mathbb{Z}$, we have $\pi_{(x,a)}(L(F)) = \iota_a \pi_x(h) \iota_a^{-1} \pi_{(x,a)}(F)$, where $\iota_a: \ell^2(G_x) \rightarrow \ell^2(G_{(x,a)})$ is the isomorphism given by $\iota_a \delta_g = \delta_{(g,a - c(g))}$ for $g \in G_x$. \label{item:rhoa}
\item Similarly, $\pi_{(x,a)}(R(F)) = \pi_{(x,a)}(F) \iota_a \pi_x(h) \iota_a^{-1} $. \label{item:rhob}
\item Using \eqref{item:rhoa} and \eqref{item:rhob}, we extend $L$ and $R$ to operators on $C^*(G(c))$ with $\Vert L \Vert = \Vert R \Vert = \Vert h \Vert$. We make an abuse of notation and denote these extensions by $L$ and $R$. \label{item:rhoc}
\item We show that $\rho(h) = (L, R) \in M(C^*(G(c)))$. \label{item:rhod}
\item We prove that $\rho$ is a $\ast$-homomorphism. \label{item:rhoe}
\end{enumerate}
\oldc{Note that we do not need to prove part (a)} \newc{The proof of item (a) is straightforward, since formula for $\pi_{(x,a)}(L(F))$ follows from the definitions of the representations $\pi_{(x,a)}$ and $\pi_x$.} Let us prove the other items.

\begin{enumerate}[(a)]
  \setcounter{enumi}{1}
  \item $\pi_{(x,a)}(L(F)) = \iota_a \pi_x(h) \iota_a^{-1} \pi_{(x,a)}(F)$
  
  Fix $(x, a) \in G^{(0)} \times \mathbb{Z}$ and $g \in G_x$. Let $\iota_a: \ell^2(G_x) \rightarrow \ell^2(G_{(x,a)})$ be the operator defined by $\iota_a \delta_g = \delta_{(g, a - c(g))}$ for $g \in G_x$. Note that $\iota_a$ is an isometry and its inverse is given by $\iota_a^{-1} \delta_{(g, a - c(g))} = \delta_g$ for $g \in G_x$.
  
  We claim that $\pi_{(x,a)}(L(F)) = \iota_a \pi_x(h) \iota_a^{-1} \pi_{(x,a)}(F)$ for $F \in C_c(G(c))$. In fact, let $F \in C_c(G(c))$ and let $g \in G_x$. Then
  \begin{align*}
      &\iota_a \pi_x(h) \iota_a^{-1} \pi_{(x,a)}(F) \delta_{(g, a - c(g))} \\
      =& \iota_a \pi_x(h) \iota_a^{-1} \sum_{l \in G_{r(g)}} F(l,a - c(lg)) \delta_{(lg,a - c(lg))} \\
      =& \sum_{l \in G_{r(g)}} F(l,a - c(lg)) \iota_a \pi_x(h) \iota_a^{-1} \delta_{(lg, a - c(lg))} \\
      =& \sum_{l \in G_{r(g)}} F(l,a - c(lg)) \iota_a \pi_x(h) \delta_{lg} \\
      =& \sum_{l \in G_{r(g)}} F(l,a - c(lg)) \iota_a \sum_{j \in G_{r(lg)}} h(j) \delta_{jlg} \\
      =& \sum_{l \in G_{r(g)}} F(l,a - c(lg)) \sum_{j \in G_{r(lg)}} h(j) \delta_{(jlg, a - c(jlg))}
      \intertext{By the change of variables $k_1 = j$, $k_2 = l$, $k = jl$, we have}
      &\iota_a \pi_x(h) \iota_a^{-1} \pi_{(x,a)}(F) \delta_{(g, a - c(g))} \\
      =& \sum_{k \in G_{r(g)}} \sum_{k_1 k_2 = k} h(k_1) F(k_2, a - c(k_2 g)) \delta_{(kg, a - c(kg))} \\
      =& \sum_{\substack{k \in G_{r(g)} \\ k_1 k_2 = k}} h(k_1) F(k_2, a - c(kg) + c(k_1)) \delta_{(kg, a - c(kg))} \\
      =& \sum_{k \in G_{r(g)}} L(F)(k, a - c(kg)) \delta_{(kg, a - c(kg))} \\
      =& \pi_{(x,a)}(L(F)) \delta_{(g, a - c(g))}.
   \end{align*}
   Therefore, $\iota_a \pi_x(h) \iota_a^{-1} \pi_{(x,a)}(F) = \pi_{(x,a)}(L(F))$.
   
   \item $\pi_{(x,a)}(R(F)) = \pi_{(x,a)}(F) \iota_a \pi_x(h) \iota_a^{-1} $
      \begin{align*}
            &\pi_{(x,a)}(F) \iota_a \pi_x(h) \iota_a^{-1} \delta_{(g, a - c(g))} \\
            =& \pi_{(x,a)}(F) \iota_a \pi_x(h) \delta_g \\
            =& \pi_{(x,a)}(F) \iota_a \oldc{\sum_{j \in G_{r(g)}}}\newc{\sum_{l \in G_{r(g)}}} h(l) \delta_{lg} \\
            =& \sum_{l \in G_{r(g)}} h(l) \pi_{(x,a)}\newc{(F)}\iota_a \delta_{lg} \\
            =& \sum_{l \in G_{r(g)}} h(l) \pi_{(x,a)}\newc{(F)} \oldc{\delta_{(lj, a - c(lg)}} \newc{\delta_{(lg, a - c(lg))}}\\
            =& \sum_{l \in G_{r(g)}} h(l) \sum_{j \in G_{r(lg)}} F(j, a - c(jlg)) \delta_{(jlg, a - c(jlg))}
            \intertext{The change of variables $k_1 = j$, $k_2 = l$, \oldc{$k_1 k_2 = jl$}\newc{$k = jl$} implies that}
            &\pi_{(x,a)}(F) \iota_a \pi_x(h) \iota_a^{-1} \delta_{(g, a - c(g))} \\
            =& \sum_{k \in G_{r(g)}} \sum_{k_1 k_2 = k} F(k_1, a - c(kg)) h(k_2) \delta_{(kg, a - c(kg))} \\
            =& \sum_{k \in G_{r(g)}} R(F)(k, a - c(kg)) \delta_{(kg,a - c(kg))} \\
            =& \pi_{(x,a)}(R(F)) \delta_{(g, a - c(g))}.
      \end{align*}
      Therefore, $\pi_{(x,a)}(F) \iota_a \pi_x(h) \iota_a^{-1} = \pi_{(x,a)}(R(F))$.
      
      \item $\Vert L \Vert = \Vert R \Vert = \Vert h \Vert$

      Since $G$ is measurewise amenable, we have that $G(c)$ is measurewise amenable by Lemma \ref{lemma:Gctopology}. Thus, the full and reduced norms on $C^*(G(c))$ \newc{are identical}. We calculate $\Vert L \Vert$ and $\Vert R \Vert$ by applying the definition of the reduced norm.
      
      Note that, for $F \in C_c(G)$, $\Vert L(F) \Vert$ is the supremum of $\Vert \pi_{(x,a)}(L(F)) \Vert$ for $(x,a) \in G^{(0)} \times \mathbb{Z}$. It follows from item (b) that
\begin{align*}
\sup_{\substack{F \in C_c(G(c)) \\ \Vert F \Vert \leq 1}} \Vert L(F) \Vert
&= \sup_{\substack{F \in C_c(G(c)) \\ \Vert F \Vert \leq 1}} \sup_{(x,a) \in G^{(0)} \times \mathbb{Z}} \Vert \pi_{(x,a)} L(F) \Vert \\
&= \sup_{\substack{F \in C_c(G(c)) \\ \Vert F \Vert \leq 1}} \sup_{(x,a) \in G^{(0)} \times \mathbb{Z}} \Vert \iota_a \pi_x(h) \iota_a^{-1} \pi_{(x,a)}(F) \Vert \\
&= \sup_{(x,a) \in G^{(0)} \times \mathbb{Z}} \sup_{\substack{F \in C_c(G(c)) \\ \Vert F \Vert \leq 1}}  \Vert \iota_a \pi_x(h) \iota_a^{-1} \pi_{(x,a)}(F) \Vert \\
&= \sup_{(x,a) \in G^{(0)} \times \mathbb{Z}} \Vert \iota_a \pi_x(h) \iota_a^{-1} \Vert \\
&= \sup_{(x,a) \in G^{(0)} \times \mathbb{Z}} \Vert \pi_x(h) \Vert \\
&= \Vert h \Vert.
\end{align*}

Then $L$ is bounded and we extend it to a map on $C^*(G(c))$, also denoted $L$, that has norm $\Vert h \Vert$ . Analogously, using item \eqref{item:rhob}, we extend $R$ to an operator on $C^*(G(c))$ with $\Vert R \Vert = \Vert h \Vert$.

\item $\rho(h) = (L, R) \in M(C^*(G(c)))$

We show that $R(F_1) \cdot F_2 = F_1 \cdot L(F_2)$ for all $F_1, F_2 \in C_c(G(c))$ and then we apply Lemma \ref{lemma:conditionmultiplier} to show that $\rho(h) \in M(C^*(G(c)))$.

Let $F_1, F_2 \in C_c(G(c))$ and $(g,a) \in G(c)$. Then
\begin{align*}
\oldc{R(F_1) \cdot F_2(g,a)}\newc{(R(F_1) \cdot F_2)(g,a)}
&= \sum_{g_1 g_2 = g} R(F_1)(g_1, a) F_2(g_2, a + c(g_1)) \\
&= \sum_{g_1 g_2 = g} \sum_{k_1 k_2 = g_1} F_1(k_1, a) h(k_2) F_2( g_2, a+ c(g_1)) \\
\intertext{We make the change of variables $j_1 = k_1$, $l_1 = k_2$, $l_2 = g_2$, $j_2 = l_1 l_2$ and obtain}
\oldc{R(F_1) \cdot F_2(g,a)}\newc{(R(F_1) \cdot F_2)(g,a)}
&= \sum_{j_1 j_2 = g} \sum_{l_1 l_2 = j_2} F_1(j_1, a) h(l_1) F_2(l_2, a + c(j_1 l_1) \\
&= \sum_{j_1 j_2 = g} F_1(j_1, a) \sum_{l_1 l_2 = j_2} h(l_1) F_2(l_2, a + c(j_1) + c(l_1)) \\
&= \sum_{j_1 j_2 = g} F_1(j_1, a) L(F_2)(j_2, a + c(j_1)) \\
&= \oldc{ F_1 \cdot L(F_2)(j_2, a + c(j_1))}
\newc{(F_1 \cdot L(F_2))(j_2, a + c(j_1))} \\
&= \oldc{F_1 \cdot L(F_2)(g,a).}
\newc{(F_1 \cdot L(F_2))(g,a).}
\end{align*}
Thus \oldc{$R(F_1) \cdot F_2 = R(F_1) \cdot F_2$}\newc{$R(F_1) \cdot F_2 = F_1 \cdot L(F_2)$} for all $F_1, F_2 \in C_c(G(c))$. Note that $C_c(G(c))$ is dense in $C^*(G(c))$. Then, by Lemma \ref{lemma:conditionmultiplier}, $\rho(h) = (L, R) \in M(C^*(G(c)))$.

\item $\rho$ is a $\ast$-homomorphism

It follows from equations \eqref{eqn:rhomultiplierL} and \eqref{eqn:rhomultiplierR} that the map $h \in C_c(G) \mapsto \rho(h)$ is linear. Part \ref{item:rhoc} implies that $\Vert \rho(h) \Vert = \Vert h \Vert$ for each $h \in C_c(G)$. Then this map has a unique extension $\rho: C^*(G) \rightarrow M(C^*(G(c)))$ and $\rho$ is an isometry.

If we show that the restriction of $\rho$ to $C_c(G(c))$ is a $\ast$-homomophism, then the continuity of $\rho$ implies that $\rho$ is a $\ast$-homomophism.

Let $h_1, h_2, h \in C^*(G)$. Set $\rho(h_i) = (L_i, R_i)$ for $i=1,2$, and let $\rho(h) = (L, R)$.
\begin{enumerate}[(i)]
\item $\rho(h^*) = \rho(h)^*$

Let $F \in C^*(G(c))$ and $(g,a) \in G(c)$. Then
\begin{align}
[\rho(h)^* \cdot F](g,a)
&= R^\sharp(F)(g,a) \nonumber \\
&= R(F^*)^*(g,a) \nonumber \\
&= \overline{R(F^*)[(g,a)^{-1}]} \nonumber \\
&= \overline{R(F^*)[(g^{-1},a + c(g))]} \nonumber \\
&= \sum_{k_1 k_2 = g^{-1}} \overline{F^*(k_1, a + c(g))h(k_2)}. \nonumber \\
\intertext{We make the change of variables $g_1 = k_2^{-1}, g_2 = k_1^{-1}$ and obtain}
[\rho(h)^* \cdot F](g,a)
&= \sum_{g_1 g_2 = g} \overline{h(g_1^{-1})} \overline{F^*(g_2^{-1}, a + c(g))} \nonumber \\
&= \sum_{g_1 g_2 = g} h^*(g_1) F(g_2, a + c(g) + c(g_2^{-1})) \\
&= \sum_{g_1 g_2 = g} h^*(g_1) F(g_2, a + c(g) - c(g_2)) \nonumber \\
&= \sum_{g_1 g_2 = g} h^*(g_1) F(g_2, a + c(g_1)) \nonumber \\
&= [\rho(h^*) \cdot F](g,a) \label{eqn:Lrhohast}.
\end{align}
By Lemma \ref{lemma:equalmultipliers}, \eqref{eqn:Lrhohast} implies that $\rho(h^*) = \rho(h)^*$.

\item $\rho(h_1 \cdot h_2) = \rho(h_1) \cdot \rho(h_2)$

Let $F \in C^*(G(c))$ and $(g,a) \in G(c)$. Then
\begin{align}
[\rho(h_1 \cdot h_2) \cdot F](g,a)
&= \sum_{k_1 k_2 = g} (h_1 \cdot h_2)(k_1) F(k_2, a + c(k_1)) \nonumber \\
&= \sum_{k_1 k_2 = g} \sum_{j_1 j_2 = k_1} h_1(j_1) h_2(j_2) F(k_2, a + c(k_1)). \nonumber
\intertext{Make the change of variables $g_1 = j_1$, \oldc{$l_1 = j$}\newc{$l_1 = j_2$}, $l_2 = k_2$, $g_2 =j_1 j_2$. Then}
[\rho(h_1 \cdot h_2) \cdot F](g,a)
&= \sum_{g_1 g_2 = g} \sum_{l_1 l_2 = g_2} h_1(g_1) h_2(l_1) F(l_2, a + c(g_1 l_1)) \nonumber \\
&= \sum_{g_1 g_2 = g} h_1(g_1) \sum_{l_1 l_2 = g_2} h_2(l_1) F(l_2, a + c(g_1 l_1)) \nonumber \\
&= \sum_{g_1 g_2 = g} h_1(g_1) \oldc{\rho(h_2)(F)}\newc{(\rho(h_2) \cdot F)}(g_2, a + c(g_1)) \nonumber \\
&= [\rho(h_1) \cdot \rho(h_2) \cdot F](g, a). \label{eqn:Lrhoprod}
\end{align}
By Lemma \ref{lemma:equalmultipliers}, \eqref{eqn:Lrhoprod} implies that $\rho(h_1) \cdot \rho(h_2) = \rho(h_1 \cdot h_2)$. Therefore, $\rho$ is a $\ast$-homomorphism. \qedhere
\end{enumerate}
\end{enumerate}
\end{proof}

The following lemma defines the group homomorphism $u: \mathbb{T} \rightarrow M(C^*(G(c)))$ of the covariant representation $(\rho, u)$.
\begin{lemma}
\label{lemma:uzmultiplier}
There exists a strongly continuous unitary-valued homomorphism $u: \mathbb{T} \rightarrow M(C^*(G(c)))$ with $z \mapsto u_z$ and such that
\oldc{\begin{align*}
u_z \cdot F(g,a) = z^{-a} F(g,a)
\hspace{10pt}\text{and}\hspace{10pt}
F \cdot u_z (g,a) = z^{-a - c(g)} F(g,a),
\end{align*}}
\newc{
\begin{align*}
(u_z \cdot F)(g,a) = z^{-a} F(g,a)
\hspace{10pt}\text{and}\hspace{10pt}
(F \cdot u_z) (g,a) = z^{-a - c(g)} F(g,a),
\end{align*}
}
for $(g,a) \in G(c)$ and $F \in C_c(G(c))$.
\end{lemma}
\begin{proof}
We divide the proof into a few parts:
\begin{enumerate}[(i)]
\item $u$ is well-defined,
\item $u_z u_w = u_{zw}$ for $z, w \in \mathbb{T}$,
\item $u_z$ is unitary with $u_z^* = u_{z^{-1}}$ for $z \in \mathbb{T}$, and
\item $u$ is strongly continuous.
\end{enumerate}
We show each of these properties.

\begin{enumerate}[(i)]
\item $u$ is well-defined

Fix $z \in \mathbb{T}$. Let $P, Q: C_c(G(c)) \rightarrow C_c(G(c))$ be defined by
\begin{align}
P(F)(g,a) = z^{-a} F(g,a),
\text{\hspace{10pt}}
Q(F)(g,a) = z^{-a-c(g)} F(g,a).
\label{eqn:uzmultiplier}
\end{align}

It is straightforward that these functions are well-defined and linear.

We want to extend $P$ and $Q$ to operators on $C^*(G(c))$, also denoted by $P$ and $Q$, such that $u_z = (P, Q)$ is in $M(C^*(G(c)))$. Before we extend these maps, we show that they are bounded.

We show that $P$ is bounded by studying the definition of the reduced norm. We prove that, for all $(x,a) \in G^{(0)} \times \mathbb{Z}$ and $F \in C_c(G(c))$, we have $\pi_{(x,a)}(P(F)) = S_{(x,a)} \pi_{(x,a)}(F)$, where $S_{(x,a)}$ is a unitary. Then it follows from the definition of the reduced norm that $\Vert P(F) \Vert = \Vert F \Vert$.

Given $(x,a) \in G^{(0)} \times \mathbb{Z}$, let $S_{(x,a)}: \ell^2(G_{(x,a)}) \rightarrow \ell^2(G_{(x,a)})$ be defined by
\begin{align*}
S_{(x,a)} \delta_{(g, a - c(g))} = z^{-a+c(g)} \delta_{(g, a - c(g))}
\text{\hspace{15pt} for $g \in G_x$.}
\end{align*}
Then $S_{(x,a)}$ is unitary. Given $F \in C_c(G(c))$, we have
\begin{align*}
S_{(x,a)} \pi_{(x,a)}(F) \delta_{(g,a - c(g))}
&= S_{(x,a)} \sum_{k \in G_{r(g)}} F(k, a - c(kg)) \delta_{(kg, a - c(kg))} \\
&= \sum_{k \in G_{r(g)}} F(k, a - c(kg)) S_{(x,a)} \delta_{(kg, a - c(kg))} \\
&= \sum_{k \in G_{r(g)}} z^{-a + c(kg)} F(k,a - c(kg)) \delta_{(kg, a - c(kg))} \\
&= \sum_{k \in G_{r(g)}} P(F)(k, a - c(kg)) \delta_{(kg, a - c(kg))} \\
&= \pi_{(x,a)}(P(F)) \delta_{(g, a - c(g))}.
\end{align*}
Therefore, using the equation above and the fact that $S_{(x,a)}$ is an isometry, we have
\begin{align*}
\Vert P(F) \Vert
&= \sup_{(x,a) \in G^{(0)} \times \mathbb{Z}} \Vert \pi_{(x,a)} (P(F)) \Vert\\
&= \sup_{(x,a) \in G^{(0)} \times \mathbb{Z}} \Vert S_{(x,a)} \pi_{(x,a)}(F) \Vert\\
&= \sup_{(x,a) \in G^{(0)} \times \mathbb{Z}} \Vert \pi_{(x,a)}(F) \Vert\\
&= \Vert F \Vert.
\end{align*}

Then $P$ is bounded.

Next we show that $Q$ is bounded by applying the same ideas we used to study $P$. Given $(x,a) \in G^{(0)} \times \mathbb{Z}$, let $T_{(x,a)}: \ell^2(G_{(x,a)}) \rightarrow \ell^2(G_{(x,a)})$ be defined by
\begin{align*}
T_{(x,a)} \delta_{(g, a - c(g))} = z^{-a +c(g)} \delta_{(g, a - c(g))}
\text{\hspace{15pt} for $g \in G_x$.}
\end{align*}

The map $T_{(x,a)}$ is unitary. Then, for $g \in G_x$, we have
\begin{align}
T_{(x,a)} \pi_{(x,a)}(F) \delta_{(g, a - c(g))}
&= T_{(x,a)} \sum_{k \in G_{r(g)}} F(k,a - c(kg)) \delta_{(kg, a - c(kg))} \nonumber\\
&= \sum_{k \in G_{r(g)}} F(k, a - c(kg)) T_{(x,a)} \delta_{(kg,a - c(kg))} \nonumber\\
&= \sum_{k \in G_{r(g)}} F(k, a - c(kg))z^{-a + c(kg)} \delta_{(kg,a-c(kg))}\nonumber\\
&= \sum_{k \in G_{r(g)}} z^{-a + c(kg)} F(k,a - c(kg)) \delta_{(kg,a - c(kg))}\nonumber\\
&= \sum_{k \in G_{r(g)}} Q(F)(k, a - c(kg)) \delta_{(kg,a - c(kg))} \nonumber\\
&= \pi_{(x,a)}(Q(F)) \delta_{(g, a - c(g))}.\label{eqn:pixaQFz}
\end{align}


Then $\pi_{(x,a)}(Q(F))  = T_{(x,a)} \pi_{(x,a)}(F)$ and, therefore,
\begin{align*}
\Vert Q(F) \Vert
&= \sup_{(x,a) \in G^{(0)} \times \mathbb{Z}} \Vert \pi_{(x,a)}(Q(F)) \Vert\\
&= \sup_{(x,a) \in G^{(0)} \times \mathbb{Z}} \Vert T_{(x,a)} \pi_{(x,a)}(F) \Vert\\
&= \sup_{(x,a) \in G^{(0)} \times \mathbb{Z}} \Vert \pi_{(x,a)}(F) \Vert\\
&= \Vert F \Vert.
\end{align*}
So $Q$ is also bounded. We extend $P$ and $Q$ to maps on $C^*(G(c))$, also denoted $P$ and $Q$.

Finally, we show that $u_z = (P, Q) \in M(C^*(G(c)))$. Using Lemma \ref{lemma:conditionmultiplier} and the continuity of $P$ and $Q$, we only need to show that $Q(F_1) \cdot F_2 = F_1 \cdot P(F_2)$ for $F_1, F_2 \in C_c(G(c))$. In fact, given $(g,a) \in G(c)$ and $F_1, F_2 \in C_c(G(c))$, we have
\begin{align*}
(Q(F_1) \cdot F_2)(g,a)
&= \sum_{k_1 k_2 = g} Q(F_1)(k_1, a)F_2(k_2, a + c(k_1))\\
&= \sum_{k_1 k_2 = g} z^{-a - c(k_1)}F_1(k_1, a)F_2(k_2, a + c(k_1))\\
&= \sum_{k_1 k_2 = g} F_1(k_1, a)P(F_2)(k_2, a + c(k_1))\\
&= (F_1 \cdot P(F_2))(g,a).
\end{align*}
Thus $Q(F_1) \cdot F_2 = F_1 \cdot P(F_2)$ and therefore $(P, Q) \in M(C^*(G(c)))$.

\item $u_z u_w = u_{zw}$

Let $z, w \in \mathbb{T}$. Since $C_c(G(c))$ is dense in $C^*(G(c))$, it follows from Proposition \ref{prop:multiplieralgebra} and Lemma \ref{lemma:equalmultipliers} that $u_{zw} = u_z \cdot u_w$ if, and only if,
\begin{align*}
u_z \cdot u_w \cdot F = u_{zw} \cdot F
\hspace{10pt}
\text{for } F \in C_c(G(c)).
\end{align*}
Indeed, given $F \in C_c(G(c))$ and $(g,a) \in G(c)$, we have
\newc{
\begin{align*}
(u_z \cdot (u_w \cdot F))(g,a)
= z^{-a} w^{-a} F(g,a)
= (u_{zw} \cdot F)(g,a).
\end{align*}
}
Therefore, $u_z \cdot u_w = u_{zw}$.

\item $u_z^* = u_{z^{-1}}$

Let $z \in \mathbb{T}$. Write $u_z = (P, Q)$ with $P, Q$ being bounded operators on $C^*(G(c))$. Hence, $u_z^* = (Q^\sharp, P^\sharp)$. Moreover, for $F \in C_c(G(c))$, we have
\begin{align*}
(u_z^* \cdot F)(g,a)
&= Q^\sharp(F)(g,a)\\
&= Q(F^*)^*(g,a)\\
&= \overline{Q(F^*)(g^{-1}, a + c(g))}\\
&= \overline{z^{-(a + c(g)) - c(g^{-1})} F^*(g^{-1}, a + c(g))}\\
&= \overline{z^{-a}\overline{F(g,a)}}\\
&= z^a F(g,a)\\
&= (u_{z^{-1}} \cdot F)(g,a).
\end{align*}
Therefore $u_{z}^* = u_{z^{-1}}$. So, by item \oldc{(i)}\newc{(ii)}, we have $u_{z} u_{z}^* = u_{z}^* u_{z} = u_1$. Then $u_z$ is unitary.

\item $u$ is strongly continuous

First we show that $\Vert F \cdot u_w - F \cdot u_z \Vert \rightarrow 0$ as $w \rightarrow z$ for $z \in \mathbb{T}$ and $F \in C_c(G(c))$. Then we use that $C_c(G(c))$ is dense in $C^*(G(c))$ to show that $z \mapsto u_z$ is strongly continuous.

We begin by proving the convergence for a particular class of functions. Let $b \in \mathbb{Z}$ and $f \in C_c(G(c))$ be such that $\mathrm{supp\hphantom{.}} f \subset G \times \lbrace b \rbrace$. Let $w, z \in \mathbb{T}$ and $(x, a) \in G^{(0)} \times \mathbb{Z}$. Then, for $g \in G_x$,
\begin{align*}
\pi_{(x,a)}(u_w \cdot f) \delta_{(g, a - c(g))}
&= \sum_{k \in G_{r(g)}} (u_w \cdot f)(k, a - c(kg)) \delta_{(kg, a - c(kg))} \\
&= \sum_{k \in G_{r(g)}} w^{-a + c(kg)} f(k, a - c(kg)) \delta_{(kg, a - c(kg))} \\
&= \sum_{k \in G_{r(g)}} w^{-b- c(kg) + c(kg)} f(k, a - c(kg)) \delta_{(kg, a - c(kg))}
\intertext{because $f(k, a - c(kg)) \neq 0$ implies that $a - c(kg) = b$ and, then, $a = b + c(kg)$. So,}
\pi_{(x,a)}(u_w \cdot f) \delta_{(g, a - c(g))}
&= w^{-b} \sum_{k \in G_{r(g)}} f(k, a - c(kg)) \delta_{(kg, a - c(kg))} \\
&= w^{-b} \pi_{(x,a)}(F).
\end{align*}
Thus, $\pi_{(x,a)}(u_w \cdot f) = w^{-b} \pi_{(x,a)}(f)$. Similarly, $\pi_{(x,a)}(u_z \cdot f) = z^{-b} \pi_{(x,a)}(f)$. Then
\begin{align*}
\Vert u_w \cdot f - u_z \cdot f \Vert
&= \sup_{(x,a) \in G^{(0)} \times \mathbb{Z}} \Vert \pi_{(x,a)}(u_w \cdot f) - \pi_{(x,a)}(u_z \cdot f) \Vert \\
&= \sup_{(x,a) \in G^{(0)} \times \mathbb{Z}} \Vert w^{-b} \pi_{(x,a)}(f) - z^{-b} \pi_{(x,a)}(f) \Vert \\
&= \vert w^{-b} - z^{-b} \vert \sup_{(x,a) \in G^{(0)} \times \mathbb{Z}} \Vert \pi_{(x,a)}(f) \Vert \\
&= \vert w^{-b} - z^{-b} \vert \Vert f \Vert.
\end{align*}
This implies that $\Vert u_w \cdot f - u_z \cdot f \Vert \rightarrow 0$ as $w \rightarrow z$.

Now let $F \in C_c(G(c))$ and $z \in \mathbb{T}$. Write $F = F_1 + \hdots + F_n$ such that, for each $i = 1, \hdots, n$, $\mathrm{supp}\hphantom{.} F \subset G \times \lbrace b_i \rbrace$ for some $b_i$. Then $\Vert u_w \cdot F_i - u_z \cdot F_i \Vert \rightarrow 0$ as $w \rightarrow z$.

Note that, for each $w \in \mathbb{T}$, we have
\begin{align*}
\Vert u_w \cdot F - u_z \cdot F \Vert
\leq \sum_{i=1}^n \Vert u_w \cdot F_i - u_z \cdot F_i \Vert.
\end{align*}
Then $\Vert u_w \cdot F - u_z \cdot F \Vert \rightarrow 0$ as $w \rightarrow z$.

Finally, we show that $u$ is strongly continuous. Fix $\widetilde{F} \in C^*(G(c))$ and $z \in \mathbb{T}$, and let $\varepsilon > 0$. Then there exists $F \in C_c(G(c))$ with $\Vert \widetilde{F} - F \Vert < \frac{\varepsilon}{4}$. 

As shown previously, there exists $\delta > 0$ such that for $w \in \mathbb{T}$ with $\vert w - z \vert < \delta$, we have $\Vert F \cdot u_w - F \cdot u_z \Vert < \varepsilon/2$. In this case,
\begin{align*}
\Vert \widetilde{F} \cdot u_w - \widetilde{F} \cdot u_z \Vert
&\leq \Vert \widetilde{F} \cdot u_w - F \cdot u_w \Vert +
      \Vert F \cdot u_w - F \cdot u_z \Vert \\
      &\hphantom{\leq.} + \Vert F \cdot u_z - \widetilde{F} \cdot u_z \Vert\\
&=    \Vert (\widetilde{F} - F)\cdot u_w \Vert +
      \Vert F \cdot u_w - F \cdot u_z \Vert \\
      &\hphantom{\leq.} + \Vert (F - \widetilde{F}) \cdot u_z \Vert\\
&\leq \Vert \widetilde{F} - F \Vert \Vert u_w \Vert +
      \Vert F \cdot u_w - F \cdot u_z \Vert +
      \Vert F - \widetilde{F} \Vert \Vert u_z \Vert\\
&= \Vert \widetilde{F} - F \Vert +
      \Vert F \cdot u_w - F \cdot u_z \Vert +
      \Vert F - \widetilde{F} \Vert\\
&= 2 \Vert \widetilde{F} - F \Vert +
      \Vert F \cdot u_w - F \cdot u_z \Vert\\
&< 2 \frac{\varepsilon}{4} + \frac{\varepsilon}{2}\\
&= \varepsilon.
\end{align*}
Then $\widetilde{F} \cdot u_w \rightarrow \widetilde{F} \cdot u_z$ as $w \rightarrow z$. Therefore $u$ is strongly continuous. \qedhere
\end{enumerate}
\end{proof}

In Proposition \ref{prop:crossedproduct-isomorphism} we show that $\rho \rtimes u$ is the isomorphism from $C^*(G) \rtimes_{\alpha^c} \mathbb{T}$ onto $C^*(G(c))$. We apply Proposition \ref{prop:covariant} and get $\rho \rtimes u$ as a  map from $C^*(G) \rtimes_{\alpha^c} \mathbb{T}$ to $M(C^*(G(c)))$. We use the lemma below to prove that the image of $\rho \rtimes u$ is precisely $C^*(G(c))$.

\begin{lemma}
\label{lemma:Zdense}
The functions of the form
\begin{align*}
z \mapsto z^n h,
\hspace{20pt}
\text{for $z \in \mathbb{T}$, $n \in \mathbb{Z}$,}
\end{align*}
and $h \in C_c(G)$ with $c(\mathrm{supp}\hphantom{.} h)$ singleton span a dense subset of $C^*(G) \rtimes_{\alpha^c} \mathbb{T}$.
\end{lemma}
\begin{proof}
Let $\mathcal{Z}$ be the set spanned by functions $z \mapsto z^n h$ as above. We begin by proving that $\mathcal{Z}$ approximates each element of the form $b 1_A \in C^*(G) \rtimes_{\alpha^c} \mathbb{T}$ in the norm $\Vert \cdot \Vert_1$, such that $A \subset \mathbb{T}$ is Borel \oldc{with finite measure} and $b \in C^*(G)$. Then we show that elements of the form $b 1_A$  span a dense subset of $C^*(G) \rtimes_{\alpha^c} \mathbb{T}$. Note that, for every $f \in C(\mathbb{T}, C^*(G))$, we have $\Vert f \Vert_1 \leq 2 \pi \Vert f \Vert_\infty$.

Fix $b \in C^*(G)$ and \oldc{$A \subset \mathbb{T}$ with finite measure}\newc{let $A \subset \mathbb{T}$ be a Borel subset}. Since the case $b 1_A = 0$ is trivial, we suppose that $b 1_A \neq 0$. So $b \neq 0$ and $A$ has a positive measure.

Let $0 < \varepsilon < 1$. It follows from the Stone-Weierstrass theorem (see \cite[Corollary 7.5.4]{Dixmier}) that there exists a function $p \in C(\mathbb{T})$ given by
\begin{align*}
p(z) = \sum_{i = 1}^N \alpha_i z^{n_i}
\end{align*}
with $n_i \in \mathbb{Z}$ and $\alpha_i \in \mathbb{C}$, such that $\Vert p - 1_A \Vert_\infty \leq \min \lbrace \frac{\varepsilon}{4 \pi \Vert b \Vert}, \varepsilon \rbrace$. The triangle inequality implies that $p \neq 0$. In fact,
\begin{align*}
\Vert p \Vert_\infty
\geq
\Vert 1_A \Vert_\infty - \Vert p - 1_A \Vert_\infty
= 1 - \Vert p - 1_A \Vert_\infty
> 1 - \varepsilon
> 0.
\end{align*}

Now we choose an $h \in C_c(G)$ such that $ph$ approximates $b 1_A$. Here $ph: \mathbb{T} \rightarrow C_c(G)$ is the function defined by $(ph)(z) = p(z) h$. Note that $ph \in \mathcal{Z}$. We show that $ph$ approximates $b 1_A$ in the norm $\Vert \cdot \Vert_1$.

Let $h \in C_c(G)$ be such that $\Vert b -h \Vert < \frac{\varepsilon}{4 \pi \Vert p \Vert_\infty}$. Then
\begin{align*}
\Vert b 1_A - p h \Vert_1
&\leq \oldc{\Vert b 1_A - p h \Vert_1}\newc{\Vert b 1_A - p b \Vert_1} + \Vert p b - p h \Vert_1 \\
&= \Vert b \Vert \Vert 1_A - p \Vert_1 + \Vert p \Vert_1 \Vert b -h \Vert \\
&\leq 2 \pi \Vert b \Vert \Vert 1_A - p \Vert_\infty + 2\pi \vert p \Vert_\infty \Vert b - h \Vert \\
&< 2 \pi \Vert b \Vert \frac{\varepsilon}{4 \pi \Vert b \Vert} + 2\pi \Vert p \Vert_\infty \frac{\varepsilon}{4 \pi \Vert p \Vert_\infty}. \\
&= \frac{\varepsilon}{2} + \frac{\varepsilon}{2} \\
&= \varepsilon.
\end{align*}
We have that the functions in $\mathcal{Z}$ approximates all elements of the form \oldc{$1_A$}\newc{$b 1_A$} in the $\Vert \cdot \Vert_1$-norm.

Let $\mathcal{W}$ be the set consisting of all linear combinations  of the form $1_A b$ as before. The elements of $\mathcal{W}$ are called \newterm{simple functions} in Williams's book \cite[Appendix B]{Williams-crossedproducts}. It follows from Proposition B.32 of the book that $\mathcal{W}$ approximates $C(\mathbb{T}, C^*(G))$ in the $\Vert \cdot \Vert_1$-norm. This implies that $\mathcal{Z}$ approximates $C(\mathbb{T}, C^*(G))$ in the $\Vert \cdot \Vert_1$-norm.

Now we show that $\mathcal{Z}$ is dense in $C^*(G) \rtimes_{\alpha^c} \mathbb{T}$. Let $a \in C^*(G) \rtimes_{\alpha^c} \mathbb{T}$ and $\varepsilon > 0$. There exists $v \in C(\mathbb{T}, C^*(G))$ such that $\Vert a - v \Vert < \frac{\varepsilon}{2}$. There exists $f \in \mathcal{Z}$ with $\Vert v - f \Vert_1 < \varepsilon/2$. Recall from Lemma \ref{lemma:crossedproducts} that $\Vert \cdot \Vert$ is dominated by $\Vert \cdot \Vert_1$. Then
\begin{align*}
\Vert a - f \Vert
\leq \Vert a - v \Vert + \Vert v - f \Vert
\leq \Vert a - v \Vert + \Vert v - f \Vert_1
\leq \frac{\varepsilon}{2} + \frac{\varepsilon}{2}
= \varepsilon.
\end{align*}
Therefore $\mathcal{Z}$ is dense in $C^*(G) \rtimes_{\alpha^c} \mathbb{T}$.
\end{proof}

The proof of the following proposition is the alternative to Renault's \cite[Proposition II.5.7]{Renault}. We assume $\rho: C^*(G) \rightarrow M(C^*(G(c)))$ is that $\ast$-homomorphism of Lemma \ref{lemma:rhomultiplier}, and that $u: \mathbb{T} \rightarrow M(C^*(G(c)))$ is the unitary representation of Lemma \ref{lemma:uzmultiplier}.

\begin{proposition}
\label{prop:crossedproduct-isomorphism}
There exists a unique isomorphism $\rho \rtimes u: C^*(G) \rtimes_{\alpha^c} \mathbb{T} \rightarrow C^*(G(c))$ such that
\begin{align*}
\rho \rtimes u(f)(g,a)
&= \int_\mathbb{T} z^{-c(g) - a} f(z)(g) dz
\end{align*}
for \oldc{$f \in C_c(\mathbb{Z}, C_c(G))$}\newc{$f \in C_c(\mathbb{T}, C_c(G))$} and $(g,a) \in G(c)$. Moreover, the inverse of $\rho \rtimes u$ is given by
\begin{align*}
(\rho \rtimes u)^{-1}(F)(z)(g)
&= \frac{1}{2\pi} \sum_{a \in \mathbb{Z}} z^{c(g) + a} F(g,a),
\end{align*}
for $F \in C_c(G(c))$, $z \in \mathbb{T}$, $g \in G$.
\end{proposition}
\begin{proof}
First we prove that $(\rho, u)$ is a covariant homomorphism. Because $C_c(G(c))$ is dense in $C^*(G(c))$, we have that $(\rho, u)$ is a covariant homomorphism if the equation below holds.
\begin{align}
\label{eqn:rhoucovariant}
\rho(\alpha_z(h))
= u_z \rho(h) u_z^*
\hspace{23pt}
\text{for }
z \in \mathbb{T}, h \in C_c(G).
\end{align}
In fact, fix $z \in \mathbb{T}$ and $h \in C_c(G)$. Let $F \in C_c(G(c))$ and $(g,a) \in G(c)$. Then
\begin{align*}
[u_z \cdot \rho(h) \cdot u_z^* \cdot F](g,a)
&= z^{-a} [\rho(h) \cdot u_{z^{-1}} \cdot F](g,a) \\
&= z^{-a} \sum_{g_1 g_2 = g} h(g_1) [u_{z^{-1}} \cdot F](g_2, a + c(g_1)) \\
&= z^{-a} \sum_{g_1 g_2 = g} h(g_1) z^{a + c(g_1)} F(g_2, a + c(g_1)) \\
&= \sum_{g_1 g_2 = g} z^{c(g_1)} h(g_1) F(g_2, a + c(g_1)) \\
&= \sum_{g_1 g_2 = g} \alpha_z^c(h)(g_1) F(g_2, a + c(g_1)) \\
&= [\rho(\alpha_z^c(h)) \cdot F](g,a).
\end{align*}
Since $C_c(G(c))$ is dense in $C^*(G(c))$, equation \eqref{eqn:rhoucovariant} follows from Lemma \ref{lemma:equalmultipliers}. Then $(\rho, u)$ is a covariant homomorphism.

Proposition \ref{prop:covariant} implies that $\rho \rtimes u: C^*(G) \rtimes_{\alpha^c} \mathbb{T} \rightarrow M(C^*(G(c)))$ is a $\ast$-homomorphism. We claim that this map is an isomorphism onto its image with $C^*(G(c))$.

We show that the image of $\rho \rtimes u$ is a subset of $C^*(G(c))$ by applying Lemma \ref{lemma:Zdense}. Note that here we identify $C^*(G(c))$ as a subset of $M(C^*(G(c)))$ by applying the inclusion map $\iota_{C^*(G(c))}: C^*(G(c)) \rightarrow M(C^*(G(c)))$ of Proposition \ref{prop:multiplieralgebra}.

Fix $n, p \in \mathbb{Z}$, and let $h \in C_c(G)$ be such that $c( \mathrm{supp }\hphantom{.} h) = \lbrace p \rbrace$. Define $F \in C_c(G(c))$ by
\begin{align*}
F(g,a) = 1_{\lbrace n - p \rbrace}(a) h(g)
\hspace{20pt}
\text{ for $(g,a) \in G(c)$.}
\end{align*}
Let $f \in C(\mathbb{T}, C_c(G))$ be defined by $f(z) = z^n h$.
We show that $\rho \rtimes u(f) = 2\pi F$. Indeed, given $(g,a) \in G(c)$ \newc{and $H \in C_c(G(c))$,}
\begin{align*}
\newc{\rho \rtimes u(f) \cdot H (g,a)}
&= \newc{\left[ \int_\mathbb{T} \rho(f(z)) u_z dz \cdot H\right] (g,a)} \nonumber \\
&= \newc{\int_\mathbb{T} [\rho(f(z)) \cdot (u_z \cdot H)](g,a) dz} \nonumber \\
&= \newc{ \int_\mathbb{T} \sum_{g_1 g_2 = g} f(z)(g_1) u_z \cdot H(g_2, a + c(g_1)) dz} \nonumber \\
&= \newc{\int_\mathbb{T} \sum_{g_1 g_2 = g} z^n h(g_1) z^{-a-c(g_1)} H(g_2, a+c(g_1)) dz} \nonumber \\
&= \newc{\sum_{g_1 g_2 = g} \int_\mathbb{T} z^{n-a-c(g_1)} h(g_1) dz H(g_2, a+c(g_1))} \nonumber \\
&= \newc{\sum_{g_1 g_2 = g} \int_\mathbb{T} z^{n-a-p}dz h(g_1) H(g_2, a + c(g_1))} \nonumber \\
&= \newc{\sum_{g_1 g_2 = g} 2\pi 1_{\lbrace n-p \rbrace}(a) h(g_1)H(g_2, a+c(g_1))} \nonumber \\
&= \newc{\sum_{g_1 g_2 = g} 2\pi F(g_1, a) H(g_2, a + c(g_1))} \\
&= \newc{(2\pi F \cdot H)(g,a)} 
\end{align*}
\newc{Analogously, we have that $H \cdot \rho \rtimes u(f) (g,a) = (H \cdot 2\pi F)(g,a)$. Therefore $\rho \rtimes u(f) \in C_c(G(c))$ and}
\begin{align}
\label{eqn:2pi1nph}
\newc{\rho \rtimes u(f)(g,a) = 2\pi F(g,a) = 2\pi 1_{\lbrace n - p \rbrace}(a) h(g).}
\end{align}

By linearity, we have that $\rho \rtimes u(\mathcal{Z}) \subset C_c(G(c))$. Since $\mathcal{Z}$ is dense in $C^*(G) \rtimes_{\alpha^c} \mathbb{T}$ and $\rho \rtimes u$ is bounded, then the image of $\rho \rtimes u$ is a subset of $C^*(G(c))$.

We show that $\rho \rtimes u: C^*(G) \rtimes_{\alpha^c} \mathbb{T} \rightarrow C^*(G(c))$ is an isomorphism by defining the homomorphism $\varphi: C^*(G(c)) \rightarrow C^*(G) \rtimes_{\alpha^c} \mathbb{T}$, which we prove to be the inverse of $\rho \rtimes u$.

Define $\varphi: C_c(G(c)) \rightarrow C_c(\mathbb{T}, C_c(G(c)))$ by
\begin{align*}
\varphi(F)(z)(g) = \frac{1}{2\pi} \sum_{a \in \mathbb{Z}} z^{c(g) + a} F(g,a),
\end{align*}
for $F \in C_c(G(c))$, $z \in \mathbb{Z}$, $g \in G$.

Note that $\varphi$ is well-defined. For a fixed $F$, the sum above finitely many terms because $F$ is compactly supported. Also, $\varphi(F)$ is continuous because it is the composition of continuous functions.

The map $\varphi$ is a $\ast$-homomorphism. In fact, it is straightforward that it is linear. Given $F, F_1, F_2 \in C_c(G(c))$, $z \in \mathbb{T}$ and $g \in G$, we have
\begin{align*}
\varphi(F^*)(z)(g)
&= \frac{1}{2\pi} \sum_{a \in \mathbb{Z}} z^{c(g)+a} F^*(g,a) \\
&= \frac{1}{2\pi} \sum_{a \in \mathbb{Z}} z^{c(g) + a} \overline{F(g^{-1}, a + c(g))}.
\intertext{Make the change of variables $a' = a + c(g)$. Then}
\varphi(F^*)(z)(g)
&= \frac{1}{2\pi} \sum_{a' \in \mathbb{Z}} z^a \overline{F(g^{-1},a')} \\
&= z^{c(g)} \frac{1}{2\pi} \sum_{a' \in \mathbb{Z}} z^{a' + c(g^{-1})} \overline{F(g^{-1}, a)} \\
&=z^{c(g)} \frac{1}{2\pi} \sum_{a' \in \mathbb{Z}} \overline{(z^{-1})^{a' + c(g^{-1})} F(g^{-1}, a')} \\
&= z^{c(g)} \overline{\varphi(F)(z^{-1})(g^{-1})} \\
&= z^{c(g)} \varphi(F)(z^{-1})^*(g) \\
&= \alpha_z^c(\varphi(F)(z^{-1})^*)(g) \\
&= \varphi(F)^*(z)(g),
\end{align*}
and
\begin{align*}
&[\varphi(F_1) \cdot \varphi(F_2)](z)(g) \\
=& \left[ \int_\mathbb{T} \varphi(F_1)(w) \cdot \alpha_w(\varphi(F_2)(w^{-1}z))dw \right](g) \\
=& \int_\mathbb{T} \sum_{g_1 g_2 = g} \varphi(F_1)(w)(g_1) \alpha_w(\varphi(F_2))(w^{-1}z))(g_2) dw \\
=& \int_\mathbb{T} \sum_{g_1 g_2 = g} \varphi(F_1)(w)(g_1)w^{c(g_2)} \varphi(F_2)(w^{-1}z)(g_2) dw \\
=& \int_\mathbb{T} \sum_{g_1 g_2 = g} \left[ \frac{1}{2\pi} \sum_{a \in \mathbb{Z}} w^{c(g_1) + a} F_1(g_1, a) \right] w^{c(g_2)} \\
&\hphantom{........uM}\left[\frac{1}{2\pi} \sum_{b \in \mathbb{Z}} (w^{-1}z)^{c(g_2) + b} F_2(g_2,b) \right] dw\\
=& \frac{1}{4\pi^2} \sum_{a,b \in \mathbb{Z}} \int_\mathbb{T} \sum_{g_1 g_2 = g} w^{c(g_1) +a + c(g_2) - c(g_2) - b} z^{c(g_2) + b} F_1(g_1, a) F_2(g_2, b) dw\\
=& \frac{1}{4\pi^2} \sum_{a,b \in \mathbb{Z}} \sum_{g_1 g_2 = g} z^{c(g_2) + b} F_1(g_1, a) F_2(g_2, b) \int_\mathbb{T} w^{c(g_1) + a - b} dw. \\
\intertext{For $b \neq a + c(g_1)$, the integral above is zero. When $b = a + c(g_1)$ the integral is $2\pi$. Then}
&[\varphi(F_1) \cdot \varphi(F_2)](z)(g)\\
=& \frac{1}{2\pi} \sum_{a \in \mathbb{Z}} \sum_{g_1 g_2 = g} z^{c(g_2) + c(g_1) + a} F_1(g_1, a) F_2(g_2, a + c(g_1)) \\
=& \frac{1}{2\pi} \sum_{a \in \mathbb{Z}} z^{c(g) + a} \sum_{g_1 g_2 = g} F_1(g_1, a) F_2(g_2, a + c(g_1)) \\
=& \frac{1}{2\pi} \sum_{a \in \mathbb{Z}} z^{c(g) + a} (F_1 \cdot F_2)(g,a) \\
=& \varphi(F_1 \cdot F_2)(z)(g).
\end{align*}
Then $\varphi$ is a $\ast$-homomorphism and its image is a subset of the C*-algebra $C^*(G) \rtimes_{\alpha^c} \mathbb{T}$. We make an abuse of notation and extend $\varphi$ to a $\ast$-homomorphism $\varphi: C^*(G(c)) \rightarrow C^*(G) \rtimes_{\alpha^c} \mathbb{T}$.

Finally, we show that $\varphi$ is the inverse of $\rho \rtimes u$. Fix $p, n \in \mathbb{Z}$, let $h \in C_c(G)$ be such that $c(\mathrm{supp}\hphantom{.} h) = \lbrace p \rbrace$, and define $f \in C(\mathbb{T}, C_c(G))$ by $f(z) = z^n h$. Then, for $z \in \mathbb{T}$ and $g \in G$, we have
\begin{align*}
\varphi \circ \rho \rtimes u(f)(z)(g)
&= \frac{1}{2\pi} \sum_{a \in \mathbb{Z}} z^{c(g) + a} \rho \rtimes u(f)(g,a) \\
&= \frac{1}{2\pi} \sum_{a \in \mathbb{Z}} z^{c(g) + a} 2\pi 1_{\lbrace n - p \rbrace}(a) h(g)
\hspace{15pt}\text{by \eqref{eqn:2pi1nph}} \\
&= z^{c(g)+n-p} h(g) \\
&= z^{p+n-p}h(g) \hspace{15pt}\text{because $c(\mathrm{supp } h) = \lbrace p \rbrace$} \\
&= z^n h(g) \\
&= f(g).
\end{align*}
Since $f$ is arbitrary and $\mathcal{Z}$ is dense in $C^*(G) \rtimes_{\alpha^c} \mathbb{T}$, we have $\varphi \circ \rho \rtimes u = \mathrm{id}$.

Now fix $F \in C_c(G(c))$. Then, for $(g,a) \in G(c)$, we have
\begin{align*}
\rho \rtimes u \circ \varphi(F)(g,a)
&= \int_\mathbb{T} z^{-c(g)-a} \varphi(F)(z)(g) dz \\
&= \int_\mathbb{T} z^{-c(g)-a} \frac{1}{2\pi} \sum_{b \in \mathbb{Z}} z^{c(g) + b} F(g,b) dz \\
&= \sum_{a \in \mathbb{Z}} \frac{1}{2\pi} F(g,b) \int_\mathbb{T} z^{b - a} dz  \\
&= F(g,a).
\end{align*}
Then $\rho \rtimes u \circ \varphi = \mathrm{id}$. This implies that $\varphi$ is the inverse of $\rho \rtimes u$ and  therefore $\rho \rtimes u$ is an isomorphism.
\end{proof}

The $\ast$-homomorphism $\varphi:C_c(G(c)) \rightarrow C_c(\mathbb{T}, C_c(G))$ in the previous theorem was extended uniquely to a representation from $C^*(G(c))$ to $C^*(G) \rtimes_{\alpha^c} \mathbb{T}$. This is possible because, given a covariant representation $(\pi, v)$ of $\alpha^c: \mathbb{T} \rightarrow \mathrm{Aut} C^*(G)$, then $\pi \rtimes v \circ \varphi$ is a representation of $C_c(G(c))$. Thus, for $F \in C_c(G(c))$, we have
\begin{align*}
\Vert \pi \rtimes v \circ \varphi(F) \Vert \leq \Vert F \Vert.
\end{align*}
Since $\pi \rtimes v$ is arbitrary, it follows that $\Vert \varphi(F) \Vert \leq \Vert F \Vert$. By continuity, $\varphi$ can be extended continuously to a $\ast$-homomorphism from $C^*(G(c))$ to $C^*(G) \times_{\alpha^c} \mathbb{T}$.

\par\nobreak\section{The action $\beta$}
\label{subsection:crossedproducts:beta}


Now we find a formula for the action $\beta$ on $C^*(G(c))$ induced by the dual action $\widehat{\alpha}^c$. We will apply this formula in Chapter \ref{section:result} for Deaconu-Renault groupoids in order to understand in more detail the condition $H_\beta \cap K_0(B)^+ = \lbrace 0 \rbrace$, given by Brown's theorem.

\begin{proposition}
\label{prop:beta}
Let $\rho \rtimes u: C^*(G) \rtimes_{\alpha^c} \mathbb{T} \rightarrow \oldc{B}\newc{C^*(G(c))}$ be the isomorphism of Proposition \ref{prop:crossedproduct-isomorphism}. Then there exists a unique dynamical system $\beta: \mathbb{Z} \rightarrow \mathrm{Aut}(\oldc{B}\newc{C^*(G(c))})$ such that
\begin{align*}
\beta(F)(g,a) = F(g, a + 1)
\hspace{15pt}
\text{for $F \in C_c(G(c))$, $(g,a) \in G(c)$,}
\end{align*}
and $\beta = (\rho \rtimes u) \circ \widehat{\alpha} \circ (\rho \rtimes u)^{-1}$ for all $n \in \mathbb{Z}$. Moreover, $\oldc{B}\newc{C^*(G(c))} \rtimes_\beta \mathbb{Z}$ and $C^*(G) \rtimes_{\alpha^c} \mathbb{T} \rtimes_{\widehat{\alpha}^c} \mathbb{Z}$ are isomorphic.
\end{proposition}
\begin{proof}
For each $n \in \mathbb{Z}$, let $\beta_n = (\rho \rtimes u)^{-1} \circ \oldc{\widehat{\alpha}_n}\newc{\widehat{\alpha}^c_n} \circ \rho \rtimes u$. Since $\rho \rtimes u$ is an isomorphism and $\widehat{\alpha}^c$ is an automorphism, we have that $\beta: \mathbb{Z} \rightarrow \mathrm{Aut}(C^*(G(c)))$ is a dynamical system.

Fix $n \in \mathbb{Z}$. By continuity of $\beta_n$, finding the formula of this automorphism on $C_c(G(c))$ is sufficient to determine $\beta_n$ uniquely. Let $F \in C_c(G(c))$ and $(g,a) \in G(c)$. Then, Proposition \ref{prop:crossedproduct-isomorphism},
\begin{align*}
\beta_n(F)(g,a)
&= \oldc{(\rho \rtimes u)^{-1} \circ \oldc{\widehat{\alpha}_n}\newc{\widehat{\alpha}^c_n} \circ \rho \rtimes u(F)(g,a)}\newc{(\rho \rtimes u) \circ \widehat{\alpha^c}_n \circ (\rho \rtimes u)^{-1}(F)(g,a)}\\
&= \frac{1}{2\pi}\int_\mathbb{T} z^{-c(g) - a} \oldc{\widehat{\alpha}_n}\newc{\widehat{\alpha^c}_n} \circ \oldc{\rho \rtimes u}\newc{(\rho \rtimes u)^{-1}}(F)(z)(g)dz\\
&= \frac{1}{2\pi}\int_\mathbb{T} z^{-c(g) -a} z^{-n} \oldc{\rho \rtimes u}\newc{(\rho \rtimes u)^{-1}}(F)(z)(g) dz 
\hspace{15pt}\text{by Example \ref{ex:dualactionT},}\\
&= \frac{1}{2\pi}\int_\mathbb{T} z^{-c(g)-a-n} \sum_{b \in \mathbb{Z}} z^{c(g) + b} F(g,b) dz \\
&= \frac{1}{2\pi}\sum_{b \in \mathbb{Z}} F(g,b) \int_\mathbb{T} z^{-a-n+b} dz\\
&= F(g,a+n).
\end{align*}

By definition, we have that $\rho \rtimes u \circ \widehat{\alpha}^c = \beta \circ \rho \rtimes u$. Then $\rho \rtimes u$ is an equivariant isomorphism from $(C^*(G) \rtimes_{\alpha^c} \mathbb{T}, \widehat{\alpha}^c, \mathbb{Z})$ onto $(\oldc{B}\newc{C^*(G(c))}, \beta, \mathbb{Z})$. By Lemma \ref{lemma:crossedproductiso}, $\oldc{B}\newc{C^*(G(c))} \rtimes_\beta \mathbb{Z}$ and $C^*(G) \rtimes_{\alpha^c} \mathbb{T} \rtimes_{\widehat{\alpha}^c} \mathbb{Z}$ are isomorphic.
\end{proof}

\chapter{Homology groups}
\label{section:homology}

In this chapter we study the theory of homology groups for groupoids, as described in \cite{FKPS, Matui}.

As explained in the introduction, in order to characterise when the C*-algebra of a Deaconu-Renault groupoid $\mathcal{G}$ is AF embeddable, we need to study the $K_0$-group \oldc{$K_0(C^*(\mathcal{G}))$}\newc{$K_0(C^*(\mathcal{G}(c)))$} of the skew-product groupoid $\mathcal{G}(c)$, where $c: \mathcal{G} \rightarrow \mathbb{Z}$ is the continuous cocycle given by $c(x,k,y) = k$. Moreover, we need to know the positive elements of this $K_0$-group. 

It is difficult to understand this \oldc{$K_0(\mathcal{G}(c))$}\newc{$K_0(C^*(\mathcal{G}(c)))$} in detail because its construction is very abstract, depending on equivalence classes of projections in $M_n(C^*(\mathcal{G}(c)))$ for all $n \geq 1$. Fortunately, the groupoid $\mathcal{G}(c)$ is AF, and Theorem \ref{thm:K0H0} of Chapter \ref{section:K0H0} gives an isomorphism from the zeroth homology group of an AF groupoid to the $K_0$-group of the corresponding C*-algebra. This isomorphism is one of our major contributions of this thesis and generalises \cite[Corollary 5.2]{FKPS} and \cite[Theorem 4.10]{Matui}. In Chapter \ref{section:K0H0} we explain how our theorem differs from these two results from the literature.

The advantage of applying this isomorphism is that it allows us to use the techniques from the theory of homology groups for groupoids in order to understand the K-theory of $C^*(\mathcal{G}(c))$. The homology zeroth group \oldc{$H_0(\mathcal{G})$}\newc{$H_0(G)$} of an AF groupoid $G$ is much more tractable than its corresponding K-group -- its construction depends on equivalences of functions in \oldc{$C_c(\mathcal{G}(c)^{(0)}, \mathbb{Z})$}\newc{$C_c(G, \mathbb{Z})$} as opposed to complicated equivalences of projections in $\mathcal{P}_\infty(C^*(G))$. Also, the theory of homology groups provides the notion of homological similarity, which helps us to understand these groups in more detail.

Homological similarity gives isomorphisms of homology groups for different \oldc{group-oids}\newc{groupoids.} We use this technique in Section \ref{section:diagram:part1} to prove the second isomorphism of the commutative diagram \eqref{eqn:diagram} from page \pageref{eqn:diagram}.


Here we assume that $G$, $H$, $G_n$ are locally compact, Hausdorff, second countable, \'etale groupoids with totally disconnected unit spaces. In this chapter, we focus on the main ideas and definitions. So we decided to omit some of proofs that are long or more technical.

We begin in Section \ref{subsection:homology:algebra} by studying the abstract theory of homology groups. Then we define homology groups of groupoids in Section \ref{subsection:homology:groupoids}. Next, in Section \ref{subsection:homology:similarity}, we study homological similarity of groupoids. 

\par\nobreak\section{Homology in algebra}
\label{subsection:homology:algebra}

Let us see the definition of homology groups of chain complexes. The results in this section are in Dummit and Foote's book on abstract algebra \cite[Section 17.1]{dummitfoote}.

\begin{definition}
Let $\mathcal{C}$ be a sequence of abelian group homomorphisms
$$
\begin{tikzcd}
0 &
\arrow{l}{}[swap]{d_0} C_0 &
\arrow{l}{}[swap]{d_1} C_1 &
\arrow{l}{}[swap]{d_2} \cdots
\end{tikzcd}
$$
such that $d_n \circ d_{n+1} = 0$ for all $n \geq 0$. The sequence is called a \newterm{chain complex}. If $\mathcal{C}$ is a chain complex, its $n^{\text{th}}$ \newterm{homology group} is the quotient group $\ker d_{n}/\mathrm{im\hphantom{.}}d_{n+1}$, and is denoted by $H_n(\mathcal{C})$.
\end{definition}

\begin{definition}
\label{def:homomorphism_complexes}
Let $\mathcal{A} = \lbrace A_n \rbrace_{n \geq 0}$ and $\mathcal{B} = \lbrace B_n \rbrace_{n \geq 0}$ be chain complexes. A \newterm{homomorphism of complexes} $\alpha: \mathcal{A} \rightarrow \mathcal{B}$ is a set of homomorphisms $\alpha_n: A_n \rightarrow B_n$ such that for every $n$ the following diagram commutes:
$$
\begin{tikzcd}
\cdots &
A_n \arrow{l}{}[swap]{d_n} \arrow{d}{}[swap]{\alpha_n} &
A_{n+1} \arrow{l}{}[swap]{d_{n+1}} \arrow{d}{}[swap]{\alpha_{n+1}} &
\cdots \arrow{l}{}[swap]{d_{n+2}}\\
\cdots &
B_n \arrow{l}{}[swap]{d_n} &
B_{n+1} \arrow{l}{}[swap]{d_{n+1}} &
\cdots \arrow{l}{}[swap]{d_{n+2}}
\end{tikzcd}
$$
\end{definition}

\begin{proposition}
\label{prop:chain_homology_homomorphism}
A homomorphism $\alpha: \mathcal{A} \rightarrow \mathcal{B}$ of chain complexes induces a group homomorphism $H_n(\alpha): H_n(\mathcal{A}) \rightarrow H_n(\mathcal{B})$, for $n \geq 0$, given by
\begin{align}
\label{eqn:chain_homology_homomorphism}
H_n(\alpha)([a]) = [\alpha_n(a)],
\text{\hspace{10pt}for $a \in A_n$ with $d_n(a) = 0$.}
\end{align}
\end{proposition}
\begin{proof}
We need to prove that $H_n(\alpha)$ is well-defined. In other words, we need to verify that \eqref{eqn:chain_homology_homomorphism} does not depend on the choice of the representative $a$. Let $a_1, a_2 \in A_n$ such that $d_n(a_1) = d_n(a_2) = 0$ and $[a_1] = [a_2]$. By definition of the homology group $H_n(\mathcal{A})$, there is a $c \in A_{n+1}$ such that
\begin{align*}
a_1 = a_2 + d_{n+1}(c).
\end{align*}
By applying $\alpha_n$ on both sides of the equation and using the commutativity of the diagram in Definition \ref{def:homomorphism_complexes}, we obtain
\begin{align*}
\alpha_n(a_1)
= \alpha_n(a_2) + \alpha_n(d_{n+1}(c))
= \alpha_n(a_2) + d_{n+1}(\alpha_{n+1}(c)),
\end{align*}
with $\alpha_n(a_1), \alpha_n(a_2) \in B_n$ and $\alpha_{n+1}(c) \in B_{n+1}$. Then $[\alpha_n(a_1)] = [\alpha_n(a_2)]$.

Finally, we prove that $H_n(\alpha)$ is a homomorphism. Let $a_1, a_2 \in A_n$ such that $d_n(a_1) = d_n(a_2) = 0$. Then $d_n(a_1 + a_2) = 0$ and
\begin{align*}
H_n(\alpha)([a_1 + a_2])
&= [\alpha_n(a_1 + a_2)] \\
&= [\alpha_n(a_1)] + [\alpha_n(a_2)] \\
&= H_n(\alpha)([a_1]) + H_n(\alpha)([a_2])
\end{align*}
Therefore $H_n(\alpha)$ is a homomorphism.
\end{proof}

\par\nobreak\section{Homology for groupoids}
\label{subsection:homology:groupoids}

Here we define homology groups for Hausdorff, second countable, ample groupoids. We use the definitions from Matui's paper \cite{Matui} and Farsi et al's paper \cite{FKPS}. Here we fix $X$, $Y$ to be locally compact, Hausdorff, second countable, totally disconnected spaces. Moreover, $C_c(X, \mathbb{Z})$ (and analogously $C_c(Y, \mathbb{Z})$) denotes the set of continuous and compactly supported functions on $X$ assuming integer values.

\begin{definition}
\label{def:sigmaast}
Let $\sigma: X \rightarrow Y$ be a local homeomorphism. Then, for all $y \in Y$ and $f \in C_c(X, \mathbb{Z})$, we define
\begin{align}
\label{eqn:sigmaast}
\sigma_\ast(f)(y) = \sum_{\sigma(x) = y} f(x)
\end{align}
\end{definition}

\begin{lemma*}
The sum of equation \eqref{eqn:sigmaast} is finite and defines the continuous map $\sigma_\ast(f) \in C_c(Y, \mathbb{Z})$ for all $f \in C_c(X, \mathbb{Z})$.
\end{lemma*}
\begin{proof}
Since $\sigma$ is a local homeomorphism, each $x \in X$ has a compact open neighbourhood $V_x$ on which $\sigma$ is injective. By continuity, the elements of $C_c(X, \mathbb{Z})$ are locally constant and assume finitely many values. Then
\begin{align}
\label{eqn:CcXZspan}
C_c(X, \mathbb{Z}) = \mathrm{span} \lbrace 1_V : V \subset X \text{ is compact open and  $\sigma$ injective on $V$} \rbrace.
\end{align}

Let $V \subset X$ be a compact open subset on which $\sigma$ is injective, and fix $y \in X$. We claim that $\sigma_\ast(1_V)(y) = 1_{\sigma(V)}(y)$.

In fact, suppose that $y \in \sigma(V)$. Then there exists an $x_0 \in V$ such that $x_0 \in V$ and $\sigma(x_0) = y$. Since $\sigma$ is injective on $V$, it follows that for every $x \neq x_0$ with $\sigma(x) = y$, we have $x \notin V$. Then
\begin{align*}
\sum_{x: \sigma(x) = y} 1_V(x)
= 1_V(x_0)
= 1
= 1_{\sigma(V)}(y).
\end{align*}
Now suppose that $y \notin \sigma(V)$. Then, for every $x \in X$ with $\sigma(x) = y$, we must have $x \notin V$. This implies that
\begin{align*}
\sum_{x: \sigma(x) = y} 1_V(x)
= 0
= 1_{\sigma(V)}(y).
\end{align*}
Therefore $\sigma_\ast(1_V)(y) = 1_{\sigma(V)}(y)$ for $y \in X$.

It follows from the linearity of of \eqref{eqn:sigmaast} and from \eqref{eqn:CcXZspan} that, for each $f \in C_c(X, \mathbb{Z})$, equation \eqref{eqn:sigmaast} is finite and $\sigma_\ast(f) \in C_c(X, \mathbb{Z})$.
\end{proof}

Using the Definition \ref{def:sigmaast}, we have $(\phi \circ \psi)_\ast = \phi_\ast \circ \psi_\ast$.

\begin{definition}
Let $n \in \mathbb{N}$ with $n \geq 1$. Define the \newterm{space of composable $n$-tuples} by
\begin{align*}
G^{(n)}
= \lbrace (g_1, \hdots, g_n) \in \underbrace{G \times \dots \times G}_{\text{$n$ times}} : 
s(g_i) = r(g_{i+1}) \text{  for }1 \leq i < n \rbrace,
\end{align*}
while $G^{(0)}$ is the unit space of $G$. For $n \geq 2$, we equip $G^{(n)}$ with the subspace topology from $G^n = \underbrace{G \times \dots \times G}_{\text{$n$ times}}$.
\end{definition}

\begin{definition}
Let $n \geq 1$ and $0 \leq i \leq n$. We define the maps $d_i^{(n)} : G^{(n)} \rightarrow G^{(n-1)}$ such that, for $n \geq 2$,
\begin{align*}
d_i^{(n)}(g_1, \hdots, g_n) =
\begin{cases}
(g_2, \hdots, g_n) & i = 0,\\
(g_1, \hdots, g_{i-1}, g_i g_{i+1}, g_{i+2}, \hdots, g_n) & 1 \leq i \leq n - 1\\
(g_1, \hdots, g_{n-1}) & i = n.
\end{cases}
\end{align*}
If $n = 1$, we set $d_0^{(n)} = s$ and $d_1^{(n)} = r$.
\end{definition}

Using that $G$ is étale, we can show that $d_i^{(n)}$ is a local homeomorphism.

\begin{definition}
\label{def:partialn}
For $n \geq 1$, define $\partial_n: C_c(G^{(n)}, \mathbb{Z}) \rightarrow C_c(G^{(n-1)}, \mathbb{Z})$ by
\begin{align*}
\partial_n = \sum_{i=0}^n (-1)^i (d_i^{(n)})_\ast.
\end{align*}
We set $\partial_0: C_c(G^{(0)}, \mathbb{Z}) \rightarrow 0$.
\end{definition}

We want to show that the sequence
$$
\begin{tikzcd}
0 &
\arrow{l}{}[swap]{\partial_0 = 0} C_c(G^{(0)}, \mathbb{Z}) &
\arrow{l}{}[swap]{\partial_1} C_c(G, \mathbb{Z}) &
\arrow{l}{}[swap]{\partial_2} C_c(G^{(2)}, \mathbb{Z}) &
\arrow{l}{}[swap]{\partial_3}\cdots
\end{tikzcd}
$$
(also denoted $(C_c(G^{(\ast)}, \mathbb{Z}), \partial_\ast)$)  is a chain complex. So we prove that, for all $n \geq 1$,
\begin{align}
\label{eqn:compositionpartial}
\partial_n \partial_{n+1} = \sum_{i=0}^n \sum_{j=0}^{n+1}(-1)^{i+j} (d_i^{(n)} d_j^{(n+1)})_\ast = 0.
\end{align}

In order to prove this, we need the following lemma.
\begin{lemma}
\label{lemma:didj}
For $0 \leq i < j \leq n + 1$, we have
\begin{align}
\label{eqn:didj}
d_i^{(n)} d_j^{(n+1)} = d_{j-1}^{(n)} d_i^{(n+1)}.
\end{align}
\end{lemma}
\begin{proof}
In this proof, we divide $n, i, j$ into eight cases and verify the equation above for each of them using the properties of $G^{(n+1)}$.

Let $g = (g_1, \dots, g_{n+1}) \in G^{(n+1)}$. Note that $s(g_l) = r(g_{l+1})$ for $l = 1, \dots, n$.

\begin{enumerate}
\item Suppose $n=1$, $i = 0$ and $j = 1$.
\begin{align*}
d_0^{(1)} d_1^{(2)} (g_1, g_2) &= d_0^{(1)}(g_1 g_2) = s(g_1 g_2) = s(g_2), \\
d_0^{(1)} d_0^{(2)}(g_1, g_2) &= d_0^{(1)}(g_2) = s(g_2).
\end{align*}
\item Suppose $n = 1, i = 0$ and $j = 2$.
\begin{align*}
d_0^{(1)} d_2^{(2)}(g_1, g_2)
&= d_0^{(1)} (g_1) = s(g_1), \\
d_1^{(1)} d_0^{(2)} (g_1, g_2)
&= d_1^{(1)}(g_2) = r(g_2) = s(g_1).
\end{align*}
\item Suppose $n = 1$, $i = 1$ and $j = 2$.
\begin{align*}
d_1^{(1)} \oldc{d_-2^{(2)}}\newc{d_2^{(2)}} (g_1, g_2)
&= d_1^{(1)}(g_1) = r(g_1), \\
d_1^{(1)} \oldc{d_1^{(1)}}\newc{d_1^{(2)}}(g_1, g_2)
&= d_1^{(1)}\oldc{(g_1, g_2)}\newc{(g_1 g_2)} = r(g_1 g_2) = r(g_1).
\end{align*}
\item \newc{Suppose $n \geq 2$, $i = 0$ and $j = 1$.
\begin{align*}
d_0^{(n)} d_1^{(n+1)}(g_1, \dots, g_{n+1})
&= d_0^{(n)}(g_1 g_2, g_3, \dots, g_{n+1}) \\
&= (g_3, \dots, g_{n+1}), \\
d_0^{(n)} d_0^{(n+1)}(g_1, \dots, g_{n+1})
&= d_0^{(n)}(g_2, \dots, g_{n+1}) \\
&= (g_3, \dots, g_{n+1}).
\end{align*}}
\item Suppose $n \geq 2$, $i = 0$ and $\oldc{i}\newc{1} < j \leq n$.
\begin{align*}
d_0^{(n)} d_j^{(n+1)}(g_1, \dots, g_{n+1})
&= d_0^{(n)}(g_1, \dots, g_j g_{j+1}, \dots, g_{n+1})\\
& = (g_2, \dots, g_j g_{j+1}, \dots, g_{n+1}), \\
d_{j-1}^{(n)} d_0^{(n+1)}(g_1, \dots, g_{n+1})
&= d_{j-1}^{(n)}(g_2, \dots, g_{n+1}) \\
&= (g_2, \dots, g_j g_{j+1}, \dots, g_{n+1}).
\end{align*}
\item Suppose $n \geq 2$, $i = 0$ and $j = n + 1$.
\begin{align*}
d_0^{(n)} d_{n+1}^{(n+1)}(g_1, \dots, g_{n+1})
&= d_0^{(n)}(g_1, \dots, g_n) = (g_2, \dots, g_n), \\
d_n^{(n)} d_0^{(n+1)} (g_1, \dots, g_{n+1})
&= d_n^{(n)} (g_2, \dots, g_{n+1})
= (g_2, \dots, g_n).
\end{align*}
\item Suppose $n \geq 2$, $1 \leq i \leq n$ and $i + 1 < j \leq n$.

Note that the equality $d_i^{(n)} d_j^{(n+1)} = d_{j-1}^{(n)} d_i^{(n+1)}$ is obvious if $j = i + 1$. So suppose $j > i + 1$. Then
\begin{align*}
d_i^{(n)} d_j^{(n+1)}(g_1, \dots, g_{n+1})
&= d_i^{(n)}(g_1, \dots, g_j g_{j+1}, \dots, g_{n+1}) \\
&= (g_1, \dots, g_i g_{i+1}, \dots, g_j g_{j+1}, \dots, g_{n+1}), \\
d_{j-1}^{(n)} d_i^{(n+1)}(g_1, \dots, g_{n+1})
&= d_{j-1}^{(n)}(g_1, \dots, g_i g_{i+1}, \dots, g_{n+1}) \\
&= (g_1, \dots, g_i g_{i+1}, \dots, g_j g_{j+1}, \dots, g_{n+1}).
\end{align*}
\item Suppose $n \geq 2$, $i \leq i \leq n - 1$ and $j = n + 1$.
\begin{align*}
d_i^{(n)} d_{n+1}^{(n+1)}(g_1, \dots, g_{n+1})
&= d_i^{(n)}(g_1, \dots, g_n)
= (g_1, \dots, g_i g_{i+1}, \dots, g_n), \\
d_n^{(n)} d_i^{(n+1)}(g_1, \dots, g_{n+1})
&= d_n^{(n)}(g_1, \dots, g_i g_{i+1}, \dots, g_{n+1}) \\
&= (g_1, \dots, g_i g_{i+1}, \dots, g_n).
\end{align*}
\item Suppose $n \geq 2$, $i = n$ and $j = n + 1$.

In this case, the equation $d_i^{(n)} d_{n+1}^{(n+1)} = d_n^{(n)} d_i^{(n+1)}$ is obvious.
\end{enumerate}
We went \oldc{throught}\newc{through} all possible cases and showed \oldc{that $d_n^{(n)} d_i^{(n+1)}$}\newc{\eqref{eqn:didj}}.
\end{proof}


\begin{proposition}
The sequence $(C_c(G^{(\ast)}, \mathbb{Z}), \partial_\ast)$ is a chain complex.
\end{proposition}
\begin{proof}
As mentioned on page \pageref{eqn:compositionpartial}, we need to prove that $\partial_n \partial_{n+1} = 0$ for $n \geq 1$.

Note that we write $\lbrace 0, \dots, n \rbrace \times \lbrace 0, \dots, n+1 \rbrace$ as the disjoint union of the sets $I_<$ and $I_\geq$, where
\begin{align*}
I_< = \lbrace (i,j) : 0 \leq i < j \leq n + 1 \rbrace
\hspace{15pt}\text{and}\hspace{15pt}
I_\geq = \lbrace (i,j) : 0 \leq j < i \leq n \rbrace.
\end{align*}
Moreover, there exists a bijection $I_< \rightarrow I_\geq$ given by $(i,j) \mapsto (j-1, i)$. Using this bijection, we can write $\partial_n \partial_{n+1}$ as
\begin{align*}
\partial_n \partial_{n+1}
&= \sum_{i=0}^n \sum_{j=0}^{n+1} (-1)^{i+j} (d_i^{(n)} d_j^{(n+1)})_\ast \\
&= \sum_{i=0}^n \sum_{j = i + 1}^{n+1} [(-1)^{i+j}(d_i^{(n)} d_j^{(n+1)} )_\ast + (-1)^{j-1+i}(d_{j-1}^{(n)} d_i^{(n+1)})_\ast] \\
&= \sum_{i=0}^n \sum_{j=i+1}^{n+1} (-1)^{i+j}[(d_i^{(n)} d_j^{(n+1)})_\ast - (d_{j-1}^{(n)} d_i^{(n+1)})_\ast].
\intertext{Note that Lemma \ref{lemma:didj} implies that $(d_i^{(n)} d_j^{(n+1)})_\ast - (d_{j-1}^{(n)} d_i^{(n+1)})_\ast = 0$. Then}
\partial_n \partial_{n+1} &= \sum_{i=0}^n \sum_{j=i+1}^{n+1} 0 = 0.
\end{align*}
This implies that $(C_c(G^{(\ast)}, \mathbb{Z}), \partial_\ast)$  is a chain complex.
\end{proof}

\begin{definition}
\label{def:gpdhomology}
For $n \geq 0$, let $H_n(G)$ be the $n^{\text{th}}$ homology group of the chain complex $(C_c(G^{(\ast)}, \mathbb{Z}), \partial_\ast)$. In addition, we define the set of positive elements
\begin{align*}
H_0(G)^+ = \lbrace [f] \in H_0(G) : f(x) \geq 0 \text{ for all }x \in G^{(0)} \rbrace.
\end{align*}
\end{definition}

The operations on $H_0(G)$ are given by $[f_1] + [f_2] = [f_1 + f_2]$ and $-[f] = [-f]$. The group $H_0(G)$ is not necessarily an ordered group because the intesection $H_0(G)^+ \cap (- H_0(G)^+)$ may be different from $\lbrace 0 \rbrace$. See \cite[Remark 3.2]{Matui}.

We will see that certain homomorphisms of groupoids induce homomorphisms between the homology groups of these groupoids.

\begin{definition}
Let $\rho: G \rightarrow H$ be a homomorphism between étale groupoids. We define the map $\rho^{(0)}: G^{(0)} \rightarrow H^{(0)}$ by $\rho^{(0)} = \rho\vert_{G^{(0)}}$. Given $n \geq 1$, we define the map $\rho^{(n)}: G^{(n)} \rightarrow H^{(n)}$ such that
\begin{align*}
\rho^{(n)}(g_1, \hdots, g_n) = (\rho(g_1), \hdots, \rho(g_n)),
\text{\hspace{10pt}for $(g_1, \hdots, g_n) \in G^{(n)}$.}
\end{align*}
\end{definition}

Using the properties of \'etale groupoids, we have the lemma below. This lemma was stated without proof in \cite[Subsection 3.2]{Matui}.

\begin{lemma}{}
\label{lemma:rhon-equivalent}
Let $G, H$ be étale groupoids and let $\rho: G \rightarrow H$ be a homomorphism. Then the following are equivalent:
\begin{enumerate}[(i)]
\item $\rho^{(0)}$ is a local homeomorphism,
\item $\rho$ is a local homeomorphism,
\item $\rho^{(n)}$ is a local homeomorphism for all $n \in \mathbb{N}$.
\end{enumerate}
\end{lemma}
\begin{proof}
First we show (i) $\Rightarrow$ (ii). We begin by showing that $\rho$ is open. Note that $\rho \circ r: G \rightarrow H^{(0)}$ is a local homeomorphism since it is the composition of the local homeomorphisms $r: G \rightarrow G^{(0)}$ and $\rho^{(0)}: G^{(0)} \rightarrow H^{(0)}$.

\newc{Let $\mathcal{U}$ be an open set such that $r \circ \rho(\mathcal{U})$ is open and $r \circ \rho\vert_{\mathcal{U}}: \mathcal{U} \rightarrow r(\rho(\mathcal{U}))$ is a homeomorphism. Suppose $\rho(\mathcal{U})$ is not open. Then there exists a $g \in \mathcal{U}$ such that}
\begin{align*}
\newc{V \setminus \rho(\mathcal{U}) \neq \emptyset,}
\end{align*}
\newc{for every open set $V$ containing $\rho(g)$.}

\newc{Let $V_0$ be an open neighbourhood of $\rho(g)$ such that $r$ is injective on $V_0$ and $r(V_0)$ is an open subset of $r \circ \rho(\mathcal{U})$. Since $H$ is Hausdorff, we choose a decreasing sequence of distinct open sets $V_0 \supset V_1 \supset V_2 \supset \dots$ with intersection $\lbrace \rho(g) \rbrace$.}

\newc{Let $h_i$ be a sequence in $H$ such that $h_i \in V_i \setminus \rho(\mathcal{U})$ for $i=1,2, \dots$. Then $h_i \rightarrow \rho(g)$. Moreover, $r(h_i) \rightarrow r(\rho(g))$. Note that each $r(h_i) \in r \circ \rho(\mathcal{U})$.}

\newc{Let $g_i = (\rho \circ r)\vert_{\mathcal{U}}^{-1}(r(h_i))$ for all $i$. Since $\rho \circ r$ is a homeomorphism on $\mathcal{U}$, it follows that $g_i \rightarrow g$. By continuity of $\rho$, we have that $\rho(g_i) \rightarrow \rho(g)$.}

\newc{For a suffieciently large $i$, we have $\rho(g_i) \in V_1$. But $r(\rho(g_i)) = r(h_i)$ and $h_i \in V_0$. Since $r$ is injective on $V_0$, we must have $\rho(g_i) = h_i$. This implies that $h_i \in \rho(\mathcal{U})$, which is a contradiction. Then $\rho(\mathcal{U})$ is open. Since $\mathcal{U}$ is arbitrary and $r \circ \rho$ is a local homeomorphism, we have that $\rho$ is an open map.}


Fix $g \in G$, and let $\mathcal{U}$ be an open bisection containing $g$ such that $\rho^{(0)}$ is injective on $r(\mathcal{U})$, and such that $\rho(\mathcal{U})$ is an open bisection. We claim that $\rho$ is injective on $\mathcal{U}$. Indeed, let $g' \in \mathcal{U}$ be such that $\rho(g) = \rho(g')$. Since $\rho$ is a homomorphism, we have
\begin{align*}
\rho(r(g)) = r(\rho(g)) = r(\rho(g')) = \rho(r(g')).
\end{align*}
Since $r(g), r(g') \in r(\mathcal{U})$ and $\rho$ is injective on $r(\mathcal{U})$, we have $r(g) = r(g')$. But $\mathcal{U}$ is an open bisection, thus $g = g'$.

Since $G, H$ are \'etale and $\rho^{(0)}$ is a homeomorphism, there are an open bisection $\mathcal{U}_1 \subset \mathcal{U}$ containing $g$, and an open bisection $V_1 \subset \rho(\mathcal{U})$ containing $\rho(g)$ such that $\rho \circ r(\mathcal{U}_1) = r(V_1)$. Note that $\rho(\mathcal{U}_1) \subset \rho(\mathcal{U})$ is an open bisection. Let $T: \mathcal{U}_1 \rightarrow V_1$ be the homeomorphism given by
\begin{align*}
g' \mapsto T(g') = r\vert_{V_1}^{-1} \circ \rho \circ r(g').
\end{align*}
Then $\rho(g'), T(g')$ are in $r(\mathcal{U})$, and
\begin{align*}
r(T(g')) = r( r\vert_{V_1}^{-1} \circ \rho \circ r(g') )
= \rho \circ r(g'))
= r(\rho(g')).
\end{align*}
Since $\mathcal{U}_1$ is an open bisection, we have $T(g') = \rho(g')$. Therefore $\rho$ is an open homeomorphism.

The implication (ii) $\Rightarrow$ (i) is straightforward because $\rho^{(0)}$ is the restriction of $\rho$ on the unit space $G^{(0)}$\newc{, which is an open subset of $G$}. Now we prove (ii) $\Rightarrow$ (iii). Suppose $\rho$ is a local homeomorphism and let \oldc{$n \geq 1$}\newc{$n \geq 2$}.

\oldc{Then}\newc{Note that} the map $\rho^n: G^n \rightarrow H^n$ defined by
\begin{align*}
\rho^{n} = \underbrace{\rho \times \dots \times \rho}_{\text{$n$ times}}
\end{align*}
is a local homeomorphism. \newc{Let $g = (g_1, \dots, g_n) \in G$, and let $\mathcal{U}_1, \dots \mathcal{U}_n$ be open bisections such that:}\newc{
\begin{itemize}
\item $g_i \in \mathcal{U}_i$ for all $i$,
\item $r(\mathcal{U}_i) = s(\mathcal{U}_{i+1})$ for $i = 1, \dots, n-1$, and
\item $\rho\vert_{r(\mathcal{U}_i)}: r(\mathcal{U}_i) \rightarrow \rho( r(\mathcal{U}_i))$ is a homeomorphism for  $i=1, \dots, n-1$.
\end{itemize}}

\newc{Let $\mathcal{U} = \mathcal{U}_1 \times \cdots \times \mathcal{U}_n$. We claim that $\rho^{n}(\mathcal{U} \cap G^{(n)}) = \rho^{n}(\mathcal{U}) \cap H^{(n)}$. Indeed, let $\widetilde{g} = (\widetilde{g}_1, \dots, \widetilde{g}_n) \in \mathcal{U} \cap G^{(n)}$. Then $\rho^{(n)}(\widetilde{g}) \in \rho^{n}(\mathcal{U})$. Moreover, for $i = 1, \dots, n-1$, we have}
\newc{\begin{align*}
s(\rho(\widetilde{g}_i))
= \rho(s(\widetilde{g}_i))
= \rho(r(\widetilde{g}_{i+1}))
= r(\rho(\widetilde{g}_{i+1})),
\end{align*}}
\newc{thus $\rho^{n}(\widetilde{g}) \subset H^{(n)}$. Then we have the inclusion $\rho^{n}(\mathcal{U} \cap G^{(n)}) \subset \rho^{n}(\mathcal{U}) \cap H^{(n)}$.}

\newc{Now, let $\widetilde{g} = (\widetilde{g}_1, \dots, \widetilde{g}_n) \in \mathcal{U}$ be such that $\rho(\widetilde{g}) \in H^{(n)}$. Then, for all $i = 1, \dots, n-1$, we have}
\newc{
\begin{align*}
s(\rho(\widetilde{g}_i)) = r(\rho(\widetilde{g}_{i+1}))
\Rightarrow 
\rho(s(\widetilde{g}_i)) = \rho(r(\widetilde{g}_{i+1}))
\end{align*}
}
\newc{By assumption, $s(\widetilde{g}_i), r(\widetilde{g}_{i+1}) \in r(\mathcal{U}_i)$. Since $\rho$ is injective on this subset, it follows that $s(\widetilde{g}_i) = r(\widetilde{g}_{i+1})$. Therefore, $\widetilde{g} \in G^{(n)}$. This implies that $\rho^{n}(\mathcal{U} \cap G^{(n)}) = \rho^{n}(\mathcal{U}) \cap H^{(n)}$. Moreover, $\rho^{n}(\mathcal{U} \cap G^{(n)})$ is open, and}
\newc{
\begin{align*}
\rho^{n}(\mathcal{U} \cap G^{(n)}) = \rho^{n}(\mathcal{U}) \cap H^{(n)} = \rho^{(n)}(\mathcal{U}).
\end{align*}
}
\newc{Note that $\rho^{(n)}$ that is continuous and injective on $\mathcal{U} \cap G^{(n)}$. Then the restriction $\rho^{(n)}\vert_{\mathcal{U} \cap G^{(n)}}: \mathcal{U} \cap G^{(n)} \rightarrow \rho^{(n)}(\mathcal{U})$ is a homeomorphism. Therefore, $\rho^{(n)}$ is a local homeomorphism.}


The implication (iii) $\Rightarrow$ (i) is straightforward. This completes the proof.
\end{proof}

We use the previous lemma in Chapter \ref{section:diagram} to prove the commutative diagram \eqref{eqn:diagram} of page \pageref{eqn:diagram}. The following proposition \oldc{describe}\newc{describes} homology group homomorphisms induced by groupoid homomorphisms.

\begin{proposition}
Let $\rho: G \rightarrow H$ be a homomorphism that is also a local homeomorphism. Then, for each $n \geq 0$, $\rho$ induces the homomorphism $H_n(\rho): H_n(G) \rightarrow H_n(H)$ defined by
\begin{align*}
H_n(\rho)([f]) = [\rho_\ast^{(n)}(f)],
\text{\hspace{10pt}for $f \in C_c(G^{(n)}, \mathbb{Z})$.}
\end{align*}
\end{proposition}

\begin{remark}
\label{rmk:rhobracket}
In this thesis, for a function $\tau: \newc{C_c(G^{(0)}, \mathbb{Z})} \rightarrow \oldc{C_c(H^{(0}, \mathbb{Z})}\newc{C_c(H^{(0)}, \mathbb{Z})}$ we use the notation $[\tau]$ to denote the function $[\tau]: H_0(G) \rightarrow H_0(H)$ \newc{given} by $[\tau]([f]) = [\tau(f)]$. Note that, for an arbitrary $\tau$, we need to prove that $[\tau]$ is well-defined before applying this notation. Moreover, for a homomorphism $\rho: G \rightarrow H$ that is a local homeomorphism, we have \oldc{$H_0(\rho) = [\rho^{(0)}]$}\newc{$H_0(\rho) = [\rho^{(0)}_\ast]$}.
\end{remark}

Although we omit the long proof of the previous proposition, we make some comments about it. From Definition \ref{def:homomorphism_complexes} and Proposition \ref{prop:chain_homology_homomorphism}, we only need to show that $\rho$ induces the homomorphism of complexes \oldc{$\alpha: \lbrace G^{(n)} \rbrace_{n\geq 0} \rightarrow \lbrace H^{(n)} \rbrace_{n\geq 0}$ with $\alpha_n = \rho^{(n)}$}\newc{$\alpha: \lbrace C_c(G^{(n)}, \mathbb{Z}) \rbrace_{n \geq 0} \rightarrow \lbrace C_c(H^{(n)}, \mathbb{Z}) \rbrace_{n \geq 0}$, with $\alpha_n = H_n(\rho)$}. In other words, we need to show that
\begin{align*}
\oldc{\rho^{(n)} \circ \partial_{n+1} = \partial_{n+1} \circ \rho^{(n+1)}}
\newc{H_n(\rho) \circ \partial_{n+1} = \partial_{n+1} \circ H_{n+1}(\rho)}
\text{\hspace{10pt}for } n \geq 1.
\end{align*}
However, \newc{by} applying the definition of $\partial_n$ and Definition \ref{def:sigmaast}, it is sufficient to show that
\begin{align*}
\rho^{(n)} \circ d_i^{(n+1)} = d_i^{(n+1)} \circ \oldc{\rho^{(n)}}\newc{\rho^{(n+1)}},
\end{align*}
which is easier to verify. Indeed, we show the case $\rho^{(n)} \circ d_0^{(n+1)} = d_0^{(n+1)} \circ \rho^{(n)}$. Given $(g_1, \dots, g_{n+1}) \in G^{(n+1)}$, we have
\begin{align*}
\rho^{(n)} \circ d_0^{(n+1)}(g_1, \dots, g_{n+1})
&= \rho^{(n)}(g_2, \dots, g_{n+1}) \\
&= (\rho(g_2), \dots, \rho(g_{n+1})) \\
&= d_0^{(n+1)}(\rho(g_1), \dots, \rho(g_{n+1})) \\
&= d_0^{(n+1)} \circ \oldc{\rho^{(n)}}\newc{\rho^{(n+1)}(g_1, \dots, g_{n+1})}.
\end{align*}
The other cases can be proved analogously.

The following lemma gives a condition for the equivalence of functions in $H_0(R(\sigma))$. We apply this condition in Chapter \ref{section:K0H0} to study Matui's HK conjecture for AF groupoids.

\begin{lemma}
\label{lemma:traceH0Rsigma}
Let $\sigma: X \rightarrow Y$ be a surjective local homeomorphism. Then two functions $f_1, f_2 \in \oldc{C_c(X, \mathbb{Z})}\newc{C_c(X, \mathbb{N})}$ are equivalent in $H_0(R(\sigma))$ if, and only if,
\begin{align}
\label{eqn:equivconditionH0Rsigma}
\sum_{u: \sigma(u) = \sigma(x)} f_1(u)
= \sum_{u: \sigma(u) = \sigma(x)} f_2(u)
\hspace{15pt}\newc{\text{for all $x \in X$.}}
\end{align}
\end{lemma}
\begin{proof}
Let $f_1, f_2 \in C_c(X, \mathbb{N})$ be equivalent in $H_0(R(\sigma))$. Then there exists $F \in C_c(R(\sigma), \mathbb{Z})$ such that, for all $u \in X$,
\begin{align*}
f_1(u)
&= f_2(u) + \partial_1 F(u) \\
&= f_2(u) + s_\ast(F)(u) - r_\ast(F)(u) \\
&= f_2(u) + \sum_{v: \sigma(v) = \sigma(u)} [F(v,u) - F(u,v)].
\end{align*}
Fix $x \in X$. Then
\begin{align*}
\sum_{u: \sigma(u) = \sigma(x)} f_1(u)
&= \sum_{u: \sigma(u) = \sigma(x)} f_2(u) + \sum_{\substack{u,v \\ \sigma(u) =  \sigma(x) \\ \sigma(v) = \sigma(x)}} (F(v,u) - F(u,v)) \\
&= \sum_{u: \sigma(u) = \sigma(x)} f_2(u) + \sum_{\substack{u,v \\ \sigma(u) =  \sigma(x) \\ \sigma(v) = \sigma(x)}} F(v,u) - \sum_{\substack{u,v \\ \sigma(u) =  \sigma(x) \\ \sigma(v) = \sigma(x)}} (F(u,v).
\end{align*}
Note that the second and the third sums of the right-hand side are equal. Then \eqref{eqn:equivconditionH0Rsigma} holds.

Conversely, suppose that \eqref{eqn:equivconditionH0Rsigma} holds, and fix a continuous section $\varphi: Y \rightarrow X$ of $\sigma$. Let $T: C_c(X, \mathbb{N}) \rightarrow C_c(X, \mathbb{N})$ be given by
\begin{align}
\label{eqn:Tfx}
T(f)(x) = 1_{X_\varphi}(x) \sum_{u: \sigma(u) = \sigma(x)} f(u),
\end{align}
\newc{where $X_\varphi = \varphi(Y)$.}

We need to prove that $T$ is well-defined. Let $V \subset X$ be a compact open subset of which $\sigma$ is injective. We claim that $T(1_V) = 1_{\varphi(\sigma(V))}$. Indeed, let $x \in X$. Suppose that $T(1_V)(x) \neq 0$. Equation \eqref{eqn:Tfx} implies that $x \in X_\varphi$. Moreover, there exists $u \in X_\varphi$ with
\begin{align}
\label{eqn:1Vusigmax}
1_V(u) = 1 \text{\hspace{10pt} and \hspace{10pt}} \sigma(u) = \sigma(x).
\end{align}
Thus,  $u \in V$. Since $\sigma$ is injective on $V$, $u$ is the unique element in $X$ such that \eqref{eqn:1Vusigmax} holds. Then $T(1_V)(x) = 1$. Since $\sigma(u) = \sigma(x)$ and $x \in X_\varphi$, then $x = \varphi(\sigma(u)) \in \varphi(\sigma(V))$. Thus $1_{\varphi(\sigma(V))}(x) = 1$, and then $T(1_V)(x) = 1_{\varphi(\sigma(V))}(x)$.

Now suppose that $T(1_V)(x) = 0$. Then $x \notin X_\varphi$ or $\sum_{u: \sigma(u) = \sigma(x)} 1_V(u) = 0$. If $x \notin X_\varphi$, then $x \notin\varphi(\sigma(V))$ because $\varphi(\sigma(V)) \subset X_\varphi$. Then $1_{\varphi(\sigma(V))}(x) = 0$. If $\sum_{u: \sigma(u) = \sigma(x)} 1_V(u) = 0$, then
\begin{align*}
V \cap \sigma^{-1}(\sigma(x)) = \emptyset
&\Rightarrow \sigma(V) \cap \lbrace \sigma(x) \rbrace =\emptyset \\
&\Rightarrow \varphi(\sigma(V)) \cap \lbrace x \rbrace = \emptyset \\
&\newc{\Rightarrow x \notin \varphi(\sigma(V))} \\
&\Rightarrow 1_{\varphi(\sigma(V))}(x) = 0.
\end{align*}
Therefore $T(1_V) = 1_{\varphi(\sigma(V))}$.

Now let $f \in C_c(X, \mathbb{N})$ be non-zero (the case $f = 0$ is trivial). Then there are finitely many compact open subsets $V_1, \dots, V_n$ such that $\sigma$ is injective on each of them, and $f = 1_{V_1} + \dots + 1_{V_n}$. We know that each $T(1_{V_i})$ is in $C_c(X, \mathbb{N})$. It follows from \eqref{eqn:Tfx} that $T$ is a semigroup homomorphism. Then $T(f) \in C_c(X, \mathbb{N})$. Therefore $T$ is well-defined.

Note that, for $f \in C_c(X, \mathbb{N})$ and $x \in X$, we have
\begin{align}
\label{eqn:Tfvarphisigmax}
\newc{T(f)(\varphi(\sigma(x)))}
= \sum_{u: \sigma(u) = \sigma(\varphi(\sigma(x))) } f(u)
= \sum_{u: \sigma(u) = \sigma(x) } f(u).
\end{align}

Since \eqref{eqn:equivconditionH0Rsigma} holds, then $T(f_1) = T(f_2)$. Define the set
\begin{align*}
K = \lbrace x \in X : x \in \mathrm{supp\hphantom{.}}f_1 \cup \mathrm{supp\hphantom{.}} f_2 \text{\hspace{7pt} and \hspace{7pt}} T(f_1)(\varphi(\sigma(x))) \neq 0 \rbrace.
\end{align*}
This set is compact because it is the intesection of the compact set $\mathrm{supp\hphantom{.}}f_1 \cup \mathrm{supp\hphantom{.}} f_2$ and the closed set \oldc{$(T(f) \circ \varphi \circ \sigma)^{-1}(\mathbb{Z} \setminus \lbrace 0 \rbrace)$}\newc{$(T(f_1) \circ \varphi \circ \sigma)^{-1}(\mathbb{Z} \setminus \lbrace 0 \rbrace)$}.

Let $F \in C_c(R(\sigma), \mathbb{Z})$ be defined by
\begin{align*}
F(x,y)
&= \begin{cases}
\displaystyle\frac{f_1(y)f_2(x)}{T(f_1)(\varphi(\sigma(x)))}
& \text{if } (x,y) \in K \times \mathrm{supp\hphantom{.}}f_2,\\
0 &\text{otherwise.}
\end{cases}
\end{align*}
Then, for $x \in K$, we have
\begin{align}
\sum_{y: \sigma(y) = \sigma(x)} F(x,y)
&= \frac{f_2(x)}{T(f_1)(\varphi(\sigma(x)))} \sum_{y: \sigma(y) = \sigma(x)} f_1(y) \nonumber\\
&= \frac{f_2(x)}{T(f_1)(\varphi(\sigma(x)))} T(f_1)(\varphi(\sigma(x))) \text{\hspace{20pt} by \eqref{eqn:Tfvarphisigmax}} \nonumber\\
&= f_2(x). \label{eqn:partialFf1}
\end{align}
If $x \in X \setminus K$, we have
\begin{align}
\label{eqn:partialFf1zero}
\sum_{y: \sigma(y) = \sigma(x)} F(x,y) = 0 = f_2(x).
\end{align}
By similar arguments, we have
\begin{align}
\label{eqn:partialFf2}
\sum_{y: \sigma(y) = \sigma(x)} \oldc{F(x,y)}\newc{F(y,x)} = f_1(x)
\hspace{20pt}\text{for $x \in X$.}
\end{align}
Combining, \eqref{eqn:partialFf1}, \eqref{eqn:partialFf1zero} and \eqref{eqn:partialFf2}, we obtain for all $x \in X$,
\begin{align*}
f_1(x) = f_2(x) + \partial_1 F(x).
\end{align*}
Therefore $f_1 \sim f_2$.
\end{proof}

Now we show that the functor $H_0$ is continuous with respect to inductive limits, where $G = \displaystyle\lim_\rightarrow G_n$ is an inductive limit as in Lemma \ref{lemma:inductivelimitgroupoids}. That is, $G$ is the increasing union of the open subsets $G_n$, where $G_n^{(0)} = G^{(0)}$.

\begin{lemma}
\label{lemma:partial1F}
Let $G = \displaystyle\lim_\rightarrow G_n$. For each $n$, let $\partial_1^{(n)}: C_c(G_{n}, \mathbb{Z}) \rightarrow C_c(G^{(0)}, \mathbb{Z})$ be defined by $\partial_1^{(n)} = (s\vert_{G_n})_\ast - (r\vert_{G_n})_\ast$. Analogously, let $\partial_1: C_c(G, \mathbb{Z}) \rightarrow C_c(G^{(0)}, \mathbb{Z})$ be given by $\partial_1 = s_\ast - r_\ast$. Let $F \in C_c(G_n, \mathbb{Z})$. Then for all $k \geq 0$, we have
\begin{align*}
\partial_1 F = \partial_1^{(n)} F = \partial_1^{(n+k)} F.
\end{align*}
\end{lemma}
\begin{proof}
Note that $F \in C_c(G_n, \mathbb{Z}) \subset C_c(G_{n+1}, \mathbb{Z}) \subset C_c(G, \mathbb{Z})$. If $\oldc{(y_1, y_2)}\newc{g} \in G$ is such that $\oldc{F(y_1, y_2)}\newc{F(g)} \neq 0$, then $\oldc{(y_1, y_2)}\newc{g} \in G_n$. Then, for each $\oldc{x \in X}\newc{x \in G^{(0)}}$, we have
\begin{align*}
\partial_1(F)(x)
&= s_\ast(F)(x) - r_\ast(F)(x)\\
&= \newc{\sum_{g \in G \hphantom{.}:\hphantom{.} s(g) = x} F(g) - \sum_{g \in G \hphantom{.}:\hphantom{.} r(g) = x} F(g) } \\
&= \newc{\sum_{g \in G_n \hphantom{.}:\hphantom{.} s(g) = x} F(g) - \sum_{g \in G_n \hphantom{.}:\hphantom{.} r(g) = x} F(g) } \\
&= \partial_1^{(n)}(F)(x).
\end{align*}
Thus $\partial_1(F) = \partial_1^{(n)}(F)$. Analogously, $\partial_1^{(n+1)}(F) = \partial_1^{(n)}(F)$.
\end{proof}

\begin{lemma}
\label{lemma:limH0}
Let $G = \displaystyle\lim_\rightarrow G_n$. Then $H_0(G) = \displaystyle\lim_\rightarrow H_0(G_n)$ as an inductive limit of groups. Moreover, if each $H_0(G_n)$ is an ordered group, then $H_0(G)$ is an ordered group.
\end{lemma}
\begin{proof}
Let $[\hphantom{f}]_n$ and $[\hphantom{f}]$ denote the equivalence classes in $H_0(G_n)$ and $H_0(G)$, respectively. For each $n$, let $i_n: H_0(G_n) \rightarrow H_0(G_{n+1})$ be the inclusion maps defined by
\begin{align*}
i_n([f]_n) = [f]_{n+1}.
\end{align*}
Note that $i_n$ is well-defined. In fact, let $f_1, f_2 \in C_c(G^{(0)}, \mathbb{Z})$ be such that $f_1, f_2$ are equivalent in $H_0(G_n)$. Then there exists \oldc{$F \in C_c(G, \mathbb{Z})$}\newc{$F \in C_c(G_n, \mathbb{Z})$} such that
\begin{align*}
f_1 &= f_2 + \partial_1^{(n)} F \\
&= f_2 + \partial_1^{(n+1)} F
\hspace{20pt}\text{by Lemma \ref{lemma:partial1F}.}
\end{align*}
Then $f_1 \sim f_2$ in $H_0(G_{n+1})$. This implies that $i_n$ is well-defined.

For every $n \geq 1$, let $\alpha_n: H_0(G_n) \rightarrow H_0(G)$ be given by
\begin{align*}
\alpha_n([f]_n) = [f].
\end{align*}
Using analogous arguments to the proof for $i_n$, we can show that $\alpha_n$ is well-defined. Note that both $\alpha_n$ and $i_n$ are positive homomorphisms.

Now we verify the properties (i) and (ii) of Definition \ref{def:inductivelimit}.
\begin{enumerate}[(i)]
\item Let $f \in C_c(G^{(0)}, \mathbb{Z})$. Then
\begin{align*}
\alpha_{n+1}(i_n([f]_n))
= \alpha_{n+1}([f]_{n+1})
= [f]
= \alpha_n([f]).
\end{align*}

\item Let $B$ be an ordered group such that, for every $n$, there is a positive homomorphism $\lambda_n: H_0(G_n) \rightarrow B$ such that
\begin{align}
\label{eqn:lambdaninductive}
\lambda_n = \lambda_{n+1} \circ i_n.
\end{align}
\newc{Since $G^{(0)}$ is an open subgroupoid of $G_n$ for all $n$, we assume that $G_0 = G^{(0)}$, without loss of generality. Moreover, we can identify $H_0(G_0)$ with $C_c(G^{(0)}, \mathbb{Z})$.} Let $\lambda: H_0(G) \rightarrow B$ be defined by
\begin{align*}
\lambda([f]) = \lambda_0(f).
\end{align*}
Note that $\lambda$ is well-defined. Suppose $f_1, f_2 \in C_c(G^{(0)}, \mathbb{Z})$ are such that $f_1 \sim f_2$ in $H_0(G)$. Then there exists $F \in C_c(G, \mathbb{Z})$ such that $f_1 = f_2 + \partial_1 F$. There is $n$ such that $F \in C_c(G_n, \mathbb{Z})$. Then, by Lemma \ref{lemma:partial1F},
\begin{align*}
f_1 = f_2 + \partial_1^{(n)} F.
\end{align*}
Applying \eqref{eqn:lambdaninductive} multiple times, we have
\begin{align}
\lambda_0(f_1)
&= \lambda_1 \circ \oldc{i_1}\newc{i_0}(f_1) = \lambda_1([f_1]_1) \nonumber\\
&= \lambda_2 \circ \oldc{i_2}\newc{i_1}([f_1]_1) = \lambda_2([f_1]_2) \nonumber\\
&= \hdots \nonumber\\
&= \lambda_n([f_1]_n) \nonumber\\
&= \lambda_n(\oldc{[f_1]_n}\newc{[f_2]_n}) \hspace{20pt}\text{because $f_1 \sim f_2$ in $H_0(G_n)$}\nonumber\\
&= \lambda_0(f_2). \label{eqn:lambda0f1}
\end{align}
Thus $\lambda$ is well-defined.

We need to show that the diagram

\begin{equation}
\label{eqn:diagramH0Gn}
\begin{tikzcd}
& H_0(G_n) \arrow{dl}{}[swap]{\alpha_n} \arrow{dr}{\lambda_n}[swap]{} \\
H_0(G) \arrow{rr}{}[swap]{\lambda} && B
\end{tikzcd}
\end{equation}
commutes for all $n$. Let $f \in C_c(G^{(0)}, \mathbb{Z})$. Then, by the same argument used in \eqref{eqn:lambda0f1},
\begin{align*}
\lambda \circ \alpha_n([f]_n)
= \lambda([f])
= \lambda_0(f)
= \lambda_n([f]_n).
\end{align*}
Suppose that $\widetilde{\lambda}: H_0(G) \rightarrow B$ is another positive homomorphism making the diagram \eqref{eqn:diagramH0Gn} commutative for all $n$. Then, for $f \in C_c(G^{(0)}, \mathbb{Z})$,
\begin{align*}
\widetilde{\lambda}([f])
= \widetilde{\lambda}(\alpha_0(f))
= \lambda_0(f)
= \lambda([f]).
\end{align*}
Then $\widetilde{\lambda} = \lambda$. Therefore $H_0(G) =\displaystyle\lim_\rightarrow H_0(G_n)$. If each $H_0(G)$ is an ordered group, it follows from Proposition \ref{prop:inductivegroupunion} that $H_0(G)$ is an ordered group. \qedhere
\end{enumerate}
\end{proof}

Now we show that, for AF groupoids, the zeroth homology group is an ordered group.

\begin{proposition}
Let $G$ be an AF groupoid. Then $H_0(G)$ is an ordered group.
\end{proposition}
\begin{proof}
We begin by showing that, for a surjective local homeomorphism $\sigma: X \rightarrow \newc{Y}$ of locally compact, Hausdorff, second countable, totally disconnected spaces, the group $H_0(R(\sigma))$ is an ordered group. Indeed, it is straightforward from the definition that $H_0(R(\sigma)) = H_0(R(\sigma))^+ - H_0(R(\sigma))^+$. We now prove the intersection
\begin{align*}
H_0(R(\sigma))^+ \cap (-H_0(R(\sigma))^+) = \lbrace 0 \rbrace.
\end{align*}

Choose a function $f \in C_c(X, \mathbb{Z})$ with $[f] \in H_0(R(\sigma))^+ \cap (-H_0(R(\sigma))^+)$. Then there are two functions $f_1, f_2 \in C_c(X, \mathbb{Z})$ assuming non-negative values, such that
\begin{align*}
[f] = [f_1] = [-f_2].
\end{align*}
Choose $x \in X$. By Lemma \ref{lemma:traceH0Rsigma},
\begin{align*}
0
\leq f_1(x)
\leq \sum_{u: \sigma(u) = \sigma(x)} f_1(u)
= -  \sum_{u: \sigma(u) = \sigma(x)} f_2(u)
\leq - f_2(x)
\leq 0.
\end{align*}
Then $f_1(x) = f_2(x) = 0$. Since $x$ is arbitrary, we have $f_1 = f_2 = 0$. This implies that $H_0(R(\sigma))^+ \cap (-H_0(R(\sigma))^+) = \lbrace 0 \rbrace$. Therefore, $H_0(R(\sigma))$ is an ordered group.

Now let $G$ be an AF groupoid. Then there exists a sequence of surjective local homeomorphisms $\sigma_n: G^{(0)} \rightarrow Y_n$, where $G^{(0)}$ and $Y_n$ are locally compact, Hausdorff, second countable and totally disconnected, such that $G = \displaystyle\lim_\rightarrow R(\sigma_n)$. Since each $H_0(R(\sigma_n))$ is an ordered group, then Lemma \ref{lemma:limH0} implies that $H_0(G)$ is an ordered group.
\end{proof}

\par\nobreak\section{Homological similarity}
\label{subsection:homology:similarity}

We will define the notion of homolgical similarity between groupoids and show that, when two groupoids are homologically similar, then they have isomorphic homology groups. Here we also use Matui's paper \cite{Matui}.

\begin{definition}
\begin{enumerate}[(i)]
\item Two homomorphisms $\rho, \sigma$ from $G$ to $H$ are said to be \newterm{similar} if there exists a continuous map $\theta: G^{(0)} \rightarrow H$ such that
\begin{align}
\label{eqn:similar}
\theta(r(g)) \rho(g) = \sigma(g)\theta(s(g)),
\end{align}
for all $g \in G$.
\item $G$ and $H$ are said to be \newterm{homologically similar} if there exist homomorphisms $\rho: G \rightarrow H$ and $\sigma: H \rightarrow G$ that are local homeomorphisms and such that $\sigma \circ \rho$ is similar to $\mathrm{id}_G$ and $\rho \circ \sigma$ is similar to $\mathrm{id}_H$.
\end{enumerate}
\end{definition}

\begin{proposition}
\label{prop:similarhomomorphisms}
\cite[Proposition 3.5]{Matui}
Let $n \geq 1$. If $\rho, \sigma: G \rightarrow H$ are similar homomorphisms and local homeomorphisms, then $H_n(\rho) = H_n(\sigma)$.
\end{proposition}

\begin{corollary}
\label{corollary:isoHn}
\cite[Proposition 3.5]{Matui}
If $G$ and $H$ are homologically similar, then they have isomorphic homology groups. In particular, if $n \geq 1$, $\rho: G \rightarrow H$ and $\sigma: H \rightarrow G$ are local homeomorphisms and homomorphisms satisfying \eqref{eqn:similar}, then $H_n(\rho): H_n(G) \rightarrow H_n(H)$ is an isomorphism with inverse $H_n(\sigma)$. Moreover, both $H_0(\rho)$ and $H_0(\sigma)$ preserve positive elements.
\end{corollary}
\begin{proof}
It follows from \ref{prop:similarhomomorphisms} that  
\begin{align*}
H_n(\mathrm{id}_G) &= H_n(\sigma \circ \rho) = H_n(\sigma) \circ H_n(\rho), \text{ and}\\
H_n(\mathrm{id}_H) &= H_n(\rho \circ \sigma) = \oldc{H_n(\sigma)}\newc{H_n(\rho)} \circ H_n(\sigma).
\end{align*}
Then $H_n(\rho) = H_n(\sigma)^{-1}$. Therefore $H_n(\rho)$ is an isomorphism. It is straightforward from the definition of $H_0(\rho)$ and $H_0(\sigma)$ that these isomorphisms preserve positive elements.
\end{proof}

\chapter{K-theory for AF groupoids}
\label{section:K0H0}

In this chapter we show Theorem 2, that connects the K-theory with the homology group $H_0$ for AF groupoids. More precisely, we prove that, for an AF groupoid $G$, the homology group $H_0(G)$ and the $K_0$-group $K_0(C^*(G))$ are isomorphic as ordered groups. Our isomorphism generalises \cite[Corollary 5.2]{FKPS} by Farsi, Kumjian, Pask and Sims, where they prove a group isomorphism, but do not show that the map preserves positive maps from one another. Our theorem also generalises Theorem 4.10 of \cite{Matui} by Matui, where he considered AF groupoids with compact unit space. In this chapter we also provide an explicit formula for the isomorphism.

In comparison with the K-theory, homology groups of groupoids are easier to calculate and there \newc{are} useful tools in the theory to understand these homology groups in more detail. So, the isomorphism of Theorem 2 will be useful to study the main result of this thesis. Indeed, in we use this theorem to obtain the commutative diagram in Chapter \ref{section:diagram}, which is important to prove Theorem 1 in Chapter \ref{section:result}.

Although it is often difficult to find the K-theory of an arbitrary C*-algebra, this problem is simple when we consider that basic C*-algebra $M_n$. For this case, we get the isomorpism $K_0(M_n) \rightarrow \mathbb{Z}$ given by $[p]_0 \mapsto \tr p$, where $\tr p$ is the matrix trace of $p \in \mathcal{P}_\infty(M_n)$. Similarly, for a locally compact, Hausdorff, second countable, totally disconnected space $X$, the K-theory of the C*-algebra $C_0(X, M_n)$ is also described by the matrix trace. See Appendix \ref{appendix:C0XMn} for more details on the K-theory of $C_0(X, M_n)$. In this chapter, we apply the results for $C_0(X, M_n)$ to study the K-theory of the C*-algebras of AF groupoids.

An AF groupoid $G$ is an inductive limit of a sequence of groupoids $G_n$ such that $G_n \cong R(\sigma_n)$, where
\begin{align*}
\sigma_n: X_n \rightarrow Y_n
\end{align*}
are surjective local homeomorphisms, and the sets $X_n$, $Y_n$ have the same topological properties as $X$ mentioned in the previous paragraph. We study AF groupoids in Appendix \ref{appendix:AFgroupoids}, and inductive limits in Appendix \ref{appendix:inductivelimits}. Throughout this chapter, we always assume that $X, Y, X_n, Y_n$ are locally compact, Hausdorff, second countable, totally disconnected spaces.

In \oldc{most}\newc{the first} part of the chapter, we focus on the proof of the ordered group isomorphism $K_0(C^*(R(\sigma))) \cong H_0(R(\sigma))$, where $\sigma: X \rightarrow Y$ is a surjective local homeomorphism satisfying the following property: $X$ has a compact open subset $X_0$ such that
\begin{align}
\label{eqn:X0invariant}
\sigma^{-1}(\sigma(X_0)) = X_0
\hspace{15pt}
\text{and}
\hspace{15pt}
\text{$\sigma$ is injective on $X \setminus X_0$.}
\end{align}
In this case, we show that the C*-algebra $C^*(R(\sigma))$ is isomorphic to a direct sum of C*-algebras of the form $C_0(V_i, M_{n_i})$. This isomorphism and Appendix \ref{appendix:C0XMn} let us use the trace operator induced by the isomorphism to study the K-theory of $C^*(R(\sigma))$. Then we study this trace operator to prove that $K_0(C^*(R(\sigma))) \cong H_0(R(\sigma))$. We generalise this isomorphism to AF groupoids by applying inductive limits.

Throughout the chapter, we introduce the following notation. Given a continuous section\footnote{A map $\varphi: Y \rightarrow X$ is continuous section  of $\sigma: X \rightarrow Y$ if $y = \sigma(\varphi(y))$ for all $y \in Y$.} $\varphi: Y \rightarrow X$ of $\sigma: X \rightarrow Y$, we define the set $X_\varphi = \varphi(Y)$. Note that $X_\varphi$ is a clopen subset of $X$. In order to simplify our notation, for now we write $[f_{jk}]$ instead of $[f_{jk}]_{j,k=1, \dots, n}$ to denote a matrix of elements of $f_{jk}$ with  $j,k$ indexed from $1$ to \oldc{$k$}\newc{$n$}. 

\par\nobreak

\begin{remark}
\label{rmk:XvarphisetminusX0}
Note that \oldc{$\sigma^{-1}(\sigma(X_0))$}\newc{$\sigma^{-1}(\sigma(X_0)) = X_0$} implies that $\sigma^{-1}(\sigma(X \setminus X_0)) = X \setminus X_0$. We also have $X \setminus X_0 = X_\varphi \setminus X_0$. In fact, let $x \in X \setminus X_0$, and let $y = \varphi(\sigma(x))$. Then $y \in X_\varphi \setminus X_0$ and $\sigma(x) = \sigma(y)$. Since $\sigma$ is injective on $X \setminus X_0$, we have $x = y = \varphi(\sigma(x))$. But $x \in X \setminus X_0$ is arbitrary, therefore $X \setminus X_0 = X_\varphi \setminus X_0$.
\end{remark}

\begin{lemma}
\label{lemma:covmap-sets}
Let $\varphi: Y \rightarrow X$ be a continuous section of $\sigma$, and let $y \in Y$.Then $y$ has a clopen neighbourhood $W$, there is a positive integer $n$, and there are disjoint clopen subsets $\mathcal{U}_1, \dots, \mathcal{U}_n \subset X$ such that
\begin{enumerate}[(i)]
\item $\sigma^{-1}(W) = \bigcup_{i=1}^n \mathcal{U}_i$,
\item $\sigma\vert_{\mathcal{U}_i}: \mathcal{U}_i \rightarrow W$ is a homeomorphism for all $i$, and
\item $\sigma^{-1}(W) \cap X_\varphi = \mathcal{U}_1$.
\end{enumerate}
If $y \in \sigma(X_0)$, then $W$ and the sets $\mathcal{U}_i$ are compact. If $y \in Y \setminus \sigma(X_0)$, we let $n = 1$, $W = Y \setminus \sigma(X_0)$ and $\mathcal{U}_1 = X \setminus X_0$.
\end{lemma}
\begin{proof}
The proof of the case $y \in Y \setminus \sigma(X_0)$ follows from the fact that $\sigma$ is injective on $X \setminus X_0$.

Let $y \in \sigma(X_0)$. Since $\sigma$ is a local homeomorphism, $X_0$ is compact and $\sigma^{-1}(\sigma(X_0)) = X_0$, then $y$ has a compact open neighbourhood $W$ and there are finite disjoint compact open subsets $\mathcal{O}_1, \dots, \mathcal{O}_n$ with
\begin{align}
\label{eqn:sigmainvWOi}
\sigma^{-1}(W) = \bigcup_{i=1}^n \mathcal{O}_i,
\end{align}
and such that $\sigma\vert_{\mathcal{O}_i}: \mathcal{O}_i \rightarrow W$ is a homeomorphism for all $i$.

Now we define the sets $\mathcal{U}_i$. Let $\mathcal{U}_1 = \sigma^{-1}(W) \cap X_\varphi$. Since $\sigma$ is injective on $X_\varphi$, $\sigma$ is also injective on $\mathcal{U}_1$. Note that $\sigma(\mathcal{U}_1) = W$. In fact,
\begin{align*}
\sigma(\mathcal{U}_1)
= \sigma(\sigma^{-1}(W)) \cap \sigma(X_\varphi)
= W \cap Y
= W.
\end{align*}
Then $\sigma\vert_{\mathcal{U}_1}: \mathcal{U}_1 \rightarrow W$ is a homeomorphism.

We define the sets $\mathcal{U}_2, \dots, \mathcal{U}_n$ by manipulating the subsets $\mathcal{O}_1, \dots, \mathcal{O}_n$. Let $u \in \mathcal{U}_1$. By \eqref{eqn:sigmainvWOi}, Then there is a unique $k_u$ such that $u \in \mathcal{U}_1 \cap \mathcal{O}_{k_u}$. Moreover, there is a compact open set $\Omega_u$ with $u \in \Omega_u \subset \mathcal{U}_1 \cap \mathcal{O}_{k_u}$. By compactness of $\mathcal{U}_1$, there are $u_1, \dots, u_m \in \mathcal{U}_1$ such that $\mathcal{U}_1 = \Omega_{u_1} \cup \dots \cup \Omega_{u_m}$ and that each $\Omega_{u_i}$ is not a subset of $\bigcup_{\substack{j=1 \\ j \neq i}}^m \Omega_{u_j}$.

Let $\Omega_1 = \Omega_{u_1}$, and define for all $i = 2, \dots, m$ the set $\Omega_i = \Omega_{u_i} \setminus \bigcup_{j=1}^{i-1} \Omega_{u_j}$. Then we have the disjoint union
\begin{align}
\label{eqn:U1unionOmega}
\mathcal{U}_1 = \Omega_1 \cup \dots \cup \Omega_m.
\end{align}

For all $i=2, \dots, n$, and $j = 1,\dots, m$, let
\begin{align*}
\Omega_j^{(i)}
&= \begin{cases}
\mathcal{O}_{i-1} \cap \sigma^{-1}(\sigma(\Omega_j))
& \text{if } i < k_u, \\
\mathcal{O}_i \cap \sigma^{-1}(\sigma(\Omega_j))
& \text{if } i \geq k_{u_i}.
\end{cases}
\end{align*}
Note that $\Omega_j^{(1)} = \mathcal{O}_{k_{u_j}} \cap \sigma^{-1}(\sigma(\Omega_j))$.

Since the sets $\mathcal{O}_i$ are disjoint and compact open, and since the sets $\Omega_j$ are disjoint open, then the sets $\Omega_j^{(i)}$ are disjoint and compact open. Moreover, by \eqref{eqn:sigmainvWOi} and \eqref{eqn:U1unionOmega},
\begin{align}
\sigma^{-1}(W)
&= \left( \bigcup_{i=1}^n \mathcal{O}_i \right) \cap \sigma^{-1}(\sigma(\mathcal{U}_1)) \nonumber\\
&= \left( \bigcup_{i=1}^n \mathcal{O}_i \right) \cap \bigcup_{j=1}^m \sigma^{-1}(\sigma(\Omega_j)) \nonumber\\
&= \bigcup_{i=1}^n \bigcup_{j=1}^m \mathcal{O}_i \cap \sigma^{-1}(\sigma(\Omega_j)) \nonumber\\
&= \bigcup_{i=1}^n \bigcup_{j=1}^m \Omega_{j}^{(i)}. \label{eqn:unionOmegaij}
\end{align}

For all $i = 2, \dots, n$, let $\mathcal{U}_i = \Omega_i^{(1)} \cup \dots \cup \Omega_i^{(m)}$. Then property (i) follows from \eqref{eqn:unionOmegaij}.

We show (ii). Note that each $\Omega_j^{(i)}$ is a subset of either $\mathcal{O}_{i-1}$, $\mathcal{O}_{i}$ or $\mathcal{O}_{k_{u_i}}$. Since $\sigma$ is injective on the sets $\mathcal{O}_i$, then $\sigma$ is injective on all $\mathcal{U}_i = \Omega_1^{(i)} \cup \dots \cup \Omega_m^{(i)}$. Also,
\begin{align*}
\sigma(\mathcal{U}_i)
&= \sigma \left( \bigcup_{j=1}^m \mathcal{O}_i \cap \sigma^{-1}(\sigma(\Omega_j)) \right) \\
&= \sigma(\mathcal{O}_i) \cap \sigma\left(\bigcup_{j=1}^m \Omega_j \right) \\
&= \sigma(\mathcal{O}_i) \cap \sigma(\mathcal{U}_1) \\
&= W.
\end{align*}
Then property (ii) holds.

Recall that $\mathcal{U}_1 = \sigma^{-1}(W) \cap X_\varphi$. By \eqref{eqn:U1unionOmega}, $\Omega_j^{(i)} \cap X_\varphi \neq \emptyset$ if, and only if $i = 1$. Then (iii) holds by definition of the sets $\mathcal{U}_i$.
\end{proof}

\begin{remark}
\label{rmk:Rsigmaamenable}
Note that $R(\sigma)$ is amenable. Indeed, the C*-algebra $C^*(R(\sigma))$ is AF, it is also nuclear by Theorems 6.3.9 and 6.9.10 of Murphy's book \cite{Murphy}. Since \oldc{$C_r(R(\sigma))$}\newc{$C_r^*(R(\sigma))$} is a C*-subalgebra of $C^*(R(\sigma))$, then it is also nuclear. By \cite[Theorem 10.1.5]{sims2017hausdorff}, $R(\sigma)$ is amenable. Hence, the reduced and the full C*-algebras, $C_r^*(R(\sigma))$ and $C^*(R(\sigma))$, are the same by Corollary 6.2.14  of \cite{ADRenault}.
\end{remark}

\begin{proposition}
\label{prop:isoRsigmacompact}
Let $\varphi: Y \rightarrow X$ be a continuous section of $\sigma$. Then there is a positive integer $N$, there are disjoint open sets $V_1, \dots, V_N$ with $X_\varphi = V_1 \cup \dots \cup V_N$, there are positive integers $n_1, \dots, n_N$, and there is an isomorphism $\mu: C^*(R(\sigma)) \rightarrow \bigoplus_{i=1}^N C_0(V_i, M_{n_i})$ such that
\begin{enumerate}[(i)]
\item for $f \in \oldc{C_c(R(\sigma))}\newc{C^*(R(\sigma))}$ and $x \in X_\varphi$, we have
\begin{align*}
\mathrm{tr}\hphantom{.}\mu(f)(x)
= \sum_{u: \sigma(u) = \sigma(x)} f(u),
\end{align*}
\newc{where we apply Renualt's $j$ map \cite[Proposition II.4.2]{Renault} to identify $C^*(R(\sigma))$ as a subset of $C_0(R(\sigma))$.}
\item For a positive integer $k$, for $F \in M_k(C^*(R(\sigma)))$, and for $x \in X_\varphi$, we have
\begin{align*}
\mathrm{tr}\hphantom{.}\mu_k(F)(x)
= \sum_{i=1}^k \mathrm{tr}\hphantom{.}\mu(F_{ii})(x),
\end{align*}
where $\mu_k: M_k(C^*(R(\sigma))) \rightarrow \bigoplus_{i=1}^N \newc{C_0}(V_i, M_{kn_i})$ is the isomorphism given by
\begin{align*}
\mu_k(F)(x) = [\mu(F_{rs})(x)],
\hspace{15pt}
\text{for }
F = [F_{rs}] \in M_k(C^*(R(\sigma))), \text{ } x \in X_\varphi.
\end{align*}
\end{enumerate}
For the next items, we fix positive integers $k,l$, and choose projections $p \in \mathcal{P}_k(C^*(R(\sigma)))$, $q \in \mathcal{P}_l(C^*(R(\sigma)))$. Then
\begin{enumerate}[(i)]
\setcounter{enumi}{2}
\item $\mathrm{tr}\hphantom{.}\mu_k(p) \in C_c(X_\varphi, \mathbb{N})$.

\item $\mathrm{tr}\hphantom{.} \mu_{k+l}(p \oplus q) = \mathrm{tr}\hphantom{.} \mu_{k}(p) + \mathrm{tr}\hphantom{.} \mu_{l}(q)$.

\item $p \sim q
\Leftrightarrow
\mathrm{tr} \hphantom{.}\mu_k(p) = \mathrm{tr}\hphantom{.} \mu_{l}(q).$

\item $\mathrm{tr}\hphantom{.}\mu_k(p) = 0 \Rightarrow p = 0_k$.
\end{enumerate}
\end{proposition}
\begin{proof}
Let $y \in Y$. By Lemma \ref{lemma:covmap-sets}, $y$ has a clopen neighbourhood $\widetilde{W}_y$, there is a positive integer $n_y$, and there are $n_y$ disjoint clopen subsets $\widetilde{\mathcal{U}}_1^{(y)}, \widetilde{\mathcal{U}}_2^{(y)}, \dots, \widetilde{\mathcal{U}}_{n_y}^{(y)} \subset X$ such that
\begin{enumerate}[(a)]
\item $\sigma^{-1}(\widetilde{W}_y) = \bigcup_{j=1}^{n_y} \widetilde{\mathcal{U}}_j^{(y)}$,

\item $\sigma\vert_{\widetilde{\mathcal{U}}_j^{(y)}}: \widetilde{\mathcal{U}}_j^{(y)} \rightarrow \widetilde{W}_y$ is a homeomorphism for all $j$, and

\item $\sigma^{-1}(\widetilde{W}_y) \cap X_\varphi = \widetilde{\mathcal{U}}_1^{(y)}.$
\end{enumerate}

Since $\sigma(X_0)$ is compact and Hausdorff, we choose $N - 1$ elements $y_1, \dots, y_{N-1} \in \sigma(X_0)$ such that \oldc{$\sigma(X_0) = W_{y_1} \cup \dots \cup W_{y_{N-1}}$}\newc{$\sigma(X_0) = \widetilde{W}_{y_1} \cup \dots \cup \widetilde{W}_{y_{N-1}}$} and such that for all $i,j = 1, \dots, N-1$,
\begin{align}
\label{eqn:yiWj}
y_i \in \oldc{W_{y_j}}\newc{\widetilde{W}_{y_j}} \Leftrightarrow i = j.
\end{align}
Let $W_1 = \widetilde{W}_1$ and, for $i= 2, \dots, N-1$, define $W_i = \widetilde{W}_i \setminus \bigcup_{j=1}^{i-1} W_j$. It follows from property \eqref{eqn:yiWj} that the sets $W_i$ are non-empty. Moreover, $\sigma(X_0)$ is the disjoint union $\sigma(X_0) = W_1 \cup \dots \cup W_{N-1}$.

\begin{itemize}
\item If $\sigma(X_0) = Y$, we make an abuse of notation and replace $N-1$ by $N$,
\item Otherwise, let $y_N \in Y \setminus \sigma(X_0)$, and set $n_{N} = n_{y_N} = 1$, $W_{N} = Y \setminus \sigma(X_0)$, and $\oldc{\mathcal{U}_1^{(N)}}\newc{\mathcal{U}_1^{(y_N)}} = X \setminus X_0$. Then properties (a), (b) and (c) also hold by replacing $\widetilde{W}_y$, $n_y$ and $\widetilde{\mathcal{U}}_j^{(y)}$ by $W_N$, $n_N$ and $\mathcal{U}_1^{(y_n)}$, respectively.
\end{itemize}

In either case, we have the disjoint union
\begin{align*}
Y = \sigma(X_0) \cup (Y \setminus \sigma(X_0)) = W_1 \cup \dots \cup W_N.
\end{align*}

Define $n_i = n_{y_i}$ and $V_i = \mathcal{U}_1^{(y_i)}$ for all $i = 1, \dots, N$. Note that $V_1, \dots, V_N \subset X_\varphi$ by property (c). Since $\sigma$ is injective on $X_\varphi$, and since $W_1, \dots, W_N$ are disjoint, then $V_1, \dots, V_N$ are disjoint subsets. Moreover,
\begin{align*}
X_\varphi
= \sigma^{-1}(Y) \cap X_\varphi
= \sigma^{-1} \left( \bigcup_{i=1}^N W_i \right) \cap X_\varphi
= \bigcup_{i=1}^N \sigma^{-1}(W_i) \cap X_\varphi
= \bigcup_{i=1}^N V_i.
\end{align*}
Now we use the sets above to define the map \oldc{$\mu_c$}\newc{$\mu$} from \oldc{$C_c(R(\sigma))$ to  $\bigoplus_{i=1}^N C_c(V_i, M_{n_i})$} \newc{$C^*(R(\sigma))$ to  $\bigoplus_{i=1}^N C_0(V_i, M_{n_i})$} by
\begin{align}
\label{eqn:muCRsigma}
\oldc{\mu_c}\newc{\mu}(f)(x)_{rs} = f(\sigma\vert_{\mathcal{U}_r^{(y_i)}}^{-1}(\sigma(x)), \sigma\vert_{\mathcal{U}_s^{(y_i)}}^{-1}(\sigma(x))),
\end{align}
for $x \in V_i$, $r, s = 1, \dots, n_i$, $i =1, \dots, n_i$, $f \in C^*(R(\sigma))$.

It is straightforward that \oldc{$\mu_c$}\newc{$\mu$} is well-defined. We show that \oldc{$\mu_c$}\newc{$\mu$} is an $\ast$-isomorphism. \oldc{ that extends continuously to an isomorphism $\mu: C^*(R(\sigma)) \rightarrow \bigoplus_{i=1}^N C_0(V_i, M_{n_i})$.}
\begin{itemize}

\item \oldc{$\mu_c$}\newc{$\mu$} is a homomorphism

In order to simplify our notation, given $i = 1,\dots, N$, $x \in \sigma^{-1}(W_i)$, and $r = 1, \dots, n_i$, we denote $x_r = \sigma\vert_{\mathcal{U}_r^{(y_i)}}^{-1}(\sigma(x))$. Note that
\begin{align*}
\newc{\sigma^{-1}(\lbrace \sigma(x) \rbrace)} = \lbrace x_1, \dots, x_{n_i} \rbrace.
\end{align*}

Let $f, f_1, f_2 \in \oldc{C(R(\sigma))}\newc{C^*(R(\sigma))}$ and let $x \in V_i$ with $i = 1, \dots, N$. Then, for $r, s = 1, \dots, n_i$,
\begin{align*}
\oldc{\mu_c}\newc{\mu}(f^*)(x)_{rs}
= f^*(x_r, x_s) = \overline{f(x_s, x_r)} 
= \oldc{\mu_c(f)_{rs}^*} \newc{(\mu(f)(x)^*)_{rs}}.
\end{align*}
By properties (a) and (b), we have that an element $u \in X$ satisfies $\sigma(u) = \sigma(x)$ if, and only if, $u = x_j$ for some $j = 1, \dots, n_i$. Then
\begin{align*}
\oldc{\mu_c}\newc{\mu}(f_1 \cdot f_2)(x)_{rs}
&= f_1 \cdot f_2 (x_r, x_s) \\
&= \sum_{u: \sigma(u) = \sigma(x)} f_1(x_r, u) f_2(u, x_s) \\
&= \sum_{t = 1}^{n_i} f_1(x_r, x_t) f_2(x_t, x_s) \\
&= (\oldc{\mu_c}\newc{\mu}(f_1) \mu(f_2))_{rs}.
\end{align*}
Thus \oldc{$\mu_c$}\newc{$\mu$} is a homomorphism.

\item \oldc{$\mu_c$}\newc{$\mu$} is an isometry

By \cite[Proposition 9.3.1]{sims2017hausdorff}, for all $x \in X$, the $\ast$-representation $\pi_x: C_c(R(\sigma)) \rightarrow B(\ell^2(R(\sigma)_x))$ is given by
\begin{align*}
\pi_x(f)\delta_{(v,x)} = \sum_{(u,v) \in R(\sigma)} f(u,v) \delta_{(u,x)}
= \sum_{u: \sigma(u) = \sigma(x)} f(u,v) \delta_{(u,x)},
\end{align*}
for $f \in C_c(R(\sigma))$ and $v \in X$ with $\sigma(v) = \sigma(x)$.

Let $x \in X_\varphi$. Then there exists a unique $i = 1, \dots, N$ such that $x \in V_i$. To simplify our notation, we let $x_r = \sigma_{\mathcal{U}_r^{(y_i)}}^{-1}(\sigma(x))$ for all $r = 1, \dots, n_i$. Note that $\sigma^{-1}(\sigma(x)) = \lbrace x_1, \dots, x_{n_i} \rbrace$.

Let $T_x: \mathbb{C}^{n_i} \rightarrow \ell^2(R(\sigma)_x)$ be the isometry defined by $T_x e_r = \delta_{(x_r, x)}$ for $r = 1,\dots,n_i$. Then, for all $s = 1, \dots, n_i$, we have
\oldc{
\begin{align*}
\pi_x(f) T_x e_s
&= \pi_x(f) \delta_{(x_s, x)} \\
&= \sum_{u: \sigma(u) = \sigma(x)} f(u, x_s) \delta_{(x_s, x)} \\
&= \sum_{r=1}^{n_i} f(x_r, x_s) \delta_{(x_s, x)} \\
&= T_x^* \sum_{r=1}^{n_1} f(x_r, x_s) e_s \\
&=T_x^* \mu(f)(x) e_s.
\end{align*}
}
\newc{
\begin{align*}
\pi_x(f) T_x e_s
&= \pi_x(f) \delta_{(x_s, x)} \\
&= \sum_{u: \sigma(u) = \sigma(x)} f(u, x_s) \delta_{(u, x)} \\
&= \sum_{r=1}^{n_i} f(x_r, x_s) \delta_{(x_r, x)} \\
&= T_x \sum_{r=1}^{n_1} f(x_r, x_s) e_r \\
&=T_x \mu(f)(x) e_s.
\end{align*}
}
This implies that $\Vert \pi_x(f) \Vert = \Vert \mu(f)(x) \Vert$ for all $x \in X_\varphi$.

Given $x \in X$, let $u = \varphi(\sigma(x)) \in X_\varphi$. \cite[Proposition 9.3.1]{sims2017hausdorff} implies that there exists a unitary $U_{(x, u)}: \ell^2(R(\sigma)_u) \rightarrow  \ell^2(R(\sigma)_x)$ such that $\pi_x = U_{(x,u)} \pi_u U_{(x,u)}^*$. Then $\Vert \pi_x(f) \Vert = \Vert \pi_u(f) \Vert$. Therefore,
\begin{align*}
\Vert f \Vert
= \sup_{x \in X} \Vert \pi_x(f) \Vert
= \sup_{x \in X_\varphi} \Vert \pi_x(f) \Vert
= \sup_{x \in X_\varphi} \Vert \mu(f)(x) \Vert
= \Vert \oldc{\mu_c}\newc{\mu}(f) \Vert.
\end{align*}
Therefore \newc{the restriction of $\mu$ to $C_c(R(\sigma))$} is an isometry. \newc{By continuity, $\mu$ is an isometry.}

\item \newc{$\mu$ is an isomoprism}\oldc{$\mu_c$ extends to the isomorphism $\mu: C^*(R(\sigma)) \rightarrow \bigoplus_{i=1}^N C_0(V_i, M_{n_i})$}

\newc{Since $\mu$ is an isometry, it follows that it is injective.}
\oldc{Since $\mu_c$ is a homomorphism and an isomotry, then by continuity it extends uniquely to a homomorphism $\mu: C^*(R(\sigma)) \rightarrow \bigoplus_{i=1}^N C_0(V_i, M_{n_i})$, where $\mu$ is also an isometry, and thus injective.}
We show that \oldc{$\mu_c$}\newc{$\mu$} is surjective. \oldc{Then, by continuity, we will also have that $\mu$ is surjective.}

Let $F \in \bigoplus_{i=1}^N C_c(V_i, M_{n_i})$. Define $f \in C_c(R(\sigma))$ by
\begin{align*}
f(x_r, x_s)
= F(x)_{rs}
\hspace{15pt}
\text{for $x \in V_i$, $r,s = 1, \dots, n_i$.}
\end{align*}
Note that $f$ is continuous because it is the composition of continuous functions. Let $K \subset X_\varphi$ by a compact subset containing the support of $F$, then the support of $f$ is a subset of $[\sigma^{-1}(\sigma(K)) \times \sigma^{-1}(\sigma(K))] \cap R(\sigma)$, which is compact. This implies that $f \in C_c(R(\sigma))$. Moreover, $\oldc{\mu_c}\newc{\mu}(f) = F$. \newc{So far we have that the restriction}
\newc{
\begin{align*}
\mu\vert_{C_c(R(\sigma))}: C_c(R(\sigma)) \rightarrow \bigoplus_{i=1}^N C_c(V_i, M_{n_i})
\end{align*}
is surjective. Since $\mu$ is an isometry and since $\bigoplus_{i=1}^N C_c(V_i, M_{n_i})$ is dense in $\bigoplus_{i=1}^N C_0(V_i, M_{n_i})$, then $\mu$ is surjective.
}
\end{itemize}

Now we show property (i). Let $f \in \oldc{C^*(R(\sigma))}\newc{C_c(R(\sigma))}$, and $x \in V_i$ with $i = 1, \dots, N$. Since $\sigma^{-1}(x) = \lbrace x_1, \dots, x_{n_i} \rbrace$, then
\begin{align*}
\mathrm{tr}\hphantom{.}\mu(f)\newc{(x)}
= \sum_{l=1}^{n_i} f(x_l, x_l)
= \sum_{u: \sigma(u) = \sigma(x)} f(u).
\end{align*}

Now we prove property (ii). Let $k$ be a positive integer, and let $F = [F_{rs}] \in M_k(C^*(R(\sigma)))$, and fix $x \in V_i$ with $i = 1, \dots, N$. Then
\begin{align*}
\mathrm{tr}\hphantom{.}\mu_k(F)(x)
&= \sum_{q = 1}^{k n_i} \mu_k(F)_{qq} \\
&= \sum_{j= 0}^{k-1} \sum_{l=1}^{n_i} \mu_k (F)_{n_ij + l, n_ij + l} \\
&= \sum_{i=1}^k \sum_{l=1}^{n_i} \mu(F_{ii})_{ll}.
\end{align*}
Thus property (ii) holds.

Next, we prove item (iii). Let $k$ be a positive integer, and fix $p \in \mathcal{P}_k(C^*(R(\sigma)))$. Note that $C^*(R(\sigma)) = C(R(\sigma))$, so it follows that, for all $i = 1, \dots, k$, $p_{ii} \in C(R(\sigma))$. given $x \in X_\varphi$, items (i) and (ii) imply that
\begin{align*}
\mathrm{tr}\hphantom{.}\mu_k(p)(x)
= \sum_{i=1}^k \mathrm{tr}\hphantom{.}\mu(p_{ii})(x)
= \sum_{i=1}^k \sum_{u: \sigma(u) = \sigma(x)} p_{ii}(u).
\end{align*}
The equation above implies that $\mathrm{tr}\hphantom{.}\mu_k(p)$ is a composition of continuous functions, then $\mathrm{tr}\hphantom{.}\mu_k(p)$ is a continuous function.

We need to check that the image of $\mathrm{tr}\hphantom{.}\mu_k(p)$ lies in $\mathbb{N}$. Indeed, let $x \in X_\varphi$. Then there a $j \in \lbrace 1, \dots, N \rbrace$ such that $x \in V_j$. Then $\mu_k(p)(x)$ is a projection in $M_{kn_i}$. Hence, $\mathrm{tr}\hphantom{.}\mu_k(p)(x) \in \mathbb{N}$. Therefore, $\mathrm{tr}\hphantom{.}\mu_k(p)$ \oldc{is continuous}\newc{takes values in $\mathbb{N}$}.

Now we show that $\mathrm{tr}\hphantom{.}\mu_k(p)$ is compactly supported. \newc{By Remark \ref{rmk:projectioncompactlysupported}, $\mu_k(p)$ is compactly supported. Therefore $\mathrm{tr}\hphantom{.}\mu_k(p) \in C_c(X, \mathbb{N})$.}

\oldc{Since the projections $p_{11},$ $\dots,$ $p_{kk}$ are in $C_0(R(\sigma))$, there exists a compact subset $K \subset R(\sigma)$ with $X_0 \subset K$ and such that
\begin{align*}
\vert p_{ii}(x) \vert < \frac{1}{k}
\hspace{10pt}\text{for }
x \in X_0 \setminus K
\text{ and }
i = 1, \dots, k.
\end{align*}
Given $x \in X_0 \setminus K$, we have $\sigma^{-1}(\sigma(x)) = \lbrace x \rbrace$. Then
\begin{align*}
\vert \mathrm{tr}\hphantom{.}\mu_k(p)(x) \vert
= \left\vert \sum_{i=1}^k \sum_{u: \sigma(u) = \sigma(x)} p_{ii}(u) \right\vert
= \left\vert \sum_{i=1}^k p_{ii}(x) \right\vert
\leq \sum_{i=1}^k \vert p_{ii}(x) \vert
< 1.
\end{align*}
Since $\mathrm{tr}\hphantom{.}\mu_k(p)(x) \in \mathbb{N}$, we have $\mathrm{tr}\hphantom{.}\mu_k(p)(x) = 0$. Therefore $\mathrm{tr}\hphantom{.}\mu_k(p) \in C_c(X, \mathbb{N})$.}

Now we show item (iv). Let $x \in X_\varphi$. By definition of $\mu$,
\begin{align*}
\mu_{k+l}(p \oplus q)(x)
= \mu_k(p)(x) \oplus \mu_l(q)(x).
\end{align*}
Then item (ii) implies that
\begin{align*}
\mathrm{tr}\hphantom{.}\mu(p\oplus q)(x)
= 
\mathrm{tr}\hphantom{.}\mu_k(p)(x) + \mathrm{tr}\hphantom{.}\mu_k(q)(x).
\end{align*}

Now we show (v). Suppose $p \sim q$. Then $p \oplus 0_l,$ $q \oplus 0_k \in \mathcal{P}_l(C^*(R(\sigma)))$ are equivalent.

Given $x \in X_\varphi$, then $\mu_{k+l}(p \oplus 0_l), \mu_{k+l}(q \oplus 0_k)$ are equivalent in $M_{k+l}$. This  fact and item (iv) imply that
\begin{align*}
\mathrm{tr}\hphantom{.}\mu_k(p)(x)
= \mathrm{tr}\hphantom{.}\mu_{k+l}(p \oplus 0_l)(x)
= \mathrm{tr}\hphantom{.}\mu_{k+l}(q \oplus 0_k)(x)
= \mathrm{tr}\hphantom{.}\mu_l(q)(x).
\end{align*}

Conversely, assume that $\mathrm{tr}\hphantom{.}\mu_k(p) = \mathrm{tr}\hphantom{.} \mu_l(q)$. By item (iv), $\mathrm{tr}\hphantom{.}\mu_{k+l}(p \oplus 0_l) = \mathrm{tr}\hphantom{.} \mu_{k+l}(q \oplus 0_k)$. Hence, by Proposition \ref{prop:traceCXMn}, $\mu_{k+l}(p \oplus 0_l) \sim \mu_{k+l}(q \oplus_k)$. Since $\mu_{k+l}$ is an isomorphism, we have $p \oplus 0_l \sim q \oplus 0_k$ and then $p \sim q$.

Finally, we prove item (vi). Suppose $\mathrm{tr}\hphantom{.}\mu_k(p) = 0$. Let $x \in X_\varphi$. Then, by Remark \ref{rmk:dimp}
\begin{align*}
0 = \mathrm{tr}\hphantom{.}\mu_k(p)(x)
= \dim \mu_k(p(x)) \mathbb{C}^k.
\end{align*}
So $\mu_k(p(x)) = 0$. Since $x$ is arbitrary and $\mu_k$ is an isomorphism, we have $p = 0$.
\end{proof}

We use the isomorphism $\mu$ to define a trace operator on $\mathcal{P}_\infty(C^*(R(\sigma)))$ that characterises the $K_0$-group of $C^*(R(\sigma))$. Given a continuous section $\varphi: Y \rightarrow X$ of $\sigma$, we define $\mathrm{tr}_\varphi: \mathcal{P}_\infty(C^*(R(\sigma))) \rightarrow C(X, \mathbb{N})$ by
\begin{align*}
\mathrm{tr}_\varphi \hphantom{.} p(x)
= \begin{cases}
\mathrm{tr}\hphantom{.}\mu_k(p)(x) & \text{if } x \in X_\varphi, \\
0 & \text{otherwise,}
\end{cases}
\end{align*}
for $p \in \mathcal{P}_k(C^*(R(\sigma)))$, $k \geq 1$, and $x \in X$.

Note that $\mathrm{tr}_\varphi \hphantom{.} p$ is the extension of $\mathrm{tr}\hphantom{.}\mu_k(p)$ to a larger domain $X \supset X_\varphi$. We do this because later we define the isomorphism from $K_0(C^*(R(\sigma)))$ to $H_0(R(\sigma))$ by $[p]_0 \mapsto [\mathrm{tr}_\varphi\hphantom{.}p]$ for $p \in \mathcal{P}_\infty(C^*(R(\sigma)))$. In order to study the equivalence class of $\mathrm{tr}_\varphi \hphantom{.} p$ in the homology group of $H_0(R(\sigma))$, the function $\mathrm{tr}_\varphi \hphantom{.} p$ needs to be defined in the unit space $X$ of $R(\sigma)$.

\oldc{
\begin{remark}
\label{rmk:trvarphi}
By applying the $j$ map from \cite[Proposition II.4.2]{Renault} (see also \cite[Proposition 9.3.3]{sims2017hausdorff}), we assume that $C^*(R(\sigma))$ is a subset of $C_0(R(\sigma))$. By Proposition \ref{prop:isoRsigmacompact}, we have
\begin{align*}
\mathrm{tr}_\varphi\hphantom{.}p(x) = 1_{X_\varphi}(x) \sum_{u: \sigma(u) = \sigma(x)} p(u)
\end{align*}
for $u \in X$ and \oldc{$p \in \mathcal{P}(C^*(R(\sigma)))$}\newc{$p \in C^*(R(\sigma))$}. Moreover,
\begin{align*}
\mathrm{tr}_\varphi\hphantom{.}p = \mathrm{tr}_\varphi\hphantom{.}p_{11} + \dots + \mathrm{tr}_\varphi\hphantom{.}p_{nn}
\end{align*}
for $p \in \mathcal{P}_n(C^*(R(\sigma)))$, $n \geq 1$. It follows from Proposition \ref{prop:isoRsigmacompact} that $\mathrm{tr}_\varphi\hphantom{.} p \in C_c(X, \mathbb{N})$.
\end{remark}}

The following lemma shows the connection of the trace operator $\mathrm{tr}_\varphi$ with the homology group $H_0(R(\sigma))$.

\begin{lemma}
\label{lemma:traceH0}
Let $\varphi, \psi$ be two continuous sections of $\sigma$. Then
\begin{enumerate}[(i)]
\item Given $p \in \mathcal{P}_\infty(C^*(R(\sigma)))$, $\mathrm{tr}_\varphi\hphantom{.}p \sim \mathrm{tr}_\psi\hphantom{.}p$.
\item For a compact open set $V \subset X$, we have $1_V \sim \mathrm{tr}_\varphi 1_V$.
\end{enumerate}
\end{lemma}
\begin{proof}
We begin by proving (i). Fix $p \in \mathcal{P}_\infty(C^*(R(\sigma)))$, then there exists a positive integer $k$ with $p \in \mathcal{P}_k(C^*(R(\sigma)))$.

Given $x \in X$, set $x_\varphi = \varphi(\sigma(x))$ and $x_\psi = \psi(\sigma(x))$. Then $x_\varphi$ (resp. $x_\psi$) is the only element in $X_\varphi$ (resp. $X_\psi$) with $\sigma(x_\varphi) = \sigma(x_\psi) = \sigma(x)$.

By definition of $\mathrm{tr}_\varphi\hphantom{.}p$ and by Proposition \ref{prop:isoRsigmacompact}, we have
\begin{align*}
\sum_{u: \sigma(u) = \sigma(x)} \mathrm{tr}_\varphi \hphantom{.} p(u)
&= \mathrm{tr}_\varphi \hphantom{.} p(x_\varphi) \\
&= \sum_{i=1}^k \mathrm{tr}_\varphi \hphantom{.} p_{ii} (x_\varphi) \\
&= \sum_{i=1}^k \sum_{v: \sigma(v) = \sigma(x_\varphi)} p_{ii}(v) \\
&= \sum_{i=1}^k \sum_{v: \sigma(v) = \sigma(x_\psi)} p_{ii}(v) \\
&= \mathrm{tr}_\psi\hphantom{.} p(x_\psi) \\
&= \sum_{u: \sigma(u) = \sigma(x)} \mathrm{tr}_\psi \hphantom{.} p(u).
\end{align*}
Since $x$ is arbitrary, it follows from Lemma \ref{lemma:traceH0Rsigma} that $\mathrm{tr}_\varphi\hphantom{.}p \sim \mathrm{tr}_\psi\hphantom{.}p$.

Now we prove item (ii). Let $V \subset X$ a compact open subset, let $x \in X$, and set $x_\varphi = \varphi(\sigma(x))$. Then
\begin{align*}
\sum_{u: \sigma(u) = \sigma(x)}\mathrm{tr}_\varphi\hphantom{.}1_V(u)
&= \mathrm{tr}_\varphi \hphantom{.} 1_V (x_\varphi) \\
&= \sum_{u: \sigma(u) = \sigma(x_\varphi)} 1_V(u) \\
&= \sum_{u: \sigma(u) = \sigma(x)} 1_V(u).
\end{align*}
Since $x$ is arbitrary, Lemma \ref{lemma:traceH0Rsigma} implies that $\mathrm{tr}_\varphi \hphantom{.} 1_V \sim 1_V$.
\end{proof}

Now we give an alternative proof of the isomorphism $\oldc{K_0(C^*(R(\sigma)))}\newc{K_0(C^*(G))} = H_0(G)$ for AF groupoids $G$ given by Farsi, Kumjian, Pask and Sims \cite[Corollary 5.2]{FKPS}.

The following theorem gives an isomorphism of the ordered groups $K_0(C^*(G))$  and $H_0(G)$. Corollary 5.2 by Farsi, Kumjian, Pask and Sims \cite{FKPS} and Matui's \cite[Theorem 4.10]{Matui} prove isomorphisms of homology groups of groupoids and the corresponding $K_0$-group. However, \cite[Corollary 5.2]{FKPS} does not prove whether such isomorphism preseves positive elements from one group to another, and Matui only considers groupoids with compact unit spaces, and does not give a detailed proof. Our theorem generalises both results, and we describe the isomorphism by an explicit formula.

\begin{customthm}{2}
\label{thm:K0H0}
Let $G$ be an AF groupoid, written as the inductive limit $G = \displaystyle\lim_\rightarrow G_n$ of elementary groupoids. Then, for every $n$, there exists \newc{a locally compact, Hausdorff, second countable, totally disconntect space $Y_n$, and there is a} surjective local homeomorphism $\sigma_n: G_n^{(0)} \rightarrow Y_n$ such that \oldc{$G_n = R(\sigma)$}\newc{$G_n = R(\sigma_n)$}. Moreover,
\begin{enumerate}
\item there exists a unique isomorphism $\nu: K_0(C^*(G)) \rightarrow H_0(G)$ defined by
\begin{align*}
\nu([p]_0) = [\mathrm{tr}_{\varphi_n}(p)]
\end{align*}
for $p \in \mathcal{P}_n(C^*(G_n))$, and such that $\varphi_n$ is an arbitrary continuous section of $\sigma_n$;
\item the map $\nu$ is an isomorphism of ordered groups, i.e., $\nu(K_0(C^*(G))^+) = H_0(G)^+$; and
\item the isomorphism $\nu$ is precisely the map given by \cite[Corollary 5.2]{FKPS}. In other words, $\nu([1_V]_0) = [1_V]$ for all $V \subset G^{(0)}$ compact open.
\end{enumerate}
\end{customthm}

We divide the proof into three parts:
\begin{enumerate}[{Part} 1]
\item First we study the groupoid $R(\sigma)$ where $\sigma$ satisfies the conditions of the beginning of the chapter, i.e., that there exists a compact subset $X_0$ such that $\sigma^{-1}(\sigma(X_0)) =X_0$ and that $\sigma$ is injective on $X \setminus X_0$.
\item Then we generalise the result to a groupoid of the form $R(\sigma)$, where $X$ no longer need to necessarily satisfy the conditions of Part 1.
\item Finally, we prove the isomorphism for an AF groupoid $G$.
\end{enumerate}

\begin{proof}[Proof of Part 1]\renewcommand{\qedsymbol}{}
Fix a continuous section $\varphi$ of $\sigma$. Given $p \in \mathcal{P}_\infty(C^*(R(\sigma)))$, let
\begin{align}
\label{eqn:trvarphip}
\nu([p]_0) = [\mathrm{tr}_\varphi \hphantom{.}p].
\end{align}
It follows from item (v) of Proposition \ref{prop:isoRsigmacompact} that $\nu([p]_0)$ is well-defined. Moreover, item (iv) of the same proposition implies that for \newc{a projection} $q \in \oldc{\mathcal{P}_\infty}\newc{\mathcal{P}_\infty(C^*(R(\sigma)))}$,
\begin{align*}
\nu([p]_0+[q]_0)
&= \nu([p\oplus q]_0) \\
&= [\mathrm{tr}_\varphi(p \oplus q)] \\
&= [\mathrm{tr}_\varphi(p) + \mathrm{tr}_\varphi(q)] \\
&= [\mathrm{tr}_\varphi(p)]+[\mathrm{tr}_\varphi(q)] \\
&= \nu([p]_0) + \nu([q]_0).
\end{align*}
Then we extend \eqref{eqn:trvarphip} uniquely to a homomorphism $\nu: K_0(C^*(R(\sigma))) \rightarrow H_0(R(\sigma))$.

We claim that $\nu$ is an isomorphism. \newc{Let $x \in K_0(C^*(R(\sigma)))$ be such that $\nu(x) = 0$. Then there are $p, q \in \mathcal{P}_\infty(C^*(R(\sigma)))$ such that \begin{align*}
x = [p]_0 - [q]_0.
\end{align*}
Then 
\begin{align*}
 [\mathrm{tr}_\varphi \hphantom{.} p] = \nu([p]_0) = \nu([q]_0) = [\mathrm{tr}_\varphi \hphantom{.} q],
\end{align*}
by definition of the homomorphism $\nu$. It follows from Lemma \ref{lemma:traceH0Rsigma} that, for every $x \in X$, we have
\begin{align*}
\sum_{u:\sigma(u) = \sigma(x)} \mathrm{tr}_\varphi \hphantom{.} p(u) = \sum_{u:\sigma(u) = \sigma(x)} \mathrm{tr}_\varphi \hphantom{.} q(u).
\end{align*}
Since $\mathrm{tr}_\varphi \hphantom{.} p$ and $\mathrm{tr}_\varphi \hphantom{.} q$ have support in $X_\varphi$, this is equivalent to the following
\begin{align*}
\mathrm{tr}_\varphi \hphantom{.} p(\varphi(\sigma(x))) = \mathrm{tr}_\varphi \hphantom{.} q(\varphi(\sigma(x))).
\end{align*}
This implies that $\mathrm{tr}_\varphi \hphantom{.} p = \mathrm{tr}_\varphi \hphantom{.} q$. It follows from item (v) of Proposition \ref{prop:isoRsigmacompact} that $p \sim q$. Hence, $x = 0$ and therefore $\nu$ is injective.
}

\oldc{It follows from item (vi) of Proposition \ref{prop:isoRsigmacompact} that $\nu$ is injective.} We show that $\nu$ is also surjective. Given a compact open subset $V \subset X$, Lemma \ref{lemma:traceH0}  implies that
\begin{align*}
[1_V] = [\mathrm{tr}_\varphi\hphantom{.} 1_V] = \nu([1_V]_0).
\end{align*}
Since the elements of the form $[1_V]$ generate $H_0(R(\sigma))$, it follows that $\nu$ is an isomorphism.

Now we show that $\nu$ is a positive isomorphism. Equation \eqref{eqn:trvarphip} and item (iii) of Proposition \ref{prop:isoRsigmacompact} imply that $\nu(K_0(C^*(R(\sigma))^+) \subset H_0(R(\sigma))^+$. Let $f \in C_c(X, \mathbb{N})$. Then there are $V_1, \dots, V_n \subset X$ compact open sets, and $a_1, \dots, a_n$ non-negative integers such that $f = a_1 1_{V_1} + \dots + a_n 1_{V_n}$. Then
\begin{align*}
[f]
&= \sum_{i=1}^n a_i [1_{V_i}] \\
&= \sum_{i=1}^n a_i \nu([1_{V_i}]_0) \\
&= \nu \left(\sum_{i=1}^n a_i [1_{V_i}]_0 \right).
\end{align*}
Then $\nu(K_0(C^*(R(\sigma)))^+) = H_0(R(\sigma))^+$ and, therefore $\nu$ is a positive isomorphism. This completes the proof of Part 1.
\end{proof}

We need the following lemma before proving Part 2. Now we no longer assume that $\sigma$ satisfies that conditions of Part 1.

\begin{lemma}
\label{lemma:tracevarphi}
Let $\sigma: X \rightarrow Y$ be a surjective local homeomorphism with continuous section $\varphi: Y \rightarrow X$. Let $X = \bigcup_{k=1}^\infty X_k$ be an increasing union of compact open sets \newc{such that $\varphi(\sigma(X_k)) \subset X_k$}. For each $k$, define the map $\sigma_k : X \rightarrow \sigma(X_k) \sqcup (X \setminus X_k)$ by
\begin{align*}
\sigma_k(x) =
\begin{cases}
\sigma(x) & \text{if }x \in X_k, \\
x & \text{otherwise.}
\end{cases}
\end{align*}
Define also the map $\varphi_k: \sigma(X_k) \sqcup (X \setminus X_k) \rightarrow X$ by
\begin{align*}
\varphi_k(y) =
\begin{cases}
\varphi(y) & \text{if } \oldc{x}\newc{y}\in \sigma(X_k), \\
y & \text{otherwise.}
\end{cases}
\end{align*}
Then
\begin{enumerate}[(i)]
\item $\sigma_k$ is a surjective local homeomorphism with continuous section $\varphi_k$.
\item $R(\sigma)$ is an inductive limit of the form $R(\sigma) = \displaystyle\lim_\rightarrow R(\sigma_k)$.
\item \newc{For $k \geq 1$, $p \in \mathcal{P}_\infty(C^*(R(\sigma_k))$, we have $\mathrm{tr}_\varphi \hphantom{.} p = \mathrm{tr}_{\varphi_l} \hphantom{.} p$ for all $l \geq k$, where $\mathrm{tr}_\varphi \hphantom{.} p$ is given by
\begin{align*}
\mathrm{tr}_\varphi \hphantom{.} p (x) = 1_{X_\varphi}(x) \sum_{i=1}^n \sum_{u : \sigma(u) = \sigma(x)}  p_{ii}(u),
\hspace{15pt}
\text{for $x \in X$.}
\end{align*}
}
\end{enumerate}
\end{lemma}
\begin{proof}
Since $\sigma$ is a surjective local homeomorphism \newc{and $X_k$ is open}, then the restriction $\sigma_k\vert_{X_k}: X_k \rightarrow \sigma(X_k)$ is a surjective local homeormorphism. The restriction $\sigma_k$ on $X \setminus X_k$ is the identity. Thus, $\sigma_k$ is a surjective local homeomorphism.

\newc{It is straightforward from its definition that $\varphi_k$ is continuous. We show that $\varphi_k$ is a continuous section of $\sigma_k$. Let $y \in \sigma(X_k)$. Then $\varphi(y) \in X_k$ by hypothesis. It follows that
\begin{align*}
\sigma_k(\varphi_k(y))
= \sigma_k(\varphi(y))
= \sigma(\varphi(y))
= y.
\end{align*}
Now let $x \in X \setminus X_k$. Then
\begin{align*}
\sigma_k(\varphi_k(x))
= \sigma_k(x)
= x.
\end{align*}
Therefore, $\varphi_k$ is a continuous section of $\sigma_k$.} This proves (i).

Now we prove (ii). Let $(x,y) \in R(\sigma)$. Since $X$ is covered by the union of the sets $X_k$, then there exists a $k$ such that $x, y \in X_k$. By definition of $\sigma_k$, we have that
\begin{align*}
\sigma_k(x) = \sigma(x) = \sigma(y) = \sigma_k(y).
\end{align*}
Hence, $(x,y) \in R(\sigma_k)$. Therefore, $R(\sigma) \subset \bigcup_{k=1}^\infty R(\sigma_k)$.

Now let $(x,y) \newc{\in} \bigcup_{k=1}^\infty R(\sigma_k)$. Pick a $k$ such that $(x,y) \in R(\sigma_k)$. Then $\sigma_k(x) = \sigma_k(y)$. If $x \notin X_k$, then $x = y$ by definition of $\sigma_k$, and therefore $(x,y) \in R(\sigma)$. If $x \in X_k$, then $y \in X_k$ by definition of $\sigma_k$. This implies that
\begin{align*}
\sigma(x) = \sigma_k(x) = \sigma_k(y) = \sigma(y).
\end{align*}
Thus, $(x,y) \in R(\sigma)$. Therefore, $R(\sigma) = \bigcup_{k=1}^\infty R(\sigma_k)$. \oldc{Note that this union}\newc{By similar arguments, we have that} this union is increasing and all groupoids have the same unit space $X$. Therefore, (ii) holds.

\newc{Finally, we show (iii). Let $k \geq 1$, and let $p \in \mathcal{P}_\infty(C_c(R(\sigma_k)))$. Then there exists an $n$ such that $p \in \mathcal{P}_n(C_c(R(\sigma)))$. By increasing $n$ if necessary, we have $p \in \mathcal{P}_n(C_c(R(\sigma_n)))$ and that the support of $p$ is included in $X_n$. Then for any $l \geq n$ and $x \in X$
\begin{align*}
\mathrm{tr}_{\varphi_l}\hphantom{.}p(x)
&= 1_{X_{\varphi_l}}(x) \sum_{i=1}^n \sum_{u: \sigma_l(u) = \sigma_l(x)} p_{ii}(u) \\
&= 1_{X_{\varphi_l}}(x) \sum_{i=1}^n \sum_{u \in X_l: \sigma_l(u) = \sigma_l(x)} p_{ii}(u) \\
&= 1_{X_{\varphi_l}}(x) \sum_{i=1}^n 1_{X_l} \sum_{u : \sigma(u) = \sigma(x)} p_{ii}(u) \\
&= 1_{X_{\varphi_l} \cap X_l}(x) \sum_{i=1}^n 1_{X_l} \sum_{u : \sigma(u) = \sigma(x)} p_{ii}(u),
\end{align*}
while
\begin{align*}
\mathrm{tr}_{\varphi}\hphantom{.}p(x)
&= 1_{X_\varphi}(x) \sum_{i=1}^n \sum_{u:\sigma(u) = \sigma(x)} p_{ii}(u) \\
&= 1_{X_\varphi}(x) \sum_{i=1}^n \sum_{u \in X_l:\sigma(u) = \sigma(x)} p_{ii}(u) \\
&= 1_{X_\varphi}(x) \sum_{i=1}^n 1_{\sigma^{-1}(\sigma(X_l))}(x) \sum_{u \in X_l:\sigma(u) = \sigma(x)} p_{ii}(u) \\
&= 1_{X_\varphi \cap \sigma^{-1}(\sigma(X_l))}(x) \sum_{i=1}^n \sum_{u \in X_l:\sigma(u) = \sigma(x)} p_{ii}(u).
\end{align*}
Under the assumption that $\varphi(\sigma(X_l))$, one sees that
\begin{align*}
X_{\varphi_l} \cap X_l = \varphi(\sigma(X_l))
= X_\varphi \cap \sigma^{-1}(\sigma(X_l)).
\end{align*}
Therefore $\mathrm{tr}_{\varphi}\hphantom{.}p(x) = \mathrm{tr}_{\varphi_l}\hphantom{.}p (x)$.
}
\oldc{
Let $p \in \mathcal{P}_1(C^*(R(\sigma))) \subset K_0(C^*(R(\sigma)))^+$. By Proposition \ref{prop:inductivegroupunion}, there exists a postive integer $n$ such that $p \in \mathcal{P}_1(C^*(R(\sigma_n)))$. By Renault's $j$ map from \cite[Proposition II.4.2]{Renault}, we assume that $p \in C_0(R(\sigma_n))$. Note that $R(\sigma_n)\vert_{X \setminus X_n} = X \setminus X_n$.

Since $p$ is a projection and since $R(\sigma_n)\vert_{X \setminus X_n} = X \setminus X_n$, for $x \in X \setminus X_n$, we have
\begin{align*}
p(x) = p^2(x) = p(x)^2.
\end{align*}
This implies that $p(x) \in \lbrace 0, 1 \rbrace$ for $x \in X \setminus X_n$.

Let $K$ be a compact open set with $X_n \subset K \subset X$ and such that $\vert p(x) \vert < \frac{1}{2}$ for every $x$ outside $K$. Then $p$ is identically zero outside $K$. 

Choose $k$ such that $K \subset X_k$, and let $x \in X$. If $x \in X_k \cap X_\varphi$, then
\begin{align*}
\mathrm{tr}_\varphi\hphantom{.}p(x)
&= \sum_{u: \sigma(u) = \sigma(x)} p(u) \\
&= \sum_{\substack{u: \sigma(u) = \sigma(x) \\ u \in X_k}} p(u)
\hspace{20pt}\text{since $p(u) = 0$ for $u \notin X_k \supset K$}, \\
&= \sum_{u: \sigma_k(u) = \sigma_k(x)} p(u) \\
&= \mathrm{tr}_{\varphi_k}\hphantom{.}p(x).
\end{align*}
If $x \in X \setminus X_k$, then
\begin{align*}
\mathrm{tr}_\varphi\hphantom{.}p(x)
&= \sum_{u: \sigma(u) = \sigma(x)} p(u) \\
&= p(x) \text{\hspace{20pt} since $R(\sigma)\vert_{X \setminus X_k} = X \setminus X_k$,} \\
&= 0 \hspace{35pt}\text{since $p$ vanishes outside $X_k \supset K$.}
\end{align*}
Similarly,
\begin{align*}
\mathrm{tr}_{\varphi_k}\hphantom{.}p(x)
&= \sum_{u: \sigma_k(u) = \sigma_k(x)} p(u) \\
&= p(x) \text{\hspace{20pt} by definition of $\sigma_k$,} \\
&= 0.
\end{align*}
Therefore $\mathrm{tr}_\varphi\hphantom{.}p = \mathrm{tr}_{\varphi_k}\hphantom{.}p(x)$. Note that $k$ is an arbitrary positive integer with $K \subset X_k$. Since the sets $X_k$ are nested, we have that $\mathrm{tr}_\varphi\hphantom{.}p = \mathrm{tr}_{\varphi_l}\hphantom{.}p(x)$ for all $l \geq k$.}

\newc{Now let $p \in \mathcal{P}_n(C^*(R(\sigma)))$ for some $n$. Since $C^*(R(\sigma)) = \displaystyle\lim_\rightarrow C^*(R(\sigma_m))$, it follows from Lemma \ref{lemma:inductivelimitgroupoids} that $K_0(C^*(R(\sigma))) = \displaystyle\lim_\rightarrow K_0(C^*(R(\sigma_m)))$. Then there exists a $m$ such that $q \in C^*(R(\sigma_m))$ and $p \sim q$. Note, however, that $C^*(R(\sigma_m)) = C_c(R(\sigma_m)) \subset C_c(R(\sigma))$. Then $q \in C_c(R(\sigma))$. Hence, there exists a $k$ such that $\mathrm{tr}_\varphi\hphantom{.}q = \mathrm{tr}_{\varphi_l}\hphantom{.}q$ for all $l \geq k$.  Then}

\begin{align}
\newc{\label{eqn:trvarphilq}
\mathrm{tr}_{\varphi}\hphantom{.}q
= \mathrm{tr}_{\varphi_l}\hphantom{.}q
= \mathrm{tr}_{\varphi_l}\hphantom{.}p.}
\end{align}
\newc{Let $x \in X$. Choose $j \geq l$ such that $X_j$ contains both $x$ and $\sigma(x)$. Then}
\begin{align*}
\newc{\mathrm{tr}_{\varphi}\hphantom{.}p(x)}
&= \newc{1_{X_\varphi}(x) \sum_{i=1}^n \sum_{u:\sigma(u)=\sigma(x)} p_{ii}(u)} \\
&= \newc{1_{X_\varphi}(x) \sum_{i=1}^n \sum_{u:\sigma_j(u)=\sigma_j(x)} p_{ii}(u)
\hspace{20pt}\text{since $x, \sigma(x) \in X_j$,}} \\
&= \newc{\mathrm{tr}_{\varphi_j}\hphantom{.}p(x)} \\
&= \newc{\mathrm{tr}_{\varphi_l}\hphantom{.}p(x)
\hspace{116pt}\text{by \eqref{eqn:trvarphilq}.}}
\tag*{\qedhere}
\end{align*}
\end{proof}


\begin{proof}[Proof of Part 2]\renewcommand{\qedsymbol}{}
Fix a continuous section $\varphi$ of $\sigma$. Since $X$ is second countable and totally disconnected, then we write \oldc{$X = \bigcup_{k=1}^\infty X_k$}\newc{$X = \bigcup_{k=1}^\infty \widetilde{X}_k$} as the increasing union of compact open sets \oldc{$X_k$}\newc{$\widetilde{X}_k$}. \newc{For each $k$, let $X_k = \widetilde{X_k} \cap \varphi(\sigma(X_k))$. Then $X = \bigcup_{k=1}^\infty X_k$ is the increasing union of the compact open sets $X_k$, with $\varphi(\sigma(X_k)) \subset X_k$.} Lemma \ref{lemma:tracevarphi} gives a sequence of surjective local homeomorphisms $\sigma_k: X_k \rightarrow \sigma(X_k) \sqcup (X \setminus X_k)$ satisfying properties (ii), (iii) of the lemma.

By Part 1, we have the isomorphism $\nu_k: K_0(C^*(R(\sigma_k))) \rightarrow H_0(R(\sigma_k))$ of ordered groups given by
\begin{align*}
\oldc{\nu_k([p])_{0,k}}\newc{\nu_k([p]_{0,k})} = [\mathrm{tr}_{\varphi_k}\hphantom{.}p]_k
\end{align*}
for $p \in \mathcal{P}_\infty(C^*(R(\sigma_k)))$. This isomorphism is such that $[1_V]_{0,k} \mapsto [1_V]_k$ for all compact open sets $V \subset X$.

Here $[\cdot]_{0,k}$, $[\cdot]_{0}$, $[\cdot]_k$ and $[\cdot]$ denote the equivalence classes in the groups
\begin{align*}
K_0(C^*(R(\sigma_k))),\hspace{10pt}
K_0(C^*(R(\sigma))), \hspace{10pt}
H_0(R(\sigma_k))\hspace{10pt}
\text{ and }\hspace{10pt}
H_0(R(\sigma)),
\end{align*}
respectively. Note that
\begin{align*}
K_0(C^*(R(\sigma))) = \displaystyle\lim_\rightarrow K_0(C^*(R(\sigma_k))
\hspace{10pt}
\text{and}
\hspace{10pt}
H_0(R(\sigma)) = \displaystyle\lim_\rightarrow H_0(R(\sigma_k)).
\end{align*}
For all $k \geq 1$, let
\begin{align*}
i_k: &K_0(C^*(R(\sigma_k))) \rightarrow K_0(C^*(R(\sigma_{k+1})))
\hspace{5pt}
\text{and}\\
\hspace{5pt}
j_k: & H_0(R(\sigma_k)) \rightarrow H_0(R(\sigma_{k+1}))
\end{align*}
be the connecting morphisms, and let $\alpha_k: K_0(C^*(R(\sigma_k))) \rightarrow K_0(C^*(R(\sigma)))$ and $\beta_k: H_0(R(\sigma_k)) \rightarrow H_0(R(\sigma))$ be the inclusion morphisms such that
\begin{align*}
i_k([p]_{0,k}) = [p]_{0,k+1}
\hspace{10pt}\text{and}\hspace{10pt}
\alpha_k([p]_{0,k}) = [p]_0
\hspace{10pt}
\text{for }
p \in \mathcal{P}_\infty(C^*(R(\sigma_k))), \\
j_k([f]_{k}) = [f]_{k+1}
\hspace{13pt}\text{and}\hspace{17pt}
\beta_k([f]_k) = [f]
\hspace{15pt}
\text{for }
f \in C(X, \mathbb{Z}).
\hspace{40pt}
\end{align*}
The morphisms $i_k$, $j_k$, $\alpha_k$, $\beta_k$ preserve positive elements from one group to another.

In order to apply Lemma \ref{lemma:homomorphisminductive}, we need to prove that $\nu_{k+1} \circ i_k = j_k \circ \nu_k$ for all $k \geq 1$. In fact, for a fixed $k$ and for a compact open $V \subset X$, we have
\begin{align*}
j_k \circ \nu_n([1_V]_{0,k})
= j_k([1_V]_k)
= [1_V]_{k+1}
= \nu_{k+1}([1_V]_{0, k+1})
= \oldc{\nu_{k+1} \circ i_k([1_V])}\newc{\nu_{k+1} \circ i_k([1_V]_{0,k})}.
\end{align*}
Note that the elements $[1_V]$ generate $H_0(R(\sigma_k))$. By the isomorphism $\nu_k$, the classes \oldc{$[1_V]_0$}\newc{$[1_V]_{0,k}$} generate $K_0(C^*(R(\sigma_k)))$. So the equation above implies that $j_k \circ \nu_k = \nu_{k+1} \circ i_k$. Lemma \ref{lemma:homomorphisminductive} gives an isomorphism $\nu: \oldc{K_0(R(\sigma))}\newc{K_0(C^*(R(\sigma)))} \rightarrow H_0(R(\sigma))$ that makes the diagram below commutative for all $k$.
\begin{equation*}
\begin{tikzcd}
K_0(C^*(R(\sigma_k)) \arrow[r,"\nu_k"] 
\arrow{d}{}[swap]{\alpha_k}
& H_0(R(\sigma_k)) \arrow[d, "\beta_k"]\\
K_0(C^*(R(\sigma)) \arrow{r}{\nu}[swap]{}
& H_0(R(\sigma))
\end{tikzcd}
\end{equation*}
Then, for all compact open subset $V \subset X$ we have that
\begin{align*}
\nu([1_V]_0) = \nu \circ \alpha_k([1_V]_{0,k})
= \beta_k \circ \nu_k([1_V]_{0,k})
= \beta_k([1_V]_k)
= [1_V].
\end{align*}

Now let $n \geq 1$ and $p \in \mathcal{P}_n(C^*(R(\sigma))) \subset K_0(C^*(R(\sigma)))^+$. \newc{By Lemma \ref{lemma:tracevarphi}, there exists a $k$ such that $\mathrm{tr}_{\varphi_k}\hphantom{.}p = \mathrm{tr}_{\varphi}\hphantom{.}p$. then}
\begin{align}
\nu([p]_0)
&= \nu(\alpha_k([p]_{0,k}) \nonumber\\
&= \beta_k \circ \nu_k([p]_{0,k}) \nonumber\\
&= \beta_k([\mathrm{tr}_{\varphi_k}\hphantom{.}p]_k)
\hspace{20pt}\text{by Part 1,}\nonumber\\
&= [\mathrm{tr}_{\varphi_k}\hphantom{.}p] \nonumber\\
&= [\mathrm{tr}_{\varphi}\hphantom{.}p]. \label{eqn:nutrvarphi}
\end{align}
This completes the proof of Part 2.
\end{proof}

\begin{proof}[Proof of Part 3]
The proof is analogous to Part 2. Let $G$ be an AF groupoid. By Proposition \ref{prop:AFRsigma}, we write $G = \displaystyle\lim_\rightarrow R(\sigma_n)$, where $\sigma: G^{(0)} \rightarrow Y_n$ is a surjective local homeomorphism of totally disconnected spaces. As in Part 2, let $i_n: K_0(C^*(R(\sigma_n))) \rightarrow K_0(C^*(R(\sigma_{n+1})))$ and $j_n: H_0(R(\sigma_n)) \rightarrow H_0(R(\sigma_{n+1}))$ be the inclusion maps. By Part 2, we have an isomorphism $\nu_n: K_0(C^*(R(\sigma_n))) \rightarrow H_0(R(\sigma_n))$ satisfying the properties 1--3 of this theorem for all $n$.

Lemma \ref{lemma:homomorphisminductive} gives an isomorphism $\nu: K_0(C^*(G)) \rightarrow H_0(G)$ of ordered groups such that the we have the following commutative diagram for all $n$, where $\alpha_n$ and $\beta_n$ are the inclusion maps.
\begin{equation*}
\begin{tikzcd}
K_0(C^*(R(\sigma_n)) \arrow[r,"\nu_n"] 
\arrow{d}{}[swap]{\alpha_n}
& H_0(R(\sigma_n)) \arrow[d, "\beta_n"]\\
K_0(C^*(G) \arrow{r}{\nu}[swap]{}
& H_0(G)
\end{tikzcd}
\end{equation*}
Using the same ideas to prove \oldc{\eqref{eqn:trvarphik} and}\eqref{eqn:nutrvarphi} in Part 2, we have that $\nu$ satisfies properties 1--3. This completes the proof of the theorem.
\end{proof}

\begin{corollary}
\label{corollary:1VgenerateK0}
Given an AF groupoid $G$, $K_0(C^*(G))^+$ is generated  by elements of the form $[1_V]_0$, where $V \subset G^{(0)}$ are compact open subsets.
\end{corollary}
\begin{proof}
Recall that $H_0(G)^+ = \lbrace [f] : f \in C_c(X, \mathbb{N}) \rbrace$. Since $G^{(0)}$ is locally compact, Hausdorff, second countable, and totally disconnected, then every $f \in C_c(G^{(0)}, \mathbb{N})$ is a linear combination of the form $f = \alpha_1 1_{V_1} + \dots + \alpha_n 1_{V_n}$, where all $\alpha_i$ non-negative integers, and $V_i$ are compact open. This implies that the classes of the form $[1_V]$, for $V \subset G^{(0)}$ compact open, generate $H_0(G)^+$.

By Theorem \ref{thm:K0H0}, there is an ordered group isomorphism $K_0(C^*(G)) \rightarrow H_0(G))^+$ that maps $[1_V]_0$ to $[1_V]$. Since the equivalence classes $[1_V]$, for $V \subset G^{(0)}$ compact open, generate $H_0(G)^+$, then the corresponding equivalence classes $[1_V]_0$ generate $K_0(C^*(G))^+$.
\end{proof}


\chapter{The commutative diagram}
\label{section:diagram}

Recall from the introduction that we want to find a group which is isomorphic to $K_0(B)$ and more manageable, where $\mathcal{G}$ is a Deaconu-Renault groupoid,  $B$ is the C*-algebra of the skew product groupoid $\mathcal{G}(c)$, and $c: \mathcal{G} \rightarrow \mathbb{Z}$ is the continuous cocycle given by $c(x,k,y) = k$, called the \newterm{canonical cocycle}.

By studying this new group which is isomorhic to $K_0(B)$, we will find a condition on $\sigma$ that is equivalent to the condition $H_\beta \cap K_0(B)^+ = \lbrace 0 \rbrace$ from Brown's theorem (see Theorem \ref{thm:brown} on page \pageref{thm:brown}).

We prove the diagram \eqref{eqn:diagram} below by using the techniques for homology groups described in Chapter \ref{section:K0H0}, and by applying Theorem \ref{thm:K0H0} from the same chapter. 


\begin{equation}
\tag{\ref{eqn:diagram}}
\begin{tikzcd}
      K_0(B) \arrow[r, "K_0(\beta)"] \arrow{d}{\cong}[swap]{\circledRed{1} \hspace{6pt}}
      & K_0(B) \arrow{d}{}[swap]{\cong} \\
      H_0(\mathcal{G}(c))  \arrow[r, "{[\widetilde{\beta}]}_{\mathcal{G}(c)}"] \arrow{d}{\cong}[swap]{\circledRed{2} \hspace{6pt}} & H_0(\mathcal{G}(c)) \arrow{d}{}[swap]{\cong}\\
      H_0(c^{-1}(0))  \arrow{r}{[\sigma_\ast] }[swap]{} & H_0(c^{-1}(0))   
\end{tikzcd}
\end{equation}

The action $\beta$ is described in Section \ref{subsection:crossedproducts:beta}. For a Deaconu-Renault groupoid $\mathcal{G}$, $\beta \in \mathrm{Aut}(B)$ is such that $\beta(f)(x,k,y,a) = f(x,k,y,a + 1)$ for $(x,k,y,a) \in \mathcal{G}(c)$, $f \in C_c(\mathcal{G}(c))$.

Note that the diagram has two parts. We prove the first part in Section \ref{section:diagram:part1} using Theorem \ref{thm:K0H0} from the previous chapter. In Section \ref{section:diagram:part2}, we prove the second part of the diagram by using the results on homology groups of Chapter \ref{section:homology}.

In order to avoid confusion among the equivalence classes of the different groups in the diagram, we use the following convention:
\begin{itemize}
 \item $[ \hphantom{f} ]_{K_0(B)}$ denotes the equivalence class in $K_0(B)$,
 \item $[ \hphantom{f} ]_{\mathcal{G}(c)}$ is the equivalence class in $H_0(\mathcal{G}(c))$,
 \item $[ \hphantom{f} ]$ is the equivalence class in $H_0(c^{-1}(0))$
\end{itemize}

In this chapter, we fix $X$ to be a locally compact, Hausdorff, second countable, totally disconnected space, and we let $\sigma: X \rightarrow X$ be a surjective local homeomorphism. Here $\mathcal{G}$ denotes the Deaconu-Renault groupoid for $\sigma: X \rightarrow X$, and we fix the notation $B = C^*(\mathcal{G}(c))$.
\par\nobreak\section{Part 1 of the diagram}
\label{section:diagram:part1}

We will \oldc{to }apply Theorem \ref{thm:K0H0} to find the isomorphism $K_0(B) \cong H_0(\mathcal{G}(c))$. Note that Theorem \ref{thm:K0H0} applied because $\mathcal{G}(c)$ is AF by \cite[Corollary 5.2]{FKPS}. This theorem gives an ordered group isomorphism $\mu: H_0(\mathcal{G}(c)) \rightarrow K_0(B)$ such that, for $V \subset X$ compact open and for $a \in \mathbb{Z}$, we have $\mu([1_{V \times \lbrace a \rbrace}]_{\mathcal{G}(c)}) = [1_{V \times \lbrace a \rbrace}]_{K_0(B)}$.

Below we define the map $\widetilde{\beta}$ of the diagram.

\begin{definition}
\label{def:betatilde}
Define $\widetilde{\beta}: C_c(X \times \mathbb{Z}, \mathbb{Z}) \rightarrow C_c(X \times \mathbb{Z}, \mathbb{Z})$ by $\oldc{\beta(f)(x,a)}\newc{\widetilde{\beta}(f)(x,a)} = f(x, a+1)$, for $(x,a) \in X \times \mathbb{Z}$.
\end{definition}

Note that \oldc{for $f \in \oldc{C_c(X, \mathbb{Z})}$}\newc{$\widetilde{\beta}$ is the restriction of $\beta$ to $C_c(X \times \mathbb{Z}, \mathbb{Z})$. Indeed, for $f \in C_c(X \times \mathbb{Z}, \mathbb{Z})$,} we have $\widetilde{\beta}(f) = \beta(f)$. In particular, this equality holds if $f = 1_{V \times \lbrace a \rbrace}$ for a compact open $V \subset X$ and for $a \in \mathbb{Z}$. Now we study the ordered group morphism $[\widetilde{\beta}]$ on $H_0(\mathcal{G}(c))$.

\begin{lemma}
\label{lemma:betatildeH0}
The map $[\widetilde{\beta}]_{\mathcal{G}(c)}: H_0(\mathcal{G}(c)) \rightarrow H_0(\mathcal{G}(c))$ is well-defined.
\end{lemma}
\begin{proof}
By Remark \ref{rmk:rhobracket}, $[\widetilde{\beta}]_{\mathcal{G}(c)}$ is given by
\begin{align*}
[\widetilde{\beta}]_{\mathcal{G}(c)}([f]_{\mathcal{G}(c)}) = [\widetilde{\beta}(f)]_{\mathcal{G}(c)}
\end{align*}

In order to prove that $[\widetilde{\beta}]_{\mathcal{G}(c)}$ is well-defined, we only need to show that, given $f_1, f_2 \in C_c(X \times \mathbb{Z}, \mathbb{Z})$ which are equivalent in $H_0(\mathcal{G}(c))$, we have $\widetilde{\beta}(f_1) \sim \widetilde{\beta}(f_2)$. So, let $f_1, f_2 \in C_c(X \times \mathbb{Z}, \mathbb{Z})$ be equivalent in $H_0(\mathcal{G}(c))$. Then there is an $F \in C_c(\mathcal{G}(c), \mathbb{Z})$ such that $f_1 = f_2 + \partial_1 F$. Given $(x,a) \in X \times \mathbb{Z}$, we have
\begin{align*}
\beta(\partial_1 F)(x,a)
&= \partial_1 F(x, a+ 1) \\
&= s_\ast(F)(x, a+1) - r_\ast(F)(x, a+1) \\
\oldc{
&= \sum_{g: r(g) = (x,a+1)} [F(g^{-1}) - F(g)] \\
&= \sum_{(x,k,y) \in \mathcal{G}} [F(y,-k,x, a + 1 + k) - F(x, k,y,a+1)] \\
&= \sum_{(x,k,y) \in \mathcal{G}} [\beta(F)(y,-k,x,a + k) - \beta(F)(x,k,y,a)] \\
&= \sum_{h : r(h) = (x,a)} [\beta(F)(h^{-1}) - \beta(F)(h)] \\}
&= \newc{\sum_{(g,b):s(g,b) = (x, a +1)} F(g,b) - \sum_{(g,b):r(g,b) = (x, a +1)} F(g,b)} \\
&= \newc{\sum_{g \in \mathcal{G}: s(g) = x} F(g,a+1 - c(g)) - \sum_{g \in \mathcal{G}: r(g) = x} F(g,a+1)} \\
&= \newc{\sum_{g \in \mathcal{G}: s(g) = x} \beta(F)(g, a - c(g)) - \sum_{g \in \mathcal{G}: r(g) = x} \beta(F)(g,a)} \\
&= s_\ast(\beta(F))(x,a) - r_\ast(\beta(F))(x,a) \\
&= \partial_1(\beta(F))(x,a).
\end{align*}
Thus, $\widetilde{\beta}(f_1) = \widetilde{\beta}(f_2) + \partial_1(\beta(F))$. Note that $\beta(F) \in C_c(\mathcal{G}(c), \mathbb{Z})$. Therefore $\beta(f_1)$ and $\beta(f_2)$ are equivalent.
\end{proof}

\begin{proposition}
The first part of the diagram \eqref{eqn:diagram} is commutative.
\end{proposition}
\begin{proof}
Since \oldc{$H_0(c^{-1}(0))$}\newc{$H_0(\mathcal{G}(c))$} is generated by elements of the form $[1_{V \times \lbrace a \rbrace}]_{\mathcal{G}(c)}$ for $V \subset X$ compact open and $a \in \mathbb{Z}$, then we only need to prove that, for all $V \subset X$ compact open and $a \in \mathbb{Z}$,
\begin{align*}
K_0(\beta) \circ \mu([1_{V \times \lbrace a \rbrace}]_{\mathcal{G}(c)})
= \mu \circ [\widetilde{\beta}]_{\mathcal{G}(c)}([1_{V \times \lbrace a \rbrace}]_{\mathcal{G}(c)}).
\end{align*}
So, fix $V$ and $a$. Then
\begin{align*}
K_0(\beta) \circ \mu([1_{V \times \lbrace a \rbrace}]_{\mathcal{G}(c)})
&= K_0(\beta)([1_{V \times \lbrace a \rbrace}]_{K_0(B)}) \\
&= [\beta(1_{V \times \lbrace a \rbrace})]_{K_0(B)} \\
&= [ \widetilde{\beta}(1_{V \times \lbrace a \rbrace})]_{K_0(B)} \\
&= \mu([\widetilde{\beta}(1_{V \times \lbrace a \rbrace})]_{\mathcal{G}(c)}) \\
&= \mu \circ [\widetilde{\beta}]_{\mathcal{G}(c)}([1_{V \times \lbrace a \rbrace}]_{\mathcal{G}(c)}).
\end{align*}
Therefore the result holds.
\end{proof}

\par\nobreak\section{Part 2 of the diagram}
\label{section:diagram:part2}

In the second part of the diagram, we have an isomorphism between homology groups of different groupoids. We will obtain this isomorphism by applying the techniques of homological similarity of groupoids described in Section \ref{subsection:homology:similarity}.

Now we define the groupoid $c^{-1}(0)$.
\begin{definition}
Define the set
\begin{align*}
c^{-1}(0) = \lbrace (x,y) \in X \times X: \sigma^n(x) = \sigma^n(y) \text{ for some }n \in \mathbb{N} \rbrace.
\end{align*}
We equip this set with the operations $(x,y)(y,z) = (x,z), (x,y)^{-1} = (y,x)$ and with the range a source maps $r,s: c^{-1}(0) \rightarrow X$ given by $r(x,y) = x$ and $s(x,y) = y$. We also equip the set with the subspace topology from $X \times X$. Then $c^{-1}(0)$ a locally compact, Hausdorff, second countable, \'etale, totally disconnected groupoid.
\end{definition}

We use the notation $c^{-1}(0)$ because of the following. If $\mathcal{G}$ is the Deaconu-Renault groupoid and $c: \mathcal{G} \rightarrow \mathbb{Z}$ is the canonical cocycle, then $c^{-1}(0)$ is isomorphic to the subgroupoid $\lbrace (x,0,y) \in \mathcal{G} \rbrace$ of $\mathcal{G}$. The isomorphism is defined by $(x,y) \mapsto (x,0,y)$.

Given $\varphi: X \rightarrow X$ to be a continuous section of $\sigma$, we define $\varphi^0 = id_X$ and we set $\varphi^a = \sigma^{-a}$ for every integer $a < 0$. This notation will be useful when we study the similarity of $\mathcal{G}(c)$ and $c^{-1}(0)$. 

The next lemma shows that a surjection $\sigma:X \rightarrow W$ always has a continuous section $\varphi$. The ideas of the next lemma are taken from the proof of \cite[Theorem 4.10]{FKPS} by Farsi et al.

\begin{lemma}
\label{lemma:continuoussection}
The surjective local homeomorphism $\sigma$ has a continuous section $\varphi: W \rightarrow X$. Moreover, $\varphi$ is a local homeomorphism
\end{lemma}
\begin{proof}
Since $X$ is second countable and totally disconnected, and since $\sigma$ is a local homeomorphism, there is a cover $\lbrace \mathcal{U}_n \rbrace_{n=1}^\infty$ of $X$ by compact open sets such that $\sigma$ is injective on each $\mathcal{U}_n$.

We define a sequence $\lbrace V_n \rbrace_{n=1}^\infty$ such that $\lbrace \sigma(V_n) \rbrace_{n=1}^\infty$ will be a partition of \oldc{$X$}\newc{$W$, and such that $\sigma\vert_{V_n}: V_n \rightarrow \sigma(V_n)$ is a homeomorphism for all $n$.} Let $V_1 = \mathcal{U}_1$ and, for every $n \geq 0$, set
\begin{align*}
V_{n+1} = \mathcal{U}_{n+1} \setminus \oldc{\sigma\vert_{\mathcal{U}_n}^{-1}}\newc{\sigma^{-1}} \left( \sigma \left( \bigcup_{i=1}^n \mathcal{U}_i \right) \right)
=  \mathcal{U}_{n+1} \setminus \bigcup_{i=1}^n \oldc{\sigma\vert_{\mathcal{U}_n}^{-1}}\newc{\sigma^{-1}} (\sigma(\mathcal{U}_i )).
\end{align*}
Note that each $V_n$ is compact open, and that $\sigma$ is injective on each $V_n$. Also, for every $n \geq 0$, we have that $\sigma(V_n)$ is compact \newc{open} and
\begin{align*}
\sigma(V_{n+1}) = \sigma(\mathcal{U}_{n+1}) \setminus \bigcup_{i=1}^n \sigma(\mathcal{U}_i).
\end{align*}
Then the sets $\sigma(V_n)$ are pairwise disjoint. Moreover, since $\lbrace \mathcal{U}_n \rbrace_{n=1}^\infty$ covers $X$, and $\sigma$ is surjective, then we have
\begin{align*}
\bigcup_{n=1}^\infty \sigma(V_n)
&= \sigma(\mathcal{U}_1) \cup
\oldc{\bigcup_{n=0}^\infty}\newc{\bigcup_{n=1}^\infty} \left(\sigma(\mathcal{U}_{n+1}) \setminus \bigcup_{i=1}^n \sigma(\mathcal{U}_i) \right)\\
&= \bigcup_{n=1}^\infty \sigma(\mathcal{U}_n)\\
&= \sigma \left(\bigcup_{n=1}^\infty \mathcal{U}_n \right)\\
&= \sigma(X)\\
&= \oldc{X}\newc{W}.
\end{align*}
Therefore $\lbrace \sigma(V_n) \rbrace_{n=1}^\infty$ is a partition of \oldc{$X$}\newc{$W$}.

Let $\varphi: \oldc{X}\newc{W} \rightarrow X$ be such that, for all $n \geq 1$, $\varphi\vert_{\sigma(V_n)}: \sigma(V_n) \rightarrow V_n$ is defined by
\begin{align}
\label{eqn:continuossection}
\varphi\vert_{\sigma(V_n)} = \sigma\vert_{V_n}^{-1}.
\end{align}
Then $\varphi$ is a continuous function by definition. Moreover, $\varphi$ is a local homeomorphism. Finally, we show that $\varphi$ is a section \newc{of $\sigma$}. Let $x \in X$. Then there exists a unique $n \geq 1$ such that \oldc{$x \in \sigma(V_n)$}\newc{$x \in V_n$}. Then, by \eqref{eqn:continuossection}, we have
\begin{align*}
\sigma(\varphi(x)) = \sigma( \sigma\vert_{\sigma(V_n)}^{-1}(x)) = x.
\end{align*}
Therefore $\varphi$ is a continuous section of $\sigma$.
\end{proof}

That $X$ is totally disconnected and second countable was fundamental in the proof above when we chose a cover of $X$ by compact open sets $\mathcal{U}_n$. Until the rest of this section, we fix $\varphi$ to be a continuous section of $\sigma$.

Let us define the map that induces the similarity of the groupoids we are studying.

\begin{lemma}
\label{lemma:rhophi}
There exists a homomorphism $\rho_\varphi: \mathcal{G}(c) \rightarrow c^{-1}(0)$ given by
\begin{align*}
\rho_\varphi(x,k,y,a) = (\varphi^a(x), \varphi^{a+k}(y)).
\end{align*}
Moreover, $\rho_\varphi$ is a local homeomorphism.
\end{lemma}
\begin{proof}
First we show that $\rho_\varphi$ is well-defined. Let $(x,k,y,a) \in \mathcal{G}(c)$, and let $m,n \in \mathbb{N}$ be such that $\sigma^m(x) = \sigma^n(y)$ and $k = m - n$. We consider different values of $a$ and $a+k$.

\begin{itemize}
\item Suppose $a \geq 0$ and $a + k \geq 0$.

By definition of $\varphi$, we have
\begin{align*}
\sigma^a(\varphi^a(x)) = x
\text{\hspace{10pt}and\hspace{10pt}}
\sigma^{a+k}(\varphi^{a+k}(y)) = y.
\end{align*}
Then
\begin{align*}
\sigma^{m+a}(\varphi^a(x))
&= \sigma^m(\sigma^a(\varphi^a(x)))
= \sigma^m(x)
= \sigma^n(y), \text{\hspace{10pt}and} \\
\sigma^{m+a}(\varphi^{a+k}(y))
&= \sigma^{n+a+k}(\varphi^{a+k}(y)), \text{\hspace{10pt} since $m=n+k$,}\\
&=\sigma^n(\sigma^{a+k}(\varphi^{a+k}(y))) \\
&= \sigma^n(y).
\end{align*}
This implies that $(\varphi^a(x), \varphi^{a+k}(y)) \in c^{-1}(0)$.

\item Suppose $a \geq 0$ and $a+k < 0$.
Then
\begin{align*}
\sigma^a(\varphi^a(x)) = x
\text{\hspace{10pt}and\hspace{10pt}}
\varphi^{a+k}(y) = \sigma^{-a-k}(y).
\end{align*}
Then
\begin{align*}
\sigma^{m+a}(\varphi^a(x))
&= \sigma^m(\sigma^a(\varphi^a(x)))
=  \sigma^m(x)
= \sigma^n(y), \text{\hspace{10pt}and}\\
\sigma^{m+a}(\varphi^{a+k}(y))
&=  \sigma^{m+a}(\sigma^{-a-k}(y))
= \sigma^{m+a-a-k}(y)
= \sigma^{m-k}(y)
= \sigma^n(y).
\end{align*}
Hence $(\varphi^a(x), \varphi^{a+k}(y)) \in c^{-1}(0)$.
\item Suppose $a < 0$ and $a+k \geq 0$.

Note that $m+a = n+k+a \geq 0$. Using this fact, we have
\begin{align*}
\sigma^{m+a}(\varphi^a(x))
&= \sigma^{m+a}(\sigma^{-a}(x))
= \sigma^m(x)
= \sigma^n(y), 
\text{\hspace{10pt} and}\\
\sigma^{m+a}(\varphi^{a+k}(y))
&= \sigma^{n+k+a}(\varphi^{a+k}(y))
= \sigma^n(\sigma^{k+a}(\varphi^{a+k}(y)))
= \sigma^n(y).
\end{align*}
Then $(\varphi^a(x), \varphi^{a+k}(y)) \in c^{-1}(0)$.
\item Suppose $a < 0$ and $a+k < 0$.
Then
\begin{align*}
\sigma^m(\varphi^a(x))
&= \sigma^m(\sigma^{-a}(x))
= \sigma^{-a}(\sigma^m(x))
= \sigma^{-a}(\sigma^n(y))
= \sigma^{n-a}(y), \\
\sigma^m(\varphi^{a+k}(y))
&= \sigma^m(\sigma^{-a-k}(y))
= \sigma^{m-k-a}(y)
= \sigma^{n-a}(y).
\end{align*}
Then $(\varphi^{a}(x), \varphi^{a+k}(y)) \in c^{-1}(0)$. Therefore $\rho_\varphi$ is well-defined.
\end{itemize}
Next we show that \oldc{$\rho$}\newc{$\rho_\varphi$} is a homomorphism. Given $(x,k,y,a), (y,l,z,a+k) \in \mathcal{G}(c)$, we have
\begin{align*}
\oldc{\rho_\varphi((x,k,y,a)^{-1})
&= \rho_\varphi(y,-k,x,a+k) \\
&= (\varphi^{a+k}(y), \varphi^a(x)) \\
&= (\varphi^a(x), \varphi^{a+k}(y))^{-1}\\
&= \rho_\varphi((x,k,y,a))^{-1}, \text{\hspace{10pt}and}\\}
\rho_\varphi((x,k,y,a)(y,l,z,a+k))
&= \rho_\varphi(x,k+l,z,a)\\
&= (\varphi^a(x), \varphi^{a+k+l}(z))\\
&= (\varphi^a(x), \varphi^{a+k}(y))(\varphi^{a+k}(y), \varphi^{a+k+l}(z))\\
&= \rho_\varphi(x,k,y,a)\rho_\varphi(y,l,z,a+k).
\end{align*}
Then $\rho_\varphi$ is a homomorphism.

Finally, we prove that $\rho_\varphi$ is a local homeomorphism. By Lemma \ref{lemma:rhon-equivalent}, we only need to show that $\rho^{(0)}$ is a local homeomorphism. Fix $a \in \mathbb{Z}$. Then we have the function
\begin{align*}
\rho_\varphi^{(0)}\vert_{X \times \lbrace a \rbrace}: X \times \lbrace a \rbrace &\rightarrow X\\
(x,a) &\mapsto \varphi^a(x).
\end{align*}
Note that $\sigma$ is a local homeomorphism by assumption and that $\varphi$ is also a local homeomorphism by Lemma \ref{lemma:continuoussection}. Thus $\varphi^a$ is local homeomorphism. This implies that $\rho_\varphi^{(0)}$ is a local homeomorphism on each clopen subset $X \times \lbrace a \rbrace$ of $X \times \mathbb{Z}$. Then $\rho_\varphi^{(0)}$ is a local homeomorphism and, therefore, so is $\rho_\varphi$.
\end{proof}

Next we show that $\rho_\varphi$ induces an isomorphism of $H_0(\mathcal{G}(c))$ and $H_0(c^{-1}(0))$. 

\begin{proposition}
\label{prop:similarityGcc0}
$H_0(\rho_\varphi): H_0(\mathcal{G}(c)) \rightarrow H_0(c^{-1}(0))$ is an isomorphism and has inverse $H_0(\eta)$, where $\eta: c^{-1}(0) \rightarrow \mathcal{G}(c)$ is defined by $\eta(x,y) = (x,0,y,0)$.
\end{proposition}
\begin{proof}
It follows from the definition that $\eta$ is a well-defined homomorphism. Note that $\eta^{(0)}$ is defined by $\oldc{\eta{(0)}(x)}\newc{\eta^{(0)}(x)} = (x,0)$. Thus $\eta^{(0)}$ is a local homeomorphism. By Lemma \ref{lemma:rhon-equivalent}, $\eta$ is also a local homeomorphism.

Now we show that $\rho_\varphi$ and $\eta$ determine a similarity between $\mathcal{G}(c)$ and $c^{-1}(0)$. Let $(x,y) \in c^{-1}(0)$. Then
\begin{align*}
\rho_\varphi \circ \eta
(x,y)
= \rho_\varphi(x,0,y,0)
= (\varphi^0(x), \varphi^0(y))
= (x,y).
\end{align*}
So $\rho_\varphi \circ \eta = \mathrm{id}_{c^{-1}(0).}$.

Now we show that $\eta \circ \rho_\varphi$ is homologically similar to $\mathrm{id}_{\mathcal{G}(c)}$. Let $\theta: X \times \mathbb{Z} \rightarrow \mathcal{G}(c)$ be defined by
\begin{align*}
\theta(x,a) =  (x,-a, \varphi^a(x), a).
\end{align*}
Note this map is well-defined. In fact, if $a \geq 0$, we have that $\sigma^a(\varphi^a(x)) = x$. Then $(x,-a,\varphi^a(x),a) \in \mathcal{G}(c)$. If $a < 0$, then $\varphi^a(x) = \sigma^{-a}(x)$, which implies that $(x,-a, \varphi^a(x),a) \in \mathcal{G}(c)$. This function is also continuous \newc{because it continuous on each coordinate.} \oldc{Let $\lbrace (x_n, a_n) \rbrace_{n \geq 1}$ be a sequence in $X \times \mathbb{Z}$ converging to $(x,a)$. Since $X \times \lbrace a \rbrace$ is clopen, we can assume that $a_n = a$ for each $n$. Since $\varphi^a$ is continuous, we have that $\varphi^a(x_n) \rightarrow \varphi^a(x)$. This implies that
\begin{align*}
(x_n, -a, \varphi^a(x_n), a) \rightarrow (x,-a,\varphi^a(x), a).
\end{align*}
Thus $\theta$ is continuous.}

Let $g = (x,k,y,a) \in \mathcal{G}(c)$. Then
\begin{align*}
\theta(r(g)) \eta \circ \rho_\varphi(g)
&= \theta(x,a) \eta \circ \rho_\varphi(x,k,y,a) \\
&= (x,-a, \varphi^a(x), a) \eta \circ \rho_\varphi(x,k,y,a) \\
&= (x, - a, \varphi^a(x), a) \eta(\varphi^a(x), \varphi^{a+k}(y)) \\
&= (x, - a, \varphi^a(x), a) \oldc{(\varphi^a(x), 0, \varphi^{a+k}(y))}\newc{(\varphi^a(x), 0, \varphi^{a+k}(y),0)} \\
&= (x, - a, \varphi^{a+k}(y), a),
\text{\hspace{7pt} and}\\
\mathrm{id}_{\mathcal{G}(c)}(g) \theta(s(g))
&= (x,k,y,a)\theta(y,a+k) \\
&= (x,k,y,a)(y,-a -k, \varphi^{a+k}(y), a+k)\\
&= (x, -a, \varphi^{a+k}(y),a).
\end{align*}
So $\eta \circ \rho_\varphi$ and $\mathrm{id}_{\mathcal{G}(c)}$ are homologically similar. Therefore, $\mathcal{G}(c)$ and $c^{-1}(0)$ are homologically similar. Moreover, by Corollary \ref{corollary:isoHn}, we have an isomorphism $H_0(\rho_\eta): H_0(\mathcal{G}(c)) \rightarrow H_0(c^{-1}(0))$  with inverse $H_0(\eta)$.
\end{proof}

Now we show that the second part of the diagram commutes.

\begin{proposition}
The diagram
\par\nobreak\begin{figure}[H]
 \centering
\begin{tikzcd}
      H_0(\mathcal{G}(c)) \arrow[r, "{[\widetilde{\beta}]}_{\mathcal{G}(c)}"] & H_0(\mathcal{G}(c))\\
      H_0(c^{-1}(0)) \arrow[u,"\cong"] \arrow{r}{[\sigma_\ast] }[swap]{} & H_0(c^{-1}(0))  \arrow{u}[swap]{\cong}
\end{tikzcd}
\end{figure}

commutes.
\end{proposition}
\begin{proof}
By Proposition \ref{prop:similarityGcc0}, $H_0(\rho_\varphi): H_0(\mathcal{G}(c)) \rightarrow H_0(c^{-1}(0))$ is an isomorphism with inverse $H_0(\eta)$.

In order to show that the diagram above commutes, we prove that $(\rho_\varphi^{(0)})_\ast \circ \widetilde{\beta} \circ \eta^{(0)}_\ast = \sigma_\ast$. Let $f \in C_c(X, \mathbb{Z})$ and $x \in X$. Then
\begin{align*}
(\rho_\varphi)_\ast^{(0)} \circ \widetilde{\beta} \circ \eta_\ast^{(0)}(f)(x)
&= \sum_{(y,a): \rho_\varphi(y,a) = x} \widetilde{\beta} \circ \eta_\ast^{(0)}(f)(y,a)\\
&= \sum_{a \in \mathbb{Z}} \sum_{y: \varphi^a(y) = x} \widetilde{\beta} \circ \eta_\ast^{(0)}(f)(y,a),
\text{\hspace{10pt}by the def. of $\rho_\varphi$,}\\
&= \sum_{a \in \mathbb{Z}} \sum_{y: \varphi^a(y) = x} \eta_\ast^{(0)}(f)(y,a + 1),
\text{\hspace{10pt}by Definition \ref{def:betatilde}}.
\end{align*}
Note that
\begin{align*}
\eta_\ast^{(0)}(f)(y,a+1)
&= \begin{cases}
f(y) & \text{for } a = -1, \\
0 & \text{otherwise.}
\end{cases}
\end{align*}
Then
\begin{align*}
(\rho_\varphi)_\ast^{(0)} \circ \widetilde{\beta} \circ \eta_\ast^{(0)}(f)(x)
&= \sum_{y: \varphi^{-1}(y) = x} f(y)
= \sum_{y: \sigma(y) = x} f(y)
= \sigma_\ast(f)(x).
\end{align*}
It follows from Remark \ref{rmk:rhobracket} that
\begin{align*}
[\sigma_\ast]
= H_0(\rho) [\widetilde{\beta}] H_0(\eta)
= H_0(\rho) [\widetilde{\beta}] H_0(\rho)^{-1}
\end{align*}
Therefore, the diagram commutes.
\end{proof}

\chapter{Main result}
\label{section:result}

Our main result is Theorem \ref{thm:AFEDR}, where we give a condition on $\sigma: X \rightarrow X$ that makes $C^*(\mathcal{G})$ AF embeddable. The condition is
\begin{align}
\label{eqn:maincondition}
C^*(\mathcal{G}) \text{ is AFE }
\Leftrightarrow
\newc{\mathrm{Im}}(\sigma_\ast - \mathrm{id}) \cap C_c(X, \mathbb{N}) = \lbrace 0 \rbrace.
\end{align}

We prove this condition by applying the commutative diagram of Chapter \ref{section:diagram} to find a condition which is equivalent to the equation $H_\beta \cap K_0(B) = \lbrace 0 \rbrace$ of Brown's theorem (Theorem \ref{thm:brown} on page \pageref{thm:brown}).

\newc{It follows from the commutative diagram that $C^*(\mathcal{G})$ is \textit{not} AFE if, and only if, there are functions $f \in C_c(X, \mathbb{Z})$, $h \in C_c(X, \mathbb{N})$ such that}
\begin{align*}
\newc{[\sigma_\ast(f) - f] = [h],}
\end{align*}
\newc{with $[h] \neq [0]$. By Lemma \ref{lemma:traceH0Rsigma}, we have that $h \neq 0$. Hence, $f$ is also nonzero. The equation above holds if, and only if, there is an $F \in C_c(c^{-1}(0), \mathbb{Z})$ such that}
\oldc{It \oldc{follow}\newc{follows} from the commutative diagram that $C^*(\mathcal{G})$ is \textit{not} AFE if, and only if, there is an \oldc{$F \in C_c(X, c^{-1}(0))$}\newc{$F \in C_c(c^{-1}(0), \mathbb{Z})$} and if there are nonzero functions $f \in C_c(X, \mathbb{Z})$, $h \in C_c(X, \mathbb{N})$ such that}
\begin{align}
\label{eqn:maincondition2}
\sigma_\ast(f) - f = h + \partial_1 F.
\end{align}

We simplify the condition above by excluding $\partial_1 F$ from the equation. As we will see, for a sufficiently large $n$, we have $\sigma_\ast^n (\partial_1 F) = \partial_1 \sigma_\ast^n (F) = 0$. If we let $\widetilde{f} = \sigma_\ast^n(f)$, $\widetilde{h} = \sigma_\ast^n(h)$ and if we apply $\sigma_\ast^n$ to both sides of \eqref{eqn:maincondition2}, we get
\begin{align}
\label{eqn:maincondition3}
\sigma_\ast(\widetilde{f}) - \widetilde{f} = \widetilde{h}.
\end{align}
With this equation, we obtain the condition \eqref{eqn:maincondition} by negation.

We begin this chapter by defining $\sigma_\ast(F)$ for functions $F \in C_c(c^{-1}(0), \mathbb{Z})$, using the fact that $c^{-1}(0)$ is the increasing union of the sets $R(\sigma^n)$.

\begin{lemma}
\label{lemma:sigmaastRsigma}
Fix $n \in \mathbb{N}$ and let $F \in C_c(R(\sigma^n), \mathbb{Z})$. We define $\sigma_\ast(F) \in C_c(R(\sigma^n), \mathbb{Z})$ by
\begin{align}
\label{eqn:sigmaastRsigma}
 \sigma_\ast(F)(x,y) = \sum_{\substack{u: \hphantom{f} \sigma(u) = x \\ v: \hphantom{f} \sigma(v) = y}} F(u,v).
\end{align}
Then $\sigma_*(F)$ is well-defined. In particular, if $n \geq 1$, we have $\sigma_\ast(F) \in C_c(R(\sigma^{n-1}), \mathbb{Z})$.
\end{lemma}
\begin{proof}
When $n = 0$, this map is well-defined because we identify $R(\sigma^0)$ with $X$ and then apply Definition \ref{def:sigmaast}.

Fix $n \geq 1$. Note that $C_c(R(\sigma^n), \mathbb{Z})$ is spanned by the characteristic functions of the sets
\begin{align*}
\mathcal{U}_{A,B}^{n}  = \lbrace (x,y) : x \in A, y \in B, \sigma^n(x) = \sigma^n(y) \rbrace,
\end{align*}
for $A, B \subset X$ compact open with $\sigma^n(A) = \sigma^n(B)$, and $\sigma^n$ injective on $A$ and $B$. Since the map $F \mapsto \sigma_\ast^n(F)$ is linear, we only need to show that the map is well-defined for these characteristic functions. In particular, we will show that $\sigma_\ast(1_{\mathcal{U}_{A,B}^{n}}) = 1_{\mathcal{U}_{\sigma(A),\sigma(B)}^{n-1}} \in C_c(R(\sigma^{n-1}), \mathbb{Z})$. In fact,
\begin{itemize}
\item Suppose $1_{\mathcal{U}_{\sigma(A),\sigma(B)}^{n-1}}(x,y) = 1$.

Then $(x,y) \in \mathcal{U}_{\sigma(A),\sigma(B)}^{n-1}$.

Then there exists $u \in A$ and $v \in B$ with $\sigma(u) = x$ and $\sigma(v) = y$. Moreover,
\begin{align*}
\sigma^n(u) = \sigma^{n-1}(\sigma(x)) = \sigma^{n-1}(\sigma(y)) = \sigma^n(v).
\end{align*}
Since $\sigma$ is injective on $A$ and $B$, such $(u,v)$ is unique. By equation \eqref{eqn:sigmaastRsigma}, we have that
\begin{align*}
\sigma_\ast(1_{\mathcal{U}_{A,B}^{n}})(x,y) = 1_{\mathcal{U}_{A,B}^{n}}(u,v) = 1.
\end{align*}
\item Suppose $\sigma_\ast(1_{\mathcal{U}_{A,B}^n})(x,y) \neq 0$.

Then, by \eqref{eqn:sigmaastRsigma} there are $u,v$ such that
\begin{align*}
& \sigma(u) = x, \text{\hspace{10pt}}
\sigma(v) = y, \text{\hspace{10pt}}
\sigma^n(u) = \sigma^n(v), \text{\hspace{10pt}}
u \in A,\text{\hspace{10pt}and }
v \in B.
\end{align*}
Since $\sigma$ is injective on $A$ and $B$, it follows that $u, v$ are unique. Thus we have $\sigma_\ast(1_{A,B}^n)(x,y) = 1$. Moreover,
\begin{align*}
&
x \in \sigma(A), \text{\hspace{10pt}}
y \in \sigma(B), \text{\hspace{10pt}}
\sigma^{n-1}(x) = \sigma^{n-1}(y)\\
\Rightarrow\text{\hspace{10pt}}&
(x,y) \in \mathcal{U}_{\sigma(A),\sigma(B)}^{n-1}\\
\Rightarrow\text{\hspace{10pt}}&
1_{\mathcal{U}_{\sigma(A),\sigma(B)}^{n-1}}(x,y) = 1.
\end{align*}
\end{itemize}
Therefore $\sigma_\ast(1_{\mathcal{U}_{A,B}^{n}}) = 1_{\mathcal{U}_{\sigma(A),\sigma(B)}^{n-1}}$.
\end{proof}

\begin{remark}
Note that the definition above coincides with Definition \ref{def:sigmaast} of page \pageref{def:sigmaast} when $F \in C_c(c^{-1}(0), \mathbb{Z})$ has support on the unit space, that we identify with $X$. In this case, we have
for $x \in X$,
\begin{align*}
\oldc{\sigma_\ast F(x)}\newc{\sigma_\ast (F)(x)}
= \sum_{\substack{y: \sigma(y) = x \\ z: \sigma(z) = x}} F(y,z)
= \sum_{y: \sigma(y) = x} F(y).
\end{align*}
\end{remark}

Now we prove that $\sigma_\ast$ and $\partial_1$ commute.

\begin{lemma}
\label{lemma:partialsigmacommute}
Let $F \in C_c(R(\sigma^n), \mathbb{Z})$. Then $\partial_1^{(n)}(\sigma_\ast(F)) = \sigma_\ast(\partial_1^{(n)}(F))$. Here, $\partial_1^{(n)}: C_c(R(\sigma^n), \mathbb{Z}) \rightarrow C_c(X, \mathbb{Z})$ denotes the operator from Lemma \ref{lemma:partial1F}.
\end{lemma}
\begin{proof}
Note that the result holds for $n = 0$ because $\partial_1^{(0)} = 0$. Fix $n \geq 1$ and let $x \in X$. Then
\begin{align*}
\sigma_\ast(\partial^{(n)}_1(F))(x)
&= \sum_{y: \sigma(y) = x} \partial_1^{(n)}(F)(y) \\
&= \sum_{y: \sigma(y) = x} \left(\sum_{z: \sigma^n(z) = \sigma^n(y)} [F(z,y) - F(y,z)] \right)\\
&= \sum_{y: \sigma(y) = x} \left( \sum_{z: \sigma^n(z) = \sigma^{n-1}(x)} [F(z,y) - F(y,z)] \right), \\ 
&= \sum_{y: \sigma(y) = x} \left( \sum_{v: \sigma^{n-1}(v) = \sigma^{n-1}(x)} \left( \sum_{z: \sigma(z) = v} [F(z,y) - F(y,z)] \right) \right) \\
&= \sum_{v: \sigma^{n-1}(v) = \sigma^{n-1}(x)} \left( \sum_{\substack{y: \sigma(y) = x \\ z:\sigma(z) = v}} [F(z,y) - F(y,z)] \right)\\
&= \sum_{v: \sigma^{n-1}(v) = \sigma^{n-1}(x)} [\sigma_\ast(F)(v,x) - \sigma_\ast(F)(x,v)] \\
&= \partial_1^{(n-1)}(\sigma_\ast(F))(x)
\end{align*}
Note that $\sigma_\ast(F) \in C_c(R(\sigma^{n-1}), \mathbb{Z})$ by Lemma \ref{lemma:sigmaastRsigma}. Then $\partial^{(n)}_1(\sigma_\ast(F)) = \partial^{(n-1)}_1(\sigma_\ast(F))$ by Lemma \ref{lemma:partial1F}. Therefore $\sigma_\ast(\partial_1^{(n)}(F)) = \partial_1^{(n)}(\sigma_\ast(F))$.
\end{proof}


Now we prove our main result. The following theorem finds a condition on $\sigma$ that characterises when $C^*(\mathcal{G})$ is AFE.




\begin{customthm}{1}
\newc{Let $\sigma$ be a surjective local homeomorphism on a locally compact, Hausdorff, second countable, totally disconnected space $X$. Denote by $\mathcal{G}$ be the Deaconu-Renault groupoid corresponding to $\sigma$. Then the following are equivalent:}
\begin{enumerate} [(i)]
\item $C^*(\mathcal{G})$ is AF embeddable,
\item $C^*(\mathcal{G})$ is quasidiagonal,
\item $C^*(\mathcal{G})$ is stably finite,
\item $\mathrm{Im}(\sigma_\ast - \mathrm{id}) \cap C_c(X, \mathbb{N}) = \lbrace 0 \rbrace$,
\end{enumerate}
\newc{where the map $\sigma_\ast: C_c(X, \mathbb{Z}) \rightarrow C_c(X, \mathbb{Z})$ is defined by}
\begin{align*}
\newc{\sigma_\ast(f)(x) = \sum_{y: \sigma(y) = x} f(y)
\hspace{20pt}
\text{for $x \in X$, $f \in C_c(X, \mathbb{Z})$.}}
\end{align*}
\end{customthm}
\begin{proof}
Let $c: \mathcal{G} \rightarrow \mathbb{Z}$ be the continuous cocycle given by $c(x,k,y) = k$ for $(x,k,y) \in \mathcal{G}$. By \cite[Corollary 5.2]{FKPS}, we have that $\mathcal{G}(c)$ is an AF groupoid. Let $B = C^*(\mathcal{G}(c))$. Then $C^*(\mathcal{G}(c))$ is AF.

Let $\beta$ be the action on $B$ from Section \ref{subsection:crossedproducts:beta}. Brown's theorem (Theorem \ref{thm:brown}) gives the equivalence of the\oldc{for} conditions below.
\begin{enumerate}[(i')]
\item $B \rtimes_\beta \mathbb{Z}$  is AFE,
\item $B \rtimes_\beta \mathbb{Z}$ is quasidiagonal,
\item $B \rtimes_\beta \mathbb{Z}$ is stably finite,
\item $H_{\beta} \cap K_0(B)^+ = \lbrace 0 \rbrace$.
\end{enumerate}

By Proposition \ref{prop:beta}, $B \rtimes_\beta \mathbb{Z}$ and $C^*(\mathcal{G}) \rtimes_{\alpha^c} \mathbb{T} \rtimes_{\widehat{\alpha}^c} \mathbb{Z}$ are isomorphic. By Takai duality, $C^*(\mathcal{G}) \rtimes_{\alpha^c} \mathbb{T} \rtimes_{\widehat{\alpha}^c} \mathbb{Z}$ and $C^*(\mathcal{G})$ are stably isomorphic. So $B \rtimes_\beta \mathbb{Z}$ and $C^*(\mathcal{G})$ are stably isomorphic.

Since AF embeddability, quasidiagonality and stable finiteness are preserved under stable isomorphisms, we have the equivalence below:
\begin{enumerate}[(i)]
\item $C^*(\mathcal{G})$  is AFE,
\item $C^*(\mathcal{G})$ is quasidiagonal,
\item $C^*(\mathcal{G})$ is stably finite,
\item[\refstepcounter{enumi}(iv')] $H_{\beta} \cap K_0(B)^+ = \lbrace 0 \rbrace$.
\end{enumerate}

Now we study the condition (iv'). By the commutative diagram \newc{\eqref{eqn:diagram}}, (iv') is false if, and only if, there are \newc{non-zero} $f \in C_c(X, \mathbb{Z})$ and $h \in C_c(X, \mathbb{N})$ such that
\begin{align}
\label{eqn:AFEDR1}
[\sigma_\ast(f) - f] = [h].
\end{align}
Suppose \eqref{eqn:AFEDR1} holds. Then there exists $F \in C_c(c^{-1}(0), \mathbb{Z})$ such that
\begin{align*}
\sigma_\ast(f) - f = h + \partial_1 F.
\end{align*}
By the topological properties of $c^{-1}(0)$, there exists $n \geq 0$ such that $F \in C_c(R(\sigma^n), \mathbb{Z})$. By Lemma \ref{lemma:partial1F}, we have
\begin{align}
\label{eqn:AFEDR2}
\sigma_\ast(f) - f = h + \partial_1^{(n)} F.
\end{align}
Recall from Lemma \ref{lemma:sigmaastRsigma} that $\sigma_\ast^n(F) \in C_c(X, \mathbb{Z})$. Then $\partial_1(\sigma_\ast^n(F)) = 0$.

Let $\widetilde{f} = \sigma_\ast^n(f)$ and $\widetilde{h} = \oldc{\sigma_\ast(h)^n}\newc{\sigma_\ast^n(h)} \in C_c(X, \mathbb{N})$. \newc{Since $h$ is non-negative, we have that $\widetilde{h} \neq 0$ by definition of $\sigma_\ast^n(h)$.} By applying $\sigma_\ast^n$ to both sides to equation \ref{eqn:AFEDR2}, we obtain
\begin{align*}
\sigma_\ast(\widetilde{f}) - \widetilde{f}
&= \widetilde{h} + \sigma_\ast^n(\partial_1^{(n)} F) \\
&= \widetilde{h} + \partial_1^{(n)}(\sigma_\ast^n( F))
\hspace{15pt}\text{by Lemma \ref{lemma:partialsigmacommute}}\\
&= \widetilde{h}.
\end{align*}
We have just proved that if $C^*(\mathcal{G})$ is not AFE, then $\mathrm{Im}(\sigma_\ast - \mathrm{id}) \cap C_c(X, \mathbb{N}) \neq \lbrace 0 \rbrace$. We will prove the converse.

Now suppose that $\mathrm{Im}(\sigma_\ast - \mathrm{id}) \cap C_c(X, \mathbb{N}) \neq \lbrace 0 \rbrace$. Then there are $f \in C_c(X, \mathbb{N})$ and $h \in C_c(X, \mathbb{N})$ such that \newc{$h \neq 0$ and}
\begin{align*}
\sigma_\ast(f) - f = h.
\end{align*}
\newc{It follows from Lemma \ref{lemma:traceH0Rsigma} that $[h] \neq 0$.} Then equation (iv') holds, which implies that $C^*(\mathcal{G})$ is not AFE.
\end{proof}

\begin{corollary}
\label{corollary:sigmanNotAFE}
Let $X$ be a locally compact Hausdorff second countable totally disconnected space. Let $\sigma: X \rightarrow X$ be a surjective local homeomorphism. Let $\mathcal{G}$ be the Deaconu-Renault groupoid. Then the following are equivalent:
\begin{enumerate}[(1)]
\item $C^*(\mathcal{G})$ is AFE,
\item $\mathrm{Im}(\sigma_\ast - \mathrm{id}) \cap C_c(X, \mathbb{N}) = \lbrace 0 \rbrace$,
\item $\mathrm{Im}(\sigma_\ast^n - \mathrm{id}) \cap C_c(X, \mathbb{N}) = \lbrace 0 \rbrace$ for some $n \geq 1$, and
\item $\mathrm{Im}(\sigma_\ast^n - \mathrm{id}) \cap C_c(X, \mathbb{N}) = \lbrace 0 \rbrace$ for all $n \geq 1$.
\end{enumerate}
\end{corollary}
\begin{proof}
The equivalence $(1) \Leftrightarrow (2)$ follows from Theorem \ref{thm:AFEDR}. The implications $(4) \Rightarrow (2) \Rightarrow (3)$ are straightforward. We show $(3) \oldc{\rightarrow}\newc{\Rightarrow} (4)$ by contrapositive.

Suppose (4) is false. Then, for some $m \geq 1$, there are nonzero $f \in C_c(X, \mathbb{Z})$ and $h \in C_c(X, \mathbb{N})$ such that $\sigma_\ast^m(f) - f = h$. Let $n \geq 1$ be arbitrary. Then, for all $k = 0, 1, \dots, n$, we have
\begin{align*}
\sigma_\ast^k(\sigma_\ast^m(f)) - \sigma_\ast^k(f) = \oldc{\sigma_\ast^*(f)}\newc{\sigma_\ast^k(h)}.
\end{align*}

Let $F = \sigma_\ast^m(f)$ and $H = h + \sigma_\ast(h) + \dots + \sigma_\ast^{n-1}(h)$. Note that $H \in C_c(X, \mathbb{N})$ and is nonzero. By a telescoping sum, we have
\begin{align*}
\sigma_\ast^n(F) - F
&= \oldc{\sigma_\ast^n(\sigma_\ast(f))}\newc{\sigma_\ast^n(\sigma_\ast^m(f))} - \sigma_\ast^m(f) \\
&= \sum_{k=0}^{n-1} [ \sigma_\ast^{k+1}(\sigma_\ast^m(f)) - \oldc{\sigma_\ast^k(\sigma_\ast^n(f))}\newc{\sigma_\ast^k(\sigma_\ast^m(f))}] \\
&= \sum_{k=0}^{n-1} \sigma_\ast^k(h) \\
&= H.
\end{align*}
Thus $\mathrm{Im}(\sigma_\ast^n - \mathrm{id}) \cap C_c(X, \mathbb{N}) \neq \lbrace 0 \rbrace$. Since $n$ is arbitrary, $(3)$ is false. Therefore, $(3) \Rightarrow (4)$.
\end{proof}

\begin{remark}
\label{rmk:sigmakgeq}
Using the same telescoping sums of the corollary before, we can show that, for $f \in C_c(X, \mathbb{Z})$ satisfying $\sigma_\ast(f) \geq f$, we have $\sigma_\ast^n(f) \geq \sigma_\ast^{n-1}(f) \geq \hdots \geq f$ for all $n \geq 1$.  This property will be useful in Corollary \ref{corollary:graphAFE} on page \pageref{corollary:graphAFE}, where we study the AFE property of graph algebras as an example of our main theorem.
\end{remark}

\chapter{Examples}
\label{section:examples}

In this chapter we show that Theorem \ref{thm:AFEDR} generalises two known results in the literature. Under certain conditions, we see in Section \ref{subsection:examples:graphs} that Schafhauser's result on graph algebras \cite{Schafhauser-AFEgraph} is an application of this theorem.

Our main result also generalises the study of the crossed product $C(X) \rtimes_\sigma \mathbb{Z}$ by Brown \cite[Theorem 11.5]{Brown-quasidiagonal} (or \cite[Theorem 9]{Pimsner} by Pimsner). We prove this fact in Section \ref{subsection:examples:crossedproduct}.

\section{Graph algebras}
\label{subsection:examples:graphs}

Before we apply our main result to graph algebras, we define graphs. We omit the definition of graph algebras here, since we do not need to know its definition to apply our main result. See the book \cite{Raeburn} by Reaburn for an introduction to graph algebras.


A \newterm{directed graph} (or simply a graph) denotes an $E = (E^{0}, E^{1}, r, s)$ equipped with two countable sets $E^0, E^1$, and two functions $r, s: E^1 \rightarrow E^0$. The elements of $E^0$ are called \newterm{vertices}, and the elements of $E^1$ are called edges. We say that $E$ is row-finite if $r^{-1}(v)$ is finite for every vertex $e$. A vertex $v$ such that $r^{-1}(v) = \emptyset$ is called a \newterm{source} of the graph, and a vertex $w$ such that \oldc{$s^{-1} = \emptyset$}\newc{$s^{-1}(w) = \emptyset$} is called a \newterm{sink}.

For $n \geq 2$, a finite sequence of edges $\mu = \mu_1 \mu_2 \dots \mu_n$ is denoted a \newterm{path} of length $n$ if $s(\mu_i) = r(\mu_{i+1})$ for $i= 1, \dots, n-1$. In this case, define $r(\mu) = r(\mu_1)$ and $s(\mu) = s(\mu_n)$. We denote by $E^n$ the set of all paths of length $n$, and we let $E^\ast = \bigcup_{n=0}^\infty E^n$ be the set of all finite paths. Every vertex $v$ is finite path of length $0$ with $r(v) = s(v) = v$.

Given two finite paths $\mu \in E^m$ and $\nu \in E^n$ such that $s(\mu) = r(\nu)$, we define the finite path $\mu \nu \in E^{m + n}$ as follows: if both $m, n \geq 1$, then
\begin{align*}
\mu\nu = \mu_1 \dots \mu_m \nu_1 \dots \nu_n.
\end{align*}
If $m = 0$, we let $\mu \nu = \nu$. If $n = 0$, then we define $\mu \nu = \mu$.

Similarly, an infinite sequence of edges $x = x_1 x_2 \dots$ is an infinite path if $s(x_i) = r(x_{i+1})$ for all $i \geq 1$. We let $r(x) = r(x_1)$. \newc{We denote by $E^\infty$ the set of all infinite paths. Moreover, for a finite path $\mu$ of length $n \geq 1$, we let $Z(\mu)$ denote the set of infinite paths $x = x_1 x_2 \dots$ such that $x_1 \dots x_n = \mu$. Similarly, for $v \in E^{(0)}$, $Z(v)$ denotes the set of infinite paths with range $v$. The sets of the form $Z(\mu)$ generate a topology on $E^\infty$. See \cite{Raeburn} for more details.}

For a finite path $\mu \in E^*$ such that $s(\mu) = r(x)$, we define the infinite path $\mu x$ in a similar way to the definition of $\mu \nu$ above.



We apply Theorem \ref{thm:AFEDR} to study the following theorem:

\begin{theorem} \cite{Schafhauser-AFEgraph}
\label{thm:SchafhauserAFEgraph}
For a countable graph $E$, the following are equivalent:
\begin{enumerate}[(i)]
\item $C^*(E)$ is AFE,
\item $C^*(E)$ is quasidiagonal,
\item $C^*(E)$ is stably finite,
\item $C^*(E)$ is finite,
\item No cycle in $E$ has an entrance.
\end{enumerate}
\end{theorem}

As we see below, under certain conditions, a graph algebra is the C*-algebra of a Deaconu-Renault groupoid.
\begin{remark}
\label{rmk:CEisoCGE}
Let $E$ be a row-finite directed graph with no sources. By \cite[Theorem 4.2]{KPRR}, there exists an isomorphism $C^*(E) \cong C^*(G_E)$ given by $1_e \mapsto 1_{Z(e,r(e))}$ , where
\begin{align*}
G_E = \lbrace (x_0 z, \vert x_0 \vert - \vert y_0 \vert , y_0 z) : x_0, y_0 \in E^*, z \in E^\infty \rbrace,
\end{align*}
and $Z(e, s(e)) = \lbrace (e x, 1, x) : x \in E^\infty, r(x) = s(e) \rbrace$.
The set $G_E$ is an example of the Deaconu-Renault, where map $\sigma$ is given by
\begin{align*}
\sigma(x_1 x_2 \hdots) = x_2 x_3 \hdots.
\end{align*}
This groupoid is almost identical to the groupoid defined in \cite[Definition 2.3]{KPRR} by Kumjian, Pask, Raeburn and Renault. The difference between their groupoid and $G_E$ is that here we are considering a different notation, which is compatible with Raeburn's book \cite{Raeburn}. In \cite{KPRR}, two neighbouring edges $x_i, x_{i+1}$ in a path satisfies $r(x_i) = s(x_{i+1})$, instead of the equality $s(x_i) = r(x_{i+1})$ that consider here. Also, the groupoid they use has elements of the form $(x, -k, y)$ for all $(x,k,y)$ in the groupoid $G_E$.
\end{remark}

We want to use the isomorphism $C^*(E) \cong C^*(G_E)$, so we fix here $E$ to be a row-finite directed graph with no \oldc{sinks}\newc{sources}. Note also that the map $\sigma$ is surjective if, and only if, $E$ has no sinks. Since we want $\sigma$ to be surjective in order to apply Theorem \ref{thm:AFEDR}, we assume in this section that $E$ has no sinks.

The following lemma studies the elements of $E^\infty$ when condition (v) of Theorem \ref{thm:SchafhauserAFEgraph} is false.


\begin{lemma}
\label{lemma:EAFEpaths}
Let $E$ be a row-finite graph with no sources nor sinks. Suppose that no cycle in $E$ has an entrance. Given $x \in E^\infty$, then either
\begin{enumerate}
\item all edges of $x$ are distinct, or
\item there are unique $n \geq 0, p \geq 1$ such that
    \begin{itemize}
    \item $\sigma^{n+p}(x) = \sigma^n(x)$ and
    \item $x_1, \hdots, x_{n+p}$ are distinct.
    \end{itemize}
\end{enumerate}
Given $x \in E^\infty$ and $i, j$ such that $x_i \neq x_j$, then $r(x_i) \neq r(x_j)$. Moreover, let $y \in E^\infty$ be such that $\sigma(y) = x$. Then (1) holds for $x$ if, and only if, it holds for $y$ too.
\end{lemma}
\begin{proof}
Let $x \in E^\infty$. If all edges of $x$ are distinct, we have that $x_{n+p+1} \neq x_{n+1}$ for all $n \geq 0, p \geq 1$. Then (2) is false.

Now suppose that (1) is false. Then $x$ has repeated edges. Let $n \geq 0$ be the smallest number such that $x_{n+1}$ appears at least twice on $x$. Let $p$ be the smallest number $p \geq 1$ with $x_{n+p+1} =  x_{n+1}$. Then $x_1, \hdots, x_{n+p}$ are distinct.

We show that, for $x \in E^\infty$, two different edges have different ranges. Suppose this is false. Then we can choose $i < j$ such that $x_i \neq x_j$, $r(x_i), \hdots, r(x_{j-1})$ are distinct, and $r(x_i) = r(x_j)$. Let $\mu = x_i \hdots x_{j-1}$. Then $\mu$ is a cycle with an entrance, which is a contradiction. Therefore, distinct edges in $x$ must have different ranges.


We claim that $\sigma^{n+p}(x) = \sigma^n(x)$. We will prove this by induction, showing that, for all $k \geq 1$,
\begin{enumerate}[(a)]
\item $x_{n+p+k} = x_{n+k}$, and
\item $x_{n+k}, \hdots, x_{n+p+k-1}$ are distinct.
\end{enumerate}
This already holds when $k = 1$. Assume this is satisfied for an arbitrary $k \geq 1$.

Suppose $x_{n+p+k+1} \neq x_{n+k+1}$. Let $\mu = x_{n+k} x_{n+k+1} \hdots x_{n+p+k-1}$. Then $\mu$ is a cycle by items (a) and (b). Also,
\begin{align*}
r(x_{n+p+k+1}) = s(x_{n+p+k}) = s(x_{n+k}) = r(x_{n+k+1}).
\end{align*}
This implies that $\mu$ is a cycle with an entrance. This is a contradiction. Therefore, we must have $x_{n+p+k+1} = x_{n+k+1}$. Since item (b) holds for $k$, we have that $$x_{n+k+1}, \hdots, x_{n+p+k-1}, x_{n+p+k} = x_{n+k}$$ are distinct elements. Therefore, by induction, (2) holds for $x$.

Now let $y \in E^\infty$ be such that $\sigma(y) = x$. Assume assume that (1) holds for $x$. Suppose that $y$ has repeated edges. Then there are $n \geq 0, p \geq 1$ such that $y_1, \hdots, y_{n+p}$ are distinct and $\sigma^{n+p}(y) = \sigma^n(y)$. This implies that
\begin{align*}
x_{n+p+1} = y_{n+p+2} = y_{n_2} = x_{n+1}.
\end{align*}
Then $x$ has repeated edges, which is a contradiction. We have just proved that, if (1) holds for $x$, it also holds for $y$.

It is straightforward that, if $y$ has distinct edges, so does $x = \sigma(y)$.
\end{proof}

In order to prove that $C^*(E)$ not AFE implies that $E$ has a cycle with an entrance, we will need to understand $\sigma_\ast(f)$ when $E$ has no cycle with an entrance. In order to do that, we need to understand the paths in $\sigma^{-1}(x)$ for $x$ in each of the two cases of Lemma \ref{lemma:EAFEpaths}. The previous lemma studied the case when $x$ has distinct edges. Now we will consider some paths with repeated edges.

\begin{lemma}
\label{lemma:cycleexit}
Let $E$ be a row-finite graph with no sources nor sinks. Suppose that no cycle in $E$ has an entrance. Let $x \in E^\infty$ be such that condition (2) of Lemma \ref{lemma:EAFEpaths} holds for $n, p \geq 1$ (note that $n \neq 0$ \newc{is assumed}). Given $y \in E^\infty$ with $\sigma(y) = x$, then it satisfies condition (2) of that lemma for $n+1$ and $p$.
\end{lemma}
\begin{proof}
By Lemma \ref{lemma:EAFEpaths}, $y$ does not have distinct edges. Also
\begin{align*}
\sigma^{n+1+p}(y) = \sigma^{n+p}(x) = \sigma^n(x) = \sigma^{n+1}(y). 
\end{align*}
Then we only need to show that $y_1, \hdots, y_{n+p+1}$ are distinct. Since $\sigma(y) = x$, we need to prove that $y_1, x_1, \hdots, x_{n+p}$ are distinct.

Suppose that this is false. Then, there exists $1 \leq l \leq n + p$ such that $y_1 = x_l$. \oldc{Assume this is false. Then there exists $l$ such that $y_1 = x_l$.} Let us study different cases:
\begin{itemize}
\item Suppose $1 \leq l \leq n$.
\oldc{If $l > 1$, let $\mu = x_1 \hdots x_{l-1} y_1$. Otherwise, let $\mu = y_1$. In both cases, $\mu$ is a cycle. Also,} \newc{Then}
\begin{align*}
r(x_{l+1}) = s(x_l) = s(y_1) = r(x_1).
\end{align*}
Since $l +1 \leq n + p$, this contradicts the fact that $r(x_1), \hdots, r(x_{n+p})$ are distinct.

\item If $l = n + p$, then
\begin{align*}
r(x_1) = s(y_1) = s(x_{n+p}) = r(x_{n+1}).
\end{align*}
This implies that $r(x_1), \hdots, r(x_{n+p})$ are not distinct, which is a contradiction.

\item If $n + 1 \leq l \leq n + p -1$, then
\begin{align*}
r(x_1) = s(y_1) = s(x_l) = r(x_{l+1}).
\end{align*}
This also leads to a contradiction.
\end{itemize}
Therefore, $r(y_1), r(x_1), \hdots, r(x_{n+p})$ are distinct and $y$ satisfies condition (2) of Lemma \ref{lemma:EAFEpaths} for $n+1$ and $p$.
\end{proof}

\begin{lemma}
\label{lemma:sigmapyx}
Let $E$ be a row-finite graph without sinks or sources and such that no cycle has an entrance. Let $x \in E^\infty$ satisfy condition (2) of Lemma \ref{lemma:EAFEpaths} for $n = 0$, $p \geq 1$. If $\sigma^p(y) = x$ and $y$ satisfies condition (2) of the same lemma for $n = 0$, then $y = x$.
\end{lemma}
\begin{proof}
Suppose $y$ satisfies condition (2) of Lemma \ref{lemma:EAFEpaths} for $n_1 = 0$ and $p_1 \geq 1$. Then $\sigma^{p_1}(y) = y$. By applying $\sigma^p$ on both sides of the equation, we get $\sigma^{p_1}(y) = x$. The properties of $x$ imply that $p_1$ is a multiple of $p$. However, since $\sigma^p(y) = x$ and $y_1, \hdots, y_{p_1}$ are distinct, we have that $p_1 = p$. Therefore, $y = \sigma^p(y) = x$.
\end{proof}

The following lemmas study functions in $C_c(E^\infty, \mathbb{Z})$, where $E$ has no cycle with an entrance. These results will be used to prove that $C^*(E)$ is AFE.

\begin{lemma}
\label{lemma:Feventuallyzero}
Let $E$ be a row-finite graph without sinks or sources and such that no cycle has an entrance. Let $x \in E^\infty$ have distinct edges. Given $F \in C_c(E^\infty, \mathbb{Z})$, there exists $n \geq 1$ such that $F(\sigma^k(x)) = 0$ for all $k \geq n$.
\end{lemma}
\begin{proof}
Suppose this is false for $x$. Then we can find an strictly increasing sequence $k_1, k_2, \hdots$ of positive integers such that $F(\sigma^{k_i}(x)) \neq 0$ for all $i \geq 1$. Note that $\lbrace Z(v) \rbrace_{v \in E^0}$ is a cover of $\mathrm{supp}\hphantom{.} F$ by open sets. We will show that it is not possible to choose a finite subcover, by proving that there are infinitely many distinct vertices $v_i$ such that $Z(v_i) \cap \mathrm{supp}\hphantom{.} F \neq \emptyset$.

Given $i \geq 1$, let $v_i = r(\sigma^{k_i}(x))$. Since $F(\sigma^{k_i}(x)) \neq 0$, we have that $Z(v_i) \cap \mathrm{supp}\hphantom{.} F \neq \emptyset$. By Lemma \ref{lemma:EAFEpaths}, the open sets $Z(v_i)$ are distinct. Therefore $F$ does not have a finite subcover by subsets in $\lbrace Z(v) \rbrace_{v \in E^0}$. This implies that $F$ is not compactly supported, which is a contradiction. Therefore, there is an $n \geq 1$ such that $F(\sigma^k(x)) = 0$ for all $k \geq n$.
\end{proof}

By a similar argument, we can prove the following:
\begin{lemma}
\label{lemma:Feventuallyzerobackwards}
Let $E$ be a row-finite graph without sinks or sources and such that no cycle has an entrance. Let $\lbrace x^{(k)} \rbrace_{k \geq 1}$ be a sequence in $E^\infty$ such that
\begin{itemize}
\item $\sigma(x^{(k+1)}) = x^{(k)}$,
\item $r(x^{(1)}), r(x^{(2)}), \hdots$ are distinct.
\end{itemize}
Then there exists $n \geq 1$ such that $F(x^{(k)}) = 0$ for all $k \geq n$.
\end{lemma}
In order to prove this lemma, just use the same strategy of Lemma \ref{lemma:Feventuallyzero}, defining $v_i = r(x^{(k_i)})$.

We need the next lemma to prove that $C^*(E)$ is not AFE when the graph has a cycle with an entrance.
\begin{lemma}
\label{lemma:sigmapath}
Let $E$ be a row-finite directed graph with no sinks and no sources. Let $\mu \in E^*$. Then
\begin{align*}
\sigma_\ast^{\vert \mu \vert + 1}(1_{Z(\mu)}) = \sum_{\substack{e \in E^1 \\ r(e) = s(\mu)}} 1_{Z(s(e))}.
\end{align*}
\end{lemma}
\begin{proof}
Let $x \in E^\infty$. Then
\begin{align*}
\sigma_\ast^{\vert \mu \vert + 1}(1_{Z(\mu)})(x)
= \sum_{y: \sigma^{\vert \mu \vert + 1}(y) = x} 1_{Z(\mu)}(y).
\end{align*}
If $1_{Z(\mu)}(y) = 1$, we must have $y = \mu e x$ for some $e \in E^1$. Then
\begin{align*}
\sigma_\ast^{\vert \mu \vert + 1}(1_{Z(\mu)})(x) = \sum_{\substack{e \in E^1 \\ r(e) = s(\mu) \\ s(e) = r(x)}} 1
= \sum_{\substack{e \in E^1 \\ r(e) = s(\mu)}} 1_{Z(s(e))}(x).
\tag*{\qedhere}
\end{align*}
\end{proof}

Now we show that Schafhauser's theorem \cite{Schafhauser-AFEgraph} is an application of our Theorem \ref{thm:AFEDR}, when $E$ has no sinks nor sources.

\begin{corollary}
\label{corollary:graphAFE}
Let $E$ be a row-finite graph with no sinks and no sources. Then the following are equivalent:
\begin{enumerate}[(i)]
\item $C^*(E)$ is AFE,
\item $C^*(E)$ is quasidiagonal,
\item $C^*(E)$ is stably finite,
\item $C^*(E)$ is finite,
\item No cycle in $E$ has an entrance.
\end{enumerate}
\end{corollary}
\begin{proof}
Note that $C^*(E) \cong C^*(G_E)$ by Remark \ref{rmk:CEisoCGE}. We will prove that $C^*(E)$ is not AFE if, and only if, $E$ has a cycle with an entrance. First we prove the implication $(\Rightarrow)$. So, suppose that $C^*(E)$ is not AFE. By Theorem \ref{thm:AFEDR}, there are two non-zero functions $f \in C_c(E^\infty, \mathbb{Z})$, $h \in C_c(E^\infty, \mathbb{N})$ such that $\sigma_\ast(f) - f = h$.

Assume that no cycle in $E$ has an entrance. With this assumption, we will find a contradiction by showing that $h = 0$. In order to do this, we will consider the different types of infinite paths $x$ given by Lemma \ref{lemma:EAFEpaths}:

\begin{enumerate}[(a)]
\item If $x$ has distinct edges, then $h(x) = 0$

Let $x \in E^\infty$ have distinct edges. Suppose that $f(x) > 0$. Since $\sigma_\ast(f)(x) \geq f(x)$, there exists $x^{(2)} \in E^\infty$ with distinct edges such that $\sigma(x^{(2)}) = x$ and $f(x^{(2)}) > 0$. By repeating this argument, we can find a sequence of infinite paths $x^{(k)}$ such that $x^{(1)} = x$ and, for all $k \geq 1$, $\sigma(x^{(k+1)}) = x^{(k)}$ and $f(x^{(k)}) > 0$.

Since $x$ has distinct edges, Lemma  \ref{lemma:Feventuallyzerobackwards} implies that there exists $k$ such that $f(x^{(k)}) = 0$, which is a contradiction. Therefore, we must have $f(x) \leq 0$. In particular, Lemma \oldc{\ref{lemma:cycleexit} implies that $f(y) \leq 0$ for all $y$ with $\sigma(y) = x$.}\newc{\ref{thm:SchafhauserAFEgraph} implies that $y$ has distinct edges. Thus $f(y) \leq 0$ for all $y$ with $\sigma(y) = x$.} 

Now suppose that $f(x) < 0$. Then
\begin{align*}
f(\sigma(x))
\leq \sigma_\ast(f)(\sigma(x))
= \sum_{y: \sigma(y) = \sigma(x)} f(y)
\leq f(x),
\end{align*}
since $f(y) \leq 0$ for all $y$ with $\sigma(y) = x$, as have proved. By repeating this argument, we have that, for all $k \geq 1$, $f(\sigma^k(x)) \leq f(x) < 0$. This contradicts Lemma \ref{lemma:Feventuallyzero}.

Therefore, $f(x) = 0$ for all $x \in E^\infty$ with distinct edges. Applying the definition of $\sigma_\ast$ and that all $y$ with $\sigma(y) = x$ have distinct edges, we also have that $\sigma_\ast(f)(x) = 0$. Therefore $h(x) = 0$.

\item If $x$ satisfies condition (2) of Lemma \ref{lemma:EAFEpaths} for $n, p \geq 1$, then $f(x) \leq 0$

Here we will apply the arguments of item (a). Let $x \in E^\infty$ have repeated edges and be such that condition (2) of Lemma \ref{lemma:EAFEpaths} holds for \oldc{$n \geq 1$, $p$}\newc{$n,p \geq 1$}. Moreover, assume that $f(x) > 0$. Since $\sigma_\ast(f)(x) \geq f(x)$, there exists $x^{(2)} \in E^\infty$ with $\sigma(x^{(2)}) = x$ and $f(x^{(2)}) > 0$. By Lemma \ref{lemma:cycleexit}, we have $r(x^{(2)}) \neq r(x^{(1)})$ and $x^{(2)}$ satisfies condition (2) of Lemma \ref{lemma:EAFEpaths} for $n+1$, $p$.

By repeating this argument, we can find a sequence $x^{(k)}$ in $E^\infty$ such that $x^{(1)} = x$, $\sigma(x^{(k+1)}) = x^{(k)}$, the vertices $r(x^{(1)}), r(x^{(2)}), \hdots$ are distinct, and $f(x^{(k)}) > 0$ for all $k \geq 1$. This contradicts Lemma \ref{lemma:Feventuallyzerobackwards}. Therefore, $f(x) \leq 0$ for all $x$ satisfying condition (2) of Lemma \ref{lemma:EAFEpaths} for $n, p \geq 1$.

\item If $x$ satisfies condition (2) of Lemma \ref{lemma:EAFEpaths} for $n = 0, p \geq 1$, then $h(x) = 0$ and $\sigma_\ast^p(f)(x) = f(x)$

Let $x \in E^\infty$ be such that condition (2) of Lemma \ref{lemma:EAFEpaths} holds for $n = 0, p \geq 1$. Then $\sigma^p(x) = x$ and
\begin{align*}
\sigma_\ast^p(f)(x)
&= \sum_{y: \sigma^p(y) = x} f(y) \\
&= \sum_{\substack{y: \sigma^p(y) = x \\ y \neq x}} f(y) + f(x) \\
&\leq f(x).
\end{align*}
by Lemma \ref{lemma:sigmapyx} and item (b). 
Then, by Remark \ref{rmk:sigmakgeq},
\begin{align*}
f(x) \leq \sigma_\ast(f)(x) \leq \sigma_\ast^2(f)(x) \leq \hdots \leq \sigma_\ast^p(f)(x) \leq f(x).
\end{align*}
This implies that $h(x) = \sigma_\ast(f)(x) - f(x) = 0$.

\item If $x$ satisfies condition (2) of Lemma  \ref{lemma:EAFEpaths} for $n = 1, p \geq 1$, then $f(x) = 0$

\oldc{
Note that $\sigma(x)$ satisfies condition (2) of that lemma for $n_1 = 0$ and $p$. Then, by item (c),
\begin{align*}
f(\sigma(x))
&= \sigma_\ast^p(f)(\sigma(x)) \\
&= \sum_{y: \sigma^p(y) = \sigma(x)} f(y) \\
&= f(\sigma(x)) + \sum_{\substack{y: \sigma^p(y) = \sigma(x) \\ y \neq \sigma(x)}} f(y).
\end{align*}
Note that $\sigma(x)$ is the unique element in the sum satisfying the condition (2) of Lemma  \ref{lemma:EAFEpaths} for $n = 0$. Then, by item (b)
\begin{align*}
f(\sigma(x)) = f(\sigma(x)) + \sum_{\substack{y: \sigma^p(y) = \sigma(x) \\ y \neq \sigma(x)}} f(y) \geq f(\sigma(x)).
\end{align*}
Then, by item (b), we have that $f(x) = 0$.
}
\newc{
Let $z_0 = \sigma^n(x)$ and, for every $k \geq 1$, let $z_k = \sigma^{kn(2p-1)}(z_0)$. Note that $\sigma^p(z_0) = z_0$. Moreover, $z_0$ satisfies condition (2) of Lemma \ref{lemma:EAFEpaths} for $0$ and $p$. We show that $\sigma^n(z_{k+1}) = z_k$. Indeed,
\begin{align*}
\sigma^n(z_{k+1})
&= \sigma^n(\sigma^{(k+1)n(2p-1)}(z_0)) \\
&= \sigma^{kn(2p-1)+n(2p-1) + n}(z_0)   \\
&= \sigma^{kn(2p-1) + 2pn}(z_0) \\
&= \sigma^{kn(2p-1)}(\sigma^{2pn}(z_0)) \\
&= \sigma^{kn(2p-1)}(z_0) \\
&= z_k.
\end{align*}
It follows from Lemma \ref{lemma:EAFEpaths} that $z_k$ satisfies condition (2) of Lemma \ref{lemma:EAFEpaths} for $0$ and $p$. For each $k$, we have
\begin{align}
\label{eqn:fzk}
f(z_k)
\leq \sigma_\ast^n(f)(z_k)
= \sum_{y: \sigma^n(y) = z_k} f(y)
= f(z_{k+1}) + \sum_{\substack{y: \sigma^n(y) = z_k \\ y \neq z_{k+1}}} f(y).
\end{align}
But if $y \neq z_{k+1}$ and $\sigma^n(y) = z_k$, then Lemma \ref{lemma:sigmapyx} implies that $y$ satisfies condition (2) of Lemma \ref{lemma:EAFEpaths} for $n$ and $p$. So $f(y) \leq 0$ by item (b) of this lemma. Thus $f(z_k) \leq f(z_{k+1})$.
}

\newc{Note that $z_p = z_0$. In fact,
\begin{align*}
z_p
= \sigma^{pn(2p-1)}(z_0)
= \sigma^{pn(2p-1)}(\sigma^n(x))
= \sigma^n(x)
= z_0.
\end{align*}
This implies that $f(z_0) = f(z_1) = \dots = f(z_p)$. By \eqref{eqn:fzk}, $f(y) = 0$ for all $y \neq z_1$ such that $\sigma^n(y) = z_0 = \sigma^n(x)$. Since $z_1$ satisfies condition (2) of Lemma \ref{lemma:EAFEpaths} for $0$ and $p$, we have that $x \neq z_1$. Therefore $f(x) = 0$.
}
\item For every $x \in E^\infty$ with repeated edges, we have $h(x) = 0$ 

Let $x \in E^\infty$ have repeated edges. Then $x$ satisfies condition (2) of Lemma \ref{lemma:EAFEpaths} for $n \geq 0$ and $p \geq 1$. We will show by induction that $f(x) = 0$ for $n \geq 1$. Note that item (d) proves this for $n = 1$.

Now, suppose that $f$ is zero for an arbitrary $n \geq 1$. Let $x$ satisfy condition (2) of Lemma \ref{lemma:EAFEpaths} for $n+1$ and $p$. Then, $\sigma(x)$ satisfies that condition for $n$ and $p$. Also,
\begin{align*}
0
&\geq \sum_{y: \sigma(y) = \sigma(x)} f(y)
\hspace{15pt}\text{by item (b) and Lemma \ref{lemma:cycleexit}}\\
&= \sigma_\ast(f)(\sigma(x))\\
&\geq f(\sigma(x)) \\
&= 0
\hspace{15pt}\text{by hypothesis}.
\end{align*}
This implies that $\sum_{y: \sigma(y) = \sigma(x)} f(y) = 0$. Then item (b) implies that $f(x) = 0$. By induction $f(x) = 0$ for any $x$ satisfying condition (2) for $n \geq 1$. Then,
by Lemma \ref{lemma:cycleexit}
\begin{align*}
h(x) = \sigma_\ast(f)(x) - f(x) = \sum_{y: \sigma(y) = x} f(y) - f(x) = 0.
\end{align*}
Note that the case $n = 0$ was proved in item (c). Therefore, $h(x) = 0$ for every $x$ with repeated edges.
\end{enumerate}
Let $x \in E^\infty$. In item (a), we showed that $h(x) = 0$ if $x$ has distinct edges; in item (e), we proved that $h(x) = 0$ if $x$ has repeated edges. By Lemma \ref{lemma:EAFEpaths}, this implies that $h = 0$, which is a contradiction. Therefore, $E$ has a cycle with an entrance.

Finally, we show $(\Leftarrow)$. Suppose $E$ has a cycle with an entrance. Then there is a finite path $\mu \in E^*$ with $s(\mu) = r(\mu)$, and there exists $e_1 \in E^1$ such that $r(e_1) = s(\mu)$ and $e_1 \neq \mu_{1}$. Since $E$ has no sinks, all cycles in $E$ must have length greater than zero. Then $\vert \mu \vert \geq 1$. 

Note that $0 \leq \sigma_\ast(1_{Z(\mu)}) \leq 1_{Z(s(\mu_1))}$. In fact, $\sigma_\ast(1_{Z(\mu)}) \geq 0$ by definition of $\sigma_\ast$. Let $x \in E^\infty$. Then
\begin{align*}
\sigma_\ast(1_{Z(\mu)})(x)
&= \sum_{y: \sigma(y) = x} 1_{Z(\mu)}(y) \\
&= \begin{cases}
1_{Z(\mu)}(\mu_1 x) & \text{if } r(x) =s(\mu_1), \\
0 & \text{otherwise.}
\end{cases} \\
& \leq 1_{Z(s(\mu_1))}(x).
\end{align*}

By Lemma \ref{lemma:sigmapath}, we have
\begin{align*}
\sigma_\ast^{\vert \mu \vert + 1}(1_{Z(\mu)}) &= \sum_{\substack{e \in E^1 \\ r(e) = s(\mu)}} 1_{Z(s(e))}
 \geq 1_{Z(s(\mu_1))} + 1_{Z(s(e_1))} 
 \geq \sigma_\ast(1_{Z(\mu)}) + 1_{Z(s(e_1))}.
\end{align*}
Then there exists a non-zero function $h \in C_c(E^\infty, \mathbb{N})$ such that
\begin{align*}
\sigma_\ast^{\vert \mu \vert}(\sigma_\ast(1_{Z(\mu)})) - \sigma_\ast(1_{Z(\mu)}) = h.
\end{align*}
By \oldc{Theorem \ref{thm:AFEDR}}\newc{Corollary \ref{corollary:sigmanNotAFE}}, this implies that $C^*(E) \cong C^*(G_E)$ is not AFE.

So far we have proved (i) $\Leftrightarrow$ (v). By Theorem \ref{thm:AFEDR}, we have the equivalence (i) $\Leftrightarrow$ (ii) $\Leftrightarrow$ (iii) $\Leftrightarrow$ (v). The implication (iii) $\Rightarrow$ (iv) is straightforward from the definition of stably finite C*-algebras (see Definition 5.1.1 and Lemma 5.1.2 of \cite{RLL-Ktheory}). We prove the (iv) $\Rightarrow$ (v) by copying the argument from \cite[Theorem 1.1]{Schafhauser-AFEgraph}.

Let $\mu, \nu$ be distinct paths such that $r(\mu) = s(\mu) = r(\nu)$. By the isomorphism $C^*(G_E) \cong C^*(E)$, we identify the following elements of graph algebras with elements of $C^*(G_E)$: 
\begin{align*}
s_\mu = 1_{\mathcal{U}_{Z(\mu), Z(s(\mu))}^{\vert \mu \vert, 0}},
\hspace{15pt}
p_{r(\mu)} = 1_{Z(r(\mu))}
\hspace{15pt}
\text{and}\hspace{10pt}
p_{s(\mu)} = 1_{Z(s(\mu))}.
\end{align*}
Then
\begin{align*}
s_\mu^* s_\mu = p_{r(\mu)}
\hspace{15pt}\text{and}\hspace{15pt}
s_\mu s_\mu^* < s_\mu s_\mu^* + s_\nu s_\nu^* \leq p_{r(\mu)}.
\end{align*}
This implies that $p_{r(\mu)}$ is an infinite projection and, therefore, $C^*(E)$ s infinite.
\end{proof}

Corollary \ref{corollary:graphAFE} also helps us to study the AF embedabillity for C*-algebras of graphs that are not row-finite. In \cite[Theorem 2.11]{DrinenTomforde}, Drinen and Tomforde show that the C*-algebra of an arbitrary graph $E$ is Morita equivalent to the C*-algebra of a row-finite graph $F$ with no sources\footnote{In their paper, they say that $F$ has no sinks instead. But here we are assuming the notation for paths that is compatible with Raeburn's book \cite{Raeburn}, as explained before.}, called the \newterm{desingularisation} of $E$. Since the AFE property is preserved under Morita equivalence, we have that $C^*(E)$ is AFE if, and only if, $C^*(F)$ is AFE.

\section{The crossed product $C(X) \rtimes_\sigma \mathbb{Z}$}
\label{subsection:examples:crossedproduct}

Now we apply Theorem \ref{thm:AFEDR} to show that our result generalises Pimsner's theorem \cite[Theorem 9]{Pimsner} (See also \cite[Theorem 11.5]{Brown-quasidiagonal} by Brown). There he assumes $X$ is a compact metric space, while here we assume the particular case where $X$ is locally compact, Hausdorff, second countable, and totally disconnected.

\begin{remark}
In \cite{Deaconu}, Deaconu shows that, for a compact Hausdorff space with a homeomorphism $\sigma:X \rightarrow X$, $C(X) \rtimes_\sigma \mathbb{Z}$ is isomorphic to $C^*(\mathcal{G})$. We will use this isomorphism to study when $C(X) \rtimes_\sigma \mathbb{Z}$ is AFE.
\end{remark}


\begin{definition}
We say that $\sigma$ \newterm{compresses} an open set $\mathcal{U} \subset X$ if $\sigma(\mathcal{U}) \subsetneq \mathcal{U}$. The homeomorphism $\sigma$ \newterm{compresses no \newc{compact open} sets} if the following condition holds for all $\mathcal{U} \subset X$ \newc{compact} open:
\begin{align*}
\sigma({\mathcal{U}}) \subset \mathcal{U} \Rightarrow \sigma(\mathcal{U}) = \mathcal{U}.
\end{align*}
\end{definition}

Note that $\sigma$ compresses no \newc{compact open} sets if no \newc{compact} open subset of $X$ is compressed by $\sigma$.

\begin{remark}
In \cite{Brown-quasidiagonal}, $\sigma$ compresses no \newc{compact open} sets if, for all $\mathcal{U} \subset X$ open with $\sigma(\overline{\mathcal{U}}) \subset \mathcal{U}$, we have $\sigma(\mathcal{U}) = \mathcal{U}$. However, \oldc{we are assuming here that $X$ is totally disconnected. Then every } \newc{under the assumptions of the corollary, every open subset of $X$ is also closed.}
\end{remark}

\newc{This corollary extends \cite[Theorem 9]{Pimsner}  and \cite[Theorem 11.5]{Brown-quasidiagonal}.}

\begin{corollary}
\label{corollary:sigmacompresses}
Let $X$ be a locally compact, Hausdorff, second countable, totally disconnected space. Then the following are equivalent:
\begin{enumerate}[(i)]
\item $\oldc{C(X)}\newc{C_0(X)} \rtimes_\sigma \mathbb{Z}$ is AFE,
\item $\oldc{C(X)}\newc{C_0(X)} \rtimes_\sigma \mathbb{Z}$ is quasidiagonal,
\item $\oldc{C(X)}\newc{C_0(X)} \rtimes_\sigma \mathbb{Z}$ is stably finite
\item $\sigma$ compresses no \newc{compact open} sets.
\end{enumerate}
\end{corollary}
\begin{proof}
Conditions (i)-(iii) are equivalent by Theorem \ref{thm:AFEDR}. Here we show the equivalence $(iv) \Leftrightarrow (i)$. We show that $\oldc{C(X)}\newc{C_0(X)} \rtimes_\sigma \mathbb{Z}$ is not AFE if, and only if, $\sigma$ compresses a \newc{compact} open subset of $X$.

Suppose $C_0(X) \rtimes_\sigma \mathbb{Z}$ is not AFE. By Theorem \ref{thm:AFEDR}, there are non-zero functions $f \in C_c(X, \mathbb{Z})$ and $h \in C_c(X, \mathbb{N})$ such that \begin{align}
\label{eqn:sigmacompresses1}
\sigma_\ast(f) - f = h.
\end{align}
Since $\sigma$ is a homeomorphism, the definition of $\sigma_\ast$ implies that $\sigma_\ast(f) = f \circ \sigma^{-1}$. By applying $\sigma$ to both sides of \eqref{eqn:sigmacompresses1}, we have
\begin{align}
\label{eqn:sigmacompresses2}
f - f \circ \sigma = h \circ \sigma.
\end{align}
Note that $h \circ \sigma \in C_c(X, \mathbb{N})$ is non-zero, and $f \geq f \circ \sigma$.

\newc{For every integer $n \geq 0$, define the function $H_n = \sum_{i=0}^n h \circ \sigma^i$. then $H_n \in C_c(X, \mathbb{N})$ is nonzero. Similarly, define $H(x)$ for every $x \in X$ by}
\begin{align*}
\newc{H(x) = \sum_{k=0}^\infty h \circ \sigma^k(x).}
\end{align*}
\newc{We claim that the equation above defines a continuous function $H: X \rightarrow \mathbb{N}$. It is straightforward that $H$ assumes values in $\mathbb{N} \cup \lbrace \infty \rbrace$. Suppose that there exists an $x \in X$ such that $H(x) = \infty$. Then there exists an sequence of indices $k_1 > k_2 > \dots$ such that}
\begin{align*}
\newc{f \circ \sigma^{k_i}(x) > f \circ \sigma^{k_{i+1}}(x).}
\end{align*}
\newc{This implies that $f$ is not bounded, which is a contradiction. Therefore $H(X) \subset \mathbb{N}$, and then $H$ is a well-defined function.}

\newc{Now we show that $H$ is continuous. Let $M$ be the maximum of $h$. Then $0 \leq h \circ \sigma^i(x) \leq M$ for all $x \in X$ and $i \geq 1$. Fix $x \in X$. We apply the continuity of $h$ and $\sigma$, and the donimated convergence theorem to obtain}
\begin{align*}
\newc{
\lim_{u \rightarrow x} H(u)
= \lim_{u \rightarrow x} \sum_{n=0}^\infty h \circ \sigma^n(u)
=  \sum_{n=0}^\infty \lim_{u \rightarrow x} h \circ \sigma^n(u)
=  \sum_{n=0}^\infty h \circ \sigma^n(x)
= H(x).}
\end{align*}
\newc{Therefore $H$ is continuous.}

\newc{The function $H$ is nonzero. So choose an $x$ with $H(x) \geq 1$. Before we find a compact open set that is compressed by $\sigma$, we show that $f(x) \leq 1$. Suppose that $f(x) \geq 1$. We claim that there is an $j_0$ such that}
\begin{align*}
\newc{h(\sigma^{-j}(x)) = 0 \hspace{15pt} \text{for all $j \geq j_0$.}}
\end{align*}
\newc{Assume this is false, then the sequence of $\sigma^{-j}(x)$ has a subsequence $\sigma^{-j_l}(x)$ such that $h(\sigma^{-j_l}(x)) > 0$ for all $l \geq 1$. Then $f(\sigma^{-j_{l+1}})(x) > f(\sigma^{-j_{l}})(x)$. This implies that $f$ is not bounded, which is a contradiction.}

\newc{Let $x_0 = \sigma^{-j_0}(x)$. Then $f(x_0) \geq f(x)$ and $H(x_0) \geq H(x)$. Define the set}
\begin{align*}
\newc{ A = \lbrace u \in X : f(u) \geq 1 \hspace{10pt}\text{and}\hspace{10pt} H(u) \geq 1 \rbrace. }
\end{align*}
\newc{Then $A$ is compact and $x_0 \in A$. For every nonzero integer $n$, define the set}
\begin{align*}
\newc{K_n = \lbrace u : f(x) \geq 1 \hspace{10pt}\text{and}\hspace{10pt} H_n(u) = 0 \rbrace}.
\end{align*}
\newc{Then $K_n$ is compact and $\sigma^{-n}(x_0) \in A \cap K_n$. Note that $K_{n+1} \subset K_n$. By the finite intersection property, there exists an $x_\infty$ such that $x_\infty \in A \cap K_n$ for all $n$. Then $H_n(x_\infty) = 0$ for all $n$, which implies that $H(x_\infty) = 0$. However, $H_n(x_\infty) \geq 1$ because $x_\infty \in A$.  This is a contradiction. Therefore, $f(x) \leq 0$.}

\newc{Let $n \geq 0$ be such that $h \circ \sigma^{n+1}(x) > 0$. By applying the equation \eqref{eqn:sigmacompresses2} several times, we have $f(\sigma^n(x)) < f(x) \leq 0$. We make an abuse of notation and replace $\sigma^n(x)$ by $x$. Then}
\begin{align*}
\newc{f(x) < 0,\hspace{10pt}
H(x) > 0, \hspace{10pt}\text{and}\hspace{10pt} h(\sigma(x)) > 0.}
\end{align*}
\newc{Define the sets}
\begin{align*}
\newc{\mathcal{U}}&= \newc{\lbrace u : f(u) = f(x), H(\sigma(u)) = H(\sigma(x)), h \circ \sigma(u) = h \circ \sigma(x) \rbrace, \hspace{10pt}\text{and}} \\
\newc{V} &= \newc{\lbrace u : f(u) < 0, H(\sigma(u)) < H(\sigma(x)) \rbrace.}
\end{align*}
\newc{Both sets are compact open. The set $\mathcal{U}$ contains $x$. We will show that $\sigma(\mathcal{U}) \subset V$. Let $u \in \mathcal{U}$. Then}
\begin{align*}
\newc{f(\sigma(u)) = f(u) - h(\sigma(u)) = f(x) - h(\sigma(x)) < f(x) < 0,}
\end{align*}
\newc{and}
\begin{align*}
\newc{H(\sigma^2(u)) = H(\sigma(u)) - h(\sigma(u)) = H(\sigma(x)) - h(\sigma(x)) < H(\sigma(x)).}
\end{align*}
\newc{Then $\sigma(u) \in V$. As a consequence, $V \neq \emptyset$.}

\newc{Now we show that $\sigma(V) \subset V$. Let $u \in V$. Then}
\begin{align*}
\newc{f(\sigma(u)) = f(u) - h(\sigma(u)) \leq f(u) < 0,}
\end{align*}
\newc{and}
\begin{align*}
\newc{H(\sigma^2(u)) = H(\sigma(u)) - h(\sigma(u))
\leq H(\sigma(u)) < H(\sigma(x)).}
\end{align*}
\newc{Then $\sigma(u) \in V$. Therefore, we have a compact open set $\mathcal{U} \cup V$ is such that $\sigma(\mathcal{U} \cup V) \subset V \subsetneq \mathcal{U} \cup V$. This implies that $\sigma$ compresses $\mathcal{U} \cup V$.}

Conversely, suppose that $\sigma$ compresses a \newc{compact} open subset $\mathcal{U} \subset X$. Let $f = - 1_\mathcal{U}$. Then, by definition of $\sigma_\ast$,
\begin{align*}
\sigma_\ast(f) - f
= - 1_\mathcal{U} \circ \sigma^{-1} + 1_\mathcal{U}
= 1_\mathcal{U} - 1_{\sigma(\mathcal{U})}.
\end{align*}
Since $\mathcal{U} \subsetneq \sigma(\mathcal{U})$, the function $\newc{h = 1_{\mathcal{U} \setminus \sigma(\mathcal{U})}} \in C_c(X, \mathbb{N})$ is nonzero and satisfies $\sigma_\ast(f) - f = h$. Therefore, $C(X) \rtimes_\sigma \mathbb{Z}$ is not AFE.
\end{proof}

\section{Topological graph algebras}
\label{subsection:examples:topologicalgraphs}

In this section we prove that Theorem \ref{thm:AFEDR} is an application of Theorem 6.7 of \cite{Schafhauser-topologicalgraphs} by Schafhauser for topological graph algebras. We define topological graphs as introduced by Katsura in \cite{KatsuraI} and then studied in a series of papers \cite{KatsuraII, KatsuraIII, KatsuraIV}. Here we will use the definition from \cite{Schafhauser-topologicalgraphs} that uses $s$ instead of $d$, and that assumes the spaces $E^0, E^1$ are Hausdorff.

\begin{definition}
A \newterm{topological graph} $E = (E^0, E^1, s, r)$ consists of two locally compact Hausdorff spaces, and two maps $r, s: E^1 \rightarrow E^0$, where $s$ is a local homeomorphism and $r$ is continuous. A $v \in E^{0}$ is called a \newterm{sink} if $s^{-1}(v) = \emptyset$. We say that $E$ is compact if $E^0, E^1$ are compact and $s$ is surjective.
\end{definition}

Note that every graph is an example of topological graph. Now we define $\varepsilon$-pseudopaths as in Schafhauser's paper.

\begin{definition}
Let $E$ be a topological graph, and let $\varepsilon > 0$. Let $d$ be a metric on $E^0$ compatible with the topology. An \newterm{$\varepsilon$-pseudopath} in $E$ is a finite sequence $(e_1, e_2, \dots, e_n)$ in $E^1$ such that for each $i=1, \dots, n-1$, $d(s(e_i), r(e_{i+1})) < \varepsilon$. We write $r(\alpha) = r(e_1)$ and $s(\alpha) = s(e_n)$. An $\varepsilon$-pseudopath $\alpha$ is called an \newterm{$\varepsilon$-pseudoloop} based at \oldc{$s(\alpha)$}\newc{$r(\alpha)$} if $d(r(\alpha), s(\alpha)) < \varepsilon$.
\end{definition}

Note that we make a slight change of notation by writing the $\alpha$ as $(e_1, \dots, e_n)$ instead of the notation $(e_n, \dots, e_1)$ from \cite{Schafhauser-topologicalgraphs}. We do this to make our notation compatible with the notation we used for directed graphs. Indeed, if $E$ is a directed graph, than every path is an $\varepsilon$-path in our notation. \newc{We also write that a $\varepsilon$-pseudoloop $\alpha$ is based at $r(\alpha)$ instead of $s(\alpha)$.}

Now we state the theorem that characterises the AFE property of topological graph algebras. Here we write $r$ instead of $s$ in item 5 to agree with our notation.

\begin{theorem}
\label{thm:Schafhausertopologicalgraph}
\cite[Theorem 6.7]{Schafhauser-topologicalgraphs} If $E$ is a compact topological graph with no sinks, then the following are equivalent:
\begin{enumerate}[(i)]
\item $C^*(E)$ is AF embeddable;
\item $C^*(E)$ is quasidiagonal;
\item $C^*(E)$ is stably finite;
\item $C^*(E)$ is finite;
\item $r$ is injective and for every $v \in E^0$ and $\varepsilon > 0$, there exists an $\varepsilon$-pseudoloop in $E$ based at $v$.
\end{enumerate}
\end{theorem}

Note that for a Deaconu-Renault groupoid $\mathcal{G}$ corresponding to $\sigma: X \rightarrow X$, then $E = (E^0, E^1, s, r) := (X,X, \sigma, \mathrm{id})$ is a topological graph with $C^*(E) \cong C^*(\mathcal{G})$ by \cite[Proposition 10.9]{KatsuraII}. By the theorem above, if $X$ is compact, then the C*-algebra $C^*(\mathcal{G})$ is AFE if, and only, $C^*(\mathcal{G})$ is finite. 

\newc{We show that for the topological graph corresponding to a Deaconu-Renault groupoid with compact unit space, the equivalence (i)-(iii) and (v) of Schafhauser's \cite[Theorem 6.7]{Schafhauser-topologicalgraphs} holds.} \oldc{We show that for the topological graph corresponding to $\mathcal{G}$, Theorem \ref{thm:AFEDR} is an application of the theorem above. Note that we do not assume that $X$ is compact.}\oldc{ We do not study the finiteness of $C^*(\mathcal{G})$ since Schafhauser's proof of the implication (iv) $\Rightarrow$ (v) of \cite[Theorem 6.7]{Schafhauser-topologicalgraphs} depends on the compactness of the topological graph.}

\begin{corollary}
\label{corollary:Schafhausersapp}
Let $E = (E^0, E^1, s, r) := (X,X, \sigma, \mathrm{id})$ be the topological graph corresponding to $\mathcal{G}$. Then the following are equivalent:
\begin{enumerate}[(i)]
\item $C^*(\mathcal{G})$ is AF embeddable;
\item $C^*(\mathcal{G})$ is quasidiagonal;
\item $C^*(\mathcal{G})$ is stably finite;
\item $\mathrm{Im}(\sigma_\ast - \mathrm{id}) \cap C_c(X, \mathbb{N}) = \lbrace 0 \rbrace$.
\end{enumerate}
\newc{Moreover, consider the following statement:}
\begin{enumerate}[(i)]
  \setcounter{enumi}{4}
\item for every $x \in X$ and $\varepsilon > 0$, there is an $\varepsilon$-pseudoloop in $E$ based at $x$.
\end{enumerate}
\newc{Then $(v) \Rightarrow (iv)$. If $X$ is compact, then all items above are equivalent.}
\end{corollary}
\begin{proof}
Theorem \ref{thm:AFEDR} gives the equivalence (i) $\Leftrightarrow$ (ii) $\Leftrightarrow$ (iii) $\Leftrightarrow$ (iv). 

We show (v) $\Rightarrow$ (iv) by proving that not (iv) $\Rightarrow$ not (v). Suppose (iv) is false. \newc{Then there are nonzero functions $f \in C_c(X, \mathbb{Z}), h \in C_c(X, \mathbb{N})$ such that $\sigma_\ast(f) - f = h$. Since $\sigma_\ast(f)$ and $f$ are continuous and compactly supported functions assuming values in $\mathbb{Z}$, and since $\sigma$ is a local homeomorphism, then there are disjoint compact open subsets $V_1, \dots, V_n$, there is a number $1 \leq m \leq n$, and a surjective function $\rho: \lbrace 1, \dots, n \rbrace \rightarrow \lbrace 1, \dots, m \rbrace$ such that:
\begin{itemize}
\item $\sigma$ is injective on each $V_i$,
\item $\sigma(V_i) = V_{\rho(i)}$,
\item There are integers $a_1, \dots, a_n$ and nonzero integers $b_1, \dots, b_m$ such that
\begin{align*}
\sigma_\ast(f) = \sum_{i=1}^n a_i 1_{V_i}
\hspace{15pt}\text{and}\hspace{15pt}
f = \sum_{i=1}^n b_i 1_{V_i}.
\end{align*}
\end{itemize}
Note that, for every $j = 1, \dots, m$,}
\begin{align*}
\newc{a_j = \sum_{i: \rho(i) = j} b_i.}
\end{align*}
\newc{In fact, let $x \in V_j$. Then}
\begin{align*}
\newc{
a_j = \sigma_\ast(f)(x) 
= \sum_{u : \sigma(u) = x} f(u) 
= \sum_{u : \sigma(u) = x} \sum_{i=1}^m  b_i 1_{V_i}(u)
= \sum_{i: \rho(i) = j}^n b_i.}
\end{align*}
\newc{We claim that $n > m$. Suppose this is false. Then $\rho$ is a bijection with $a_j = b_{\rho^{-1}(j)}$. Since $h$ is nonzero, then there exists a $l$ such that $a_j > b_j$. This implies that}
\begin{align*}
 & \newc{\sum_{i = 1}^n a_i > \sum_{i = 1}^n b_i} \\
\newc{\Rightarrow}& \newc{\sum_{i = 1}^n b_{\rho^{-1}(j)} > \sum_{i = 1}^n b_i} \\
\newc{\Rightarrow}& \newc{\sum_{i = 1}^n b_i > \sum_{i = 1}^n b_i,}
\end{align*}
\newc{which is a contradiction. Therefore, $n > m$.}

\newc{Note that $a_i > 0$ for all $i$. Let $x \in V_n$. Then, for every $l > 0$, we have that $\sigma^l(x) \in V_1 \cup \dots \cup V_m$. Let $\varepsilon > 0$ be such that $B(x, \varepsilon) \subset V_n$. This implies that there is no $\varepsilon$-pseudoloop in $E$ based at $x$. Therefore, (v) is false.}

\newc{Now assume that $X$ is compact. We}  prove (iv) $\Rightarrow$ (v) by showing not (v) $\Rightarrow$ not (iv). Suppose (v) is false. Then there are $x \in X$ and $\varepsilon > 0$ such that there exists no $\varepsilon$-pseudoloop in $E$ based at $x$. \newc{By continuity, there exists a set $V \subset B(x, \varepsilon)$, that is an open neighbourhood of $x$ such that}
\begin{align*}
d(\sigma(u), \sigma(x)) < \varepsilon
\hspace{15pt}
\text{for all $u \in V$}.
\end{align*}
Note that
\begin{align}
\label{eqn:sigmaiVcapVempty}
\sigma^i(V) \cap V = \emptyset
\hspace{10pt}\text{for all $i \geq 1$.}
\end{align}
Indeed, suppose this is false for some $i$. Then let $u \in \sigma^i(V) \cap V$, \newc{and define $\alpha = (x, \sigma(u), \sigma^2(u), \dots, \sigma^{i-1}(u))$. Then $\alpha$ is $\varepsilon$-pseudopath. Since $\sigma^i(u) \in V \subset B(x, \varepsilon)$, then $d(\sigma^i(u), x) < \varepsilon$, which implies that $\alpha$ is an $\varepsilon$-pseudoloop based at $x$, which is a contradiction.}

\newc{Let $\Omega = \bigcup_{i=0}^\infty \sigma^i(V)$. Then this set is compact open because $X$ is compact. By \eqref{eqn:sigmaiVcapVempty}, we have}
\begin{align*}
\sigma(\Omega) \subset \Omega
\hspace{15pt}\text{and}\hspace{15pt}
\sigma(\Omega) \cap V =\emptyset.
\end{align*}
So $\sigma(\Omega) \subsetneq \Omega$. Let \newc{$f = 1_{\sigma(\Omega)}$, and define $h = \sigma_\ast(f) - f$. Then, for all $y \in X$, we have}
\begin{align*}
\newc{\sigma_\ast(f)(y)
= \sum_{u: \sigma(u) = y} 1_{\sigma(\Omega)}(u)
\geq 1_\Omega(y)
\geq 1_{\sigma(\Omega)}(y)
= f(y).}
\end{align*}
\newc{This implies that $h \in C_c(X, \mathbb{N})$. Moreover,}
\begin{align*}
\newc{h = \sigma_\ast(f) - f \geq 1_\Omega - 1_{\sigma(\Omega)} = 1_{\Omega \setminus \sigma(\Omega)} \neq 0.}
\end{align*}
Therefore, (iv) is false.
\end{proof}

\newc{\begin{remark}
When $X$ is not compact, the implication $(iv) \Rightarrow (v)$ might not be true. For example, consider $X = \mathbb{Z}$ and define $\sigma: X \rightarrow X$ by $\sigma(n) = n+1$. Then $C^*(\mathcal{G})$ is AF embeddable, but there are no $\varepsilon$-pseudoloops for $0 < \varepsilon \leq 1/2$.
\end{remark}}

\begin{remark}
Note that for graph algebras and for C*-algebras of compact topological graphs, the C*-algebra is AF embeddable if, and only if, it is finite. We do not know if this equivalence holds for C*-algebras of Deaconu-Renault groupoids. We believe this is an open question.
\end{remark}

\appendix
\chapter{Multiplier algebras}


We study multiplier algebras, since we need them to study some properties of crossed products. We use results from \cite[Section 2.1]{Murphy}.

\begin{definition}
A \newterm{multiplier} of a C*-algebra $A$ is a pair $(L,R)$ of bounded linear maps of $A$ into $A$ such that
\begin{align*}
L(a)b = L(ab), \text{\hspace{10pt}}
aR(b) = R(ab), \text{\hspace{10pt}and\hspace{10pt}}
R(a)b = aL(b), \text{ for } a,b \in A.
\end{align*}
\end{definition}

\begin{example}
Let $A$ be a C*-algebra, and fix $c \in A$. Then define $L_c, R_c: A \rightarrow A$ by
\begin{align*}
L_c(a) = ca
\hspace{10pt}
\text{and}
\hspace{10pt}
R_c(a) = ac.
\end{align*}
Note that these two operators multiply elements of $A$ by a constant $c$. We claim that $(L_c, R_c)$ is a multiplier. In fact, $L_c$, $R_c$ are linear and bounded operators with norm $\Vert c \Vert$. Also, for $a, b \in A$, we have
\begin{align*}
L_c(a)b = cab =L_c(ab),
\hspace{5pt}
aR_c(b) = abc = R_c(ab)
\hspace{5pt}\text{ and }
R_c(a)b = acb = aL_c(b).
\end{align*}
Thus $(L_c, R_c)$ is a multiplier of $A$.
\end{example}

\begin{lemma}
\label{lemma:conditionmultiplier}
Suppose $(L,R)$ is a pair of maps of $A$ into $A$ such that $R(a)b = aL(b)$ for all $a, b \in A$. Then $(L,R)$ is a multiplier of $A$: the maps $L$ and $R$ are necessarily linear and bounded, and satisfy $L(a)b = L(ab)$ and $aR(b) = R(ab)$. Further, we then have $\Vert L \Vert = \Vert R \Vert$ for the operator norm on $B(A)$.
\end{lemma}

Now we equip the set of multipliers with operations that make it a C*-algebra.

\begin{definition}
Let $A$ be a C*-algebra. Given $T \in B(A)$, we define $T^\sharp \in B(A)$ by $T^\sharp(a) = T(a^*)^*$.
\end{definition}

\begin{lemma}
Let $A$ be a C*-algebra. For $a \in A$, let $(L_a, R_a)$ be a multiplier of $A$ defined by $L_a(b) = ab$ and $R_a(b) = ba$. Then
\begin{enumerate}[1)]
\item $(L_{\lambda a + \mu b}, R_{\lambda a + \mu b}) = (\lambda L_a + \mu L_b, \lambda R_a + \mu R_b)$,
\item $(L_{ab}, R_{ab}) = (L_a L_b, R_b R_a)$,
\item $(L_{a^*},R_{a^*}) = (R_a^\sharp, L_a^\sharp)$, and
\item $\Vert a \Vert = \Vert L_a \Vert = \Vert R_a \Vert$.
\end{enumerate}
If $A$ has an identity, then every multiplier of $A$ has the form $(L_a, R_a)$.
\end{lemma}

The following proposition shows that, for a C*-algebra $A$, the set of multipliers \oldc{of }is a C*-algebra containing $A$ in some sense.

\begin{proposition}
\label{prop:multiplieralgebra}
Let $A$ be a C*-algebra. Then with the operations
\begin{enumerate}[1)]
\item $\lambda(L_1, R_1) + \mu(L_2, R_2) = (\lambda L_1 + \mu L_2, \lambda R_1 + \mu R_2)$,
\item $(L_1, R_1)(L_2, R_2) = (L_1 L_2, R_2 R_1)$,
\item $(L, R)^* = (R^\sharp, L^\sharp)$, and
\item $\Vert (L, R) \Vert = \Vert L \Vert = \Vert R \Vert$,
\end{enumerate}
the set $M(A)$ of multipliers of $A$ is a C*-algebra with identity $1_{M(A)} = (id, id)$, called the \newterm{multiplier algebra} of $A$. The map $\iota_A : a \mapsto (L_a, R_a)$ is an isometric $\ast$-isomorphism of $A$ onto a closed ideal of in $M(A)$. 
If $A$ has an identity, then $\iota_A$ is an isomorphism of $A$ onto $M(A)$.
\end{proposition}

\begin{remark}
In order to avoid using the $(L,R)$ in proofs concerning multipliers, given $\rho = (L, R) \in M(A)$ and $a \in A$, we use the following notation in this text
\begin{align*}
\rho \cdot a = L(a)
\hspace{10pt}\text{and}\hspace{10pt}
a \cdot \rho = R(a).
\end{align*}
\end{remark}

\begin{theorem}
Suppose $\pi$ is a nondegenerate representation of a C*-algebra $A$ on a Hilbert space $H$. Then there is a unique representation $\bar{\pi}: M(A) \rightarrow B(H)$ such that
\begin{align*}
\bar{\pi}(m)\pi(a) = \pi(ma)
\text{\hspace{10pt} for $a \in A$  and $m \in M(A)$.}
\end{align*}
\end{theorem}

\begin{remark}
Given $\pi$, $a$ and $m$ as in the previous lemma, we also have the equation $
\pi(am) = \pi(a)\bar{\pi}(m)$. In fact,
\begin{align*}
\pi(am)
= [\pi((am)^*)]^*
= [\pi(m^* a^*)]^*
= [\bar{\pi}(m^*) \pi(a^*)]^*
= \pi(a) \bar{\pi}(m).
\end{align*}
\end{remark}

\begin{lemma}
\label{lemma:equalmultipliers}
Let $(L, R), (P,Q)$ be two multipliers of a C*-algebra $A$. Let $A'$ be a dense subset of $A$. 
\begin{enumerate}[(i)]
\item If $La' = Pa'$ for all $a' \in A'$, then $(L, R) = (P, Q)$.
\item If $Ra' = \oldc{Qa}\newc{Qa'}$ for all $a' \in A'$, then $(L, R) = (P, Q)$.
\end{enumerate}
\end{lemma}
\begin{proof}
We only show item (i) because the proof of (ii) is analogous. Suppose that $La' = Pa'$ for all $a' \in A'$. By continuity, we have $L = P$. Let $a, b \in A$. Since $(L, R), (P, Q)$ are multipliers, we have
\begin{align*}
R(a) b = a L(b) = a P(b) = Q(a) b.
\end{align*}
Since $b$ is arbitrary, then $R(a) = Q(a)$. But $a$ is also arbitrary, therefore $R = Q$.
\end{proof}

\chapter{Takai duality}
Takai duality is an important result in the theory of crossed products. Let $\alpha: Z \rightarrow \mathrm{Aut}\hphantom{.}A$ by an action of a locally compact abelian group $Z$ by automorphisms on a C*-algebra $A$. Takai duality implies that, up to a stable isomorphism, we can recover a C*-algebra $A$ from the crossed product $A \rtimes_\alpha Z$. The results here are taken from Blackadar's \cite[Sections II.10.1 and II.10.5]{Blackadar} and Williams's \cite[Chapter 7]{Williams-crossedproducts} books.

As we explained in the introduction, Takai duality will be important to understand the AFE property of the Deaconu-Renault groupoid.

\begin{definition}
Suppose $Z$ is a locally compact abelian group. Let $\widehat{Z}$ be the set of all continuous homomorphisms from $Z$ to the circle group $\mathbb{T}$ (these are called the \newterm{characters} of $Z$). Then $\widehat{Z}$ is an abelian group under the pointwise multiplication, and it is a topological group under the topology of uniform convergence on compact sets. $\widehat{Z}$ is called the \newterm{(Pontrjagin) dual group} of $Z$.
\end{definition}

\begin{example}
$\oldc{\widehat{Z}}\newc{\widehat{\mathbb{Z}}} \cong \mathbb{T}$ via $n \mapsto z^n$ for fixed $z \in \mathbb{T}$, and $\widehat{\mathbb{T}} \cong \mathbb{Z}$ via $z \mapsto z^n$ for fixed $n \in \mathbb{Z}$.
\end{example}

\begin{lemma}
\cite[page 190]{Williams-crossedproducts} Let $(A, Z, \alpha)$ be a dynamical system. Then there exists a unique action $\widehat{\alpha}: \widehat{Z} \rightarrow \mathrm{Aut}(A \rtimes_\alpha Z)$ such that, for $f \in C_c(Z, A)$ and $\gamma \in \widehat{Z}$, $\widehat{\alpha}_\gamma(f) \in C_c(Z, A)$ is given by
\begin{align*}
\widehat{\alpha}_\gamma(f)(s) = \overline{\gamma(s)}f(s)
\hspace{15pt}
\text{for }s \in Z.
\end{align*}
Moreover, $(A \rtimes_\alpha Z, \widehat{Z}, \widehat{\alpha})$ is a dynamical system called the \newterm{dual system}, and $\widehat{\alpha}$ is called the \newterm{dual action}.
\end{lemma}

\begin{example}
\label{ex:dualactionT}
If $(A, \mathbb{T}, \alpha)$ is a dynamical system, the dual action $\widehat{\alpha}: \mathbb{Z} \rightarrow \mathrm{Aut}(A \rtimes_\alpha \mathbb{T})$ is given by
\begin{align*}
\widehat{\alpha}_n(f)(z) = z^{-n} f(z)
\hspace{20pt}
\text{for }f \in C(\mathbb{T}, A), z \in \mathbb{T}, n \in \mathbb{Z}.
\end{align*}
\end{example}

\begin{theorem}[Takai duality] \cite[Theorem 7.1]{Williams-crossedproducts}
\label{thm:Takaiduality}
Suppose that $Z$ is an abelian group and that $(A, Z, \alpha)$ is a dynamical system. Then there exists an isomorphism from $(A \rtimes_\alpha Z) \rtimes_{\widehat{\alpha}} \widehat{Z}$ onto $A \otimes \mathcal{K}(L^2(Z))$.
\end{theorem}

Note that the AFE property is preserved under stable isomorphisms. So, if we show that the C*-algebra $A \rtimes_{\alpha} \mathbb{T} \rtimes_{\widehat{\alpha}} \mathbb{Z}$ is AFE, then Takai duality implies that $A$ is also AFE. We will use this idea to understand when the C*-algebra $C^*(\mathcal{G})$ of the Deaconu-Renault is AFE, as explained in the introduction.


\chapter{Inductive limits}
\label{appendix:inductivelimits}

There is a notion of convergence in category theory that describes when a sequence of objects that are connected by morphisms converge to an object called the inductive limit. Here we study inductive limits and some of its properties.

Since category theory is general, the results in this appendix apply to groups, groupoids and C*-algebras. We use Section 6.2 of R\o{}rdam, Larsen  and Laustsen's book on K-theory \cite{RLL-Ktheory} as a reference.

\begin{definition}
\label{def:inductivesequence}
An \newterm{inductive sequence} in a category $\mathscr{C}$ is a sequence $\lbrace A_n \rbrace_{n = 1}^\infty$ of objects in $\mathscr{C}$ and a sequence $\varphi_n: A_n \rightarrow A_{n+1}$ of morphisms in $\mathscr{C}$, usually written
$$
\begin{tikzcd}
A_1 \arrow{r}{\varphi_1}[swap]{} &
A_2 \arrow{r}{\varphi_2}[swap]{} &
A_3 \arrow{r}{\varphi_3}[swap]{} &
\cdots
\end{tikzcd}
$$
For $m > n$ we shall also consider the composed morphisms
\begin{align*}
\varphi_{m,n} = \varphi_{m-1} \circ \varphi_{m-2} \circ \cdots \circ \varphi_n: A_n \rightarrow A_m,
\end{align*}
which, together with the morphisms $\varphi_n$, are called the \newterm{connecting morphisms} (or \newterm{connecting maps}).
\end{definition}

\begin{definition}
\label{def:inductivelimit}
An \newterm{inductive limit} of the inductive sequence
$$
\begin{tikzcd}
A_1 \arrow{r}{\varphi_1}[swap]{} &
A_2 \arrow{r}{\varphi_2}[swap]{} &
A_3 \arrow{r}{\varphi_3}[swap]{} &
\cdots
\end{tikzcd}
$$
in a category $\mathscr{C}$ is a system $(A, \lbrace \mu_n \rbrace_{n=1}^\infty)$, where $A$ is an object in $\mathscr{C}$, $\mu_n: A_n \rightarrow A$ is a morphism in $\mathscr{C}$ for each $n \geq 1$, and where the following two conditions hold.
\begin{enumerate}[(i)]
\item The diagram
$$
\begin{tikzcd}
A_n \arrow{rr}{\varphi_n}[swap]{} \arrow{dr}{}[swap]{\mu_n} &&
A_{n+1} \arrow{dl}{\mu_{n+1}}[swap]{}\\
 & A
\end{tikzcd}
$$
commutes for each positive integer $n$.

\item If $(B, \lbrace \lambda_n \rbrace_{n=1}^\infty)$ is a system, where $B$ is an object in $\mathscr{C}$, $\lambda_n: A_n \rightarrow B$ is a morphism in $\mathscr{C}$ for each $n$, and where $\lambda_n = \lambda_{n+1} \circ \varphi_n$ for all  $n \geq 1$, then there is one and only one morphism $\lambda: A \rightarrow B$ making the diagram
$$
\begin{tikzcd}
& A_n \arrow{dl}{}[swap]{\mu_n} \arrow{dr}{\lambda_n}[swap]{} \\
A \arrow{rr}{}[swap]{\lambda} && B
\end{tikzcd}
$$
commutative for each $n$.
\end{enumerate}
In this text, will call the morphisms $\mu_n$ the \newterm{inclusion morphisms} of $A$
\end{definition}

\begin{remark}
Inductive limits, when they exist, are essentially unique in the sense that if $(A, \lbrace \mu_n \rbrace)$ and $(B, \lbrace \lambda_n \rbrace)$ are inductive limits of the same inductive sequence, then there is an isomorphism $\lambda:A \rightarrow B$ making the diagram of item (ii) in Definition \ref{def:inductivelimit} commutative. See \cite[page 93]{RLL-Ktheory} for details. Because of the essential uniqueness of inductive limits (when they exist), we shall refer to \textit{the} inductive limit (rather than \textit{an} inductive limit).
\end{remark}

\begin{notation}
The inductive limit $(A, \lbrace \mu_n \rbrace)$ is denoted by $\displaystyle\lim_\rightarrow (A_n, \varphi_n)$, or more briefly by $\displaystyle\lim_\rightarrow A_n$.
\end{notation}



\begin{example}
For each $n$, let $A_n = \mathbb{Z}$ and let $\oldc{\varphi}\newc{\varphi_n}: A_n \rightarrow A_{n+1}$ be given by $\varphi_n(k) = 2k$. Let $\mathbb{Z}\left[\frac{1}{2} \right] = \left\lbrace \frac{k}{2^n} : k \in \mathbb{Z}, n \in \mathbb{N} \right\rbrace$ be the set of diadic numbers. We claim that $\mathbb{Z}\left[\frac{1}{2} \right] = \displaystyle\lim_\rightarrow A_n$.

In fact, let $\mu_n: A_n \rightarrow \mathbb{Z}\left[\frac{1}{2} \right]$ be given by $\mu_n(k) = \frac{k}{2^n}$. We will show that item (i) of Definition \ref{def:inductivelimit} holds. Indeed, for $k \in A_n$, we have
\begin{align*}
\mu_{n+1} \circ \varphi_n(k)
= \mu_{n+1}(2k)
= \frac{2k}{2^{n+1}}
= \frac{k}{2^n}
= \mu_n(k).
\end{align*}

Now we check property (ii). Let $(B, \lbrace \lambda_n \rbrace_{n=1}^\infty )$ be a system where $B$ is a group and $\lambda_n: A_n \rightarrow B$ is a homomorphism satisfying $\lambda_n = \lambda_{n+1} \circ \varphi_n$ for all $n \geq 1$. Define $\lambda: \mathbb{Z}\left[\frac{1}{2} \right] \rightarrow B$ by
\begin{align}
\label{eqn:lambdadiadic}
\lambda\left(\frac{k}{2^n}\right) = \lambda_n(k).
\end{align}
We need to show that $\lambda$ is well-defined, i.e., that \eqref{eqn:lambdadiadic} does not depend on the choice of $n$. In fact, note that $\frac{2k}{2^{n+1}} = \frac{k}{2^n}$ and
\begin{align*}
\lambda_{n+1}(2k) = \lambda_{n+1} \circ \varphi_n(k)
= \lambda_n(k).
\end{align*}
By applying this argument recursively, we have that $\lambda$ is well-defined.

Note that the diagram of property (ii) holds for $\lambda$. Indeed, for $k \in A_n$, we have
\begin{align*}
\lambda \circ \mu_n(k) = \lambda \left( \frac{k}{2^n} \right)
= \lambda_n(k).
\end{align*}
Finally, we need to prove that $\lambda$ is the only homomorphism for which this diagram commutes. Suppose that $\widetilde{\lambda}:A \rightarrow B$ us another homomorphism with the same property. Then
\begin{align*}
\widetilde{\lambda}\left( \frac{k}{2^n} \right)
= \widetilde{\lambda}(\mu_n(k))
= \lambda_n(k)
= \lambda \left( \frac{k}{2^n} \right).
\end{align*}
This implies that $\widetilde{\lambda} = \lambda$. Therefore, $\mathbb{Z}\left[\frac{1}{2} \right] = \displaystyle\lim_\rightarrow A_n$.
\end{example}

The following proposition shows that any sequence of abelian groups has an inductive limit.

\begin{proposition} \cite[Propositions 6.2.5 and 6.2.6]{RLL-Ktheory} \label{prop:inductivegroupunion} Each sequence
$$
\begin{tikzcd}
A_1 \arrow[r, "\varphi_1"] &
A_2 \arrow[r, "\varphi_2"] &
A_3 \arrow[r, "\varphi_3"]  &
\cdots
\end{tikzcd}
$$
of abelian groups has an inductive limit $(A, \lbrace \beta_n \rbrace)$, and $A = \cup_{n=1}^\infty \beta_n(A_n)$. Moreover, if the groups $A_n$ are ordered and the homomorphisms $\varphi_n$ are positive, let
\begin{align*}
A^+ = \bigcup_{n=1}^\infty \beta_n(A_n^+).
\end{align*}
Then $(A, A^+)$ is an ordered abelian group, $\beta_n$ is a positive group homomorphism for each $n$, and $((A, A^+), \lbrace \beta_n \rbrace_{n=1}^\infty)$ is the inductive limit of the sequence in the category of ordered abelian groups.
\end{proposition}

The following lemma shows that we can induce a homomorphism of inductive limits $\mu: \displaystyle\lim_\rightarrow A_n \rightarrow \displaystyle\lim_\rightarrow B_n$ given a  sequence of homomorphisms $\mu_n: A_n \rightarrow B_n$ satisfying certain conditions. This will be useful in Chapter \ref{section:K0H0} when we study the isomorphism $K_0(C^*(G)) \cong H_0(G)$ for AF groupoids $G$ since, in this case, these groups are inductive limits.

\begin{lemma}
\label{lemma:homomorphisminductive}
Let $A$ be an object in a category $\mathscr{C}$ such that $A$ is the inductive limit of an inductive sequence with connecting morphisms $i_n: A_n \rightarrow A_{n+1}$ and inclusion morphisms $\alpha_n: A_n \rightarrow A$. Similarly, let $B$ be an inductive limit in $\mathscr{C}$ with connecting morphisms $j_n: B_n \rightarrow B_{n+1}$ and inclusion morphisms $\beta_n: B_n \rightarrow B$. Let $\lbrace \mu_n \rbrace_{n = 0}^\infty$ be a sequence of morphisms $\mu_n : A_n \rightarrow B_n$ satisfying $\mu_{n+1} \circ i_n = j_n \circ \mu_n$. Then there exists a unique morphism $\mu: A \rightarrow B$ making the diagram below commutative.
\begin{equation}
\label{eqn:homomorphisminductive}
\begin{tikzcd}
A_n \arrow[r,"\mu_n"] 
\arrow{d}{}[swap]{\alpha_n}
& B_n \arrow[d, "\beta_n"]\\
A \arrow{r}{\mu}[swap]{}
& B
\end{tikzcd}
\end{equation}
Moreover, if each $\mu_n$ is an isomorphism, we have that $\mu$ is an isomorphism.
\end{lemma}
\begin{proof}
For every $n \geq 1$, let $\lambda_n: A_n \rightarrow B$ be defined by $\lambda_n = \beta_n \circ \mu_n$. Then we have
\begin{align*}
\lambda_n
= \beta_n \circ \mu_n
= \beta_{n+1} \circ j_n \circ \mu_n
= \beta_{n+1} \circ \mu_{n+1} \circ i_n
= \lambda_{n+1} \circ i_n.
\end{align*}
Then, by definition of inductive limits, there exists a unique $\mu: A \rightarrow B$ making the diagram below commutative.
$$
\begin{tikzcd}
&A_n \arrow{dl}{}[swap]{\alpha_n}
\arrow{dr}{\lambda_n}[swap]{}\\
A \arrow{rr}{\mu}[swap]{}
&& B
\end{tikzcd}
$$
By definition of $\mu$, the diagram \eqref{eqn:homomorphisminductive} is commutative.

Now suppose that each $\mu_n$ is invertible. By similar arguments, there exists a unique $\kappa: B \rightarrow A$ such that the diagram below is commutative.
\begin{equation*}
\begin{tikzcd}
B_n \arrow[r,"\mu_n^{-1}"] 
\arrow{d}{}[swap]{\beta_n}
& A_n \arrow[d, "\alpha_n"]\\
B \arrow{r}{\kappa}[swap]{}
& A
\end{tikzcd}
\end{equation*}
We claim that $\kappa = \mu^{-1}$. Note that both diagrams below are commutative.
\begin{equation*}
\begin{tikzcd}
B_n \arrow[rr,"\mu_n \circ \mu_n^{-1}"] 
\arrow{d}{}[swap]{\beta_n}
&& B_n \arrow[d, "\beta_n"]\\
B \arrow{rr}{\mu \circ \kappa}[swap]{}
&& B
\end{tikzcd}
\hspace{40pt}
\begin{tikzcd}
B_n \arrow[rr,"\mu_n \circ \mu_n^{-1}"] 
\arrow{d}{}[swap]{\beta_n}
&& B_n \arrow[d, "\beta_n"]\\
B \arrow{rr}{1_B}[swap]{}
&& B
\end{tikzcd}
\end{equation*}
Property (ii) of Definition \ref{def:inductivelimit} implies that $\mu \circ \kappa = 1_B$. Analogously, we can show that $\kappa \circ \mu = 1_A$. Therefore $\kappa = \mu^{-1}$.
\end{proof}

Renault gives a definition of inductive limits of groupoids \cite[Section III.1]{Renault} that looks more specific than Definition \ref{def:inductivelimit}. The following lemma shows that Renault's definition is an application of inductive limits when the groupoids and the connecting maps satisfy certain conditions.

\begin{lemma}
\label{lemma:inductivelimitgroupoids}
Let $\lbrace G_n \rbrace_{n=1}^\infty$ be an increasing sequence of locally compact, Hausdorff, second countable, \'etale groupoids with the same unit space. Suppose that, for each $n$, $G_n$ is open in $G_{n+1}$ and such that $G_n$ is equipped with the subspace topology from $G_{n+1}$. Let $G$ be the union of these groupoids. Then we can equip $G$ with the topology such that
\begin{align}
\label{eqn:inductivelimittopology}
V \text{ is open in $G$ }
\Leftrightarrow
V \cap G_n \text{ is open in $G_n$ for all $n$.}
\end{align}
We call this topology the \newterm{inductive limit topology}. Moreover, $G$ is the inductive limit of the sequence $\lbrace G_n \rbrace_{n=1}^\infty$, where both the connecting morphisms $\varphi_n: G_n \rightarrow G_{n+1}$ and the inclusion morphisms $\mu_n: G_n \rightarrow G$ are inclusion maps.
\end{lemma}
\begin{proof}
If we equip $G$ with the operations induced from $G_n$, then $G$ becomes a groupoid. It follows from the topological properties of $G_n$ that \eqref{eqn:inductivelimittopology} defines a topology on $G$. Moreover, $G$ is locally compact Hausdorff second countable \'etale.

We only need to prove that $G$ is the inductive limit of the sequence of groupoids $G_n$. Since the maps $\varphi_n$ and $\mu_n$ are inclusions, it is straightforward that $\mu_{n+1} \circ \varphi_n = \mu_n$ for all $n$. Then property (i) of Definition \ref{def:inductivelimit} holds.

We will show property (ii) of the same definition. Let $(H, \lbrace \lambda_n \rbrace_{n=1}^\infty)$ be a system, where $H$ is a locally compact Hausdorff second countable \'etale groupoid and each $\lambda_n: G_n \rightarrow H$ is a continuous homomorphism satisfying $\lambda_n = \lambda_{n+1} \circ \varphi_n$. Then we can define $\lambda: G \rightarrow H$ by
\begin{align*}
\lambda(g) = \lambda_n(g)
\hspace{15pt}\text{if }g \in G_n.
\end{align*}
Since $G$ is the union of the $G_n$, then $\lambda$ is unique and well-defined. Also, $\lambda_n = \lambda \circ \mu_n$ for all $n$. Therefore, $G$ is the inductive limit of the sequence $\lbrace G_n \rbrace_{n=1}^\infty$.
\end{proof}

The following lemma shows that the K-theory of groupoid C*-algebras is continuous with respect to inductive limits of groupoids.

\begin{lemma}
Let $G = \displaystyle\lim_\rightarrow G_n$ be the inductive limit of a sequence of locally compact, Hausdorff, second countable, \'etale groupoids. Then $C^*(G) = \displaystyle\lim_\rightarrow C^*(G_n)$ and $K_\ast(C^*(G)) = \displaystyle\lim_\rightarrow K_\ast(C^*(G_n))$.
\end{lemma}
\begin{proof}
First we show that $\displaystyle\lim_\rightarrow C^*(G_n) \subset C^*(G)$. Let $f \in C_c(G_n)$. Since $G_n$ is open in $G$, we make an abuse of notation and identify $f$ as a function of $C_c(G)$ with support in $G_n$. Thus $C_c(G_n) \subset C_c(G) \subset C^*(G)$ . This implies that
\begin{align}
\label{eqn:CastGnoverline}
C^*(G_n) = \overline{C_c(G)} \subset C^*(G).
\end{align}
Now let $f \in C_c(G)$. Then there exists a compact set $K$ containing the support of $f$. Since $G$ is the increasing union of the sets $G_n$, there exists an $n$ such that $K \subset G_n$. By the inductive limit topology, the function $G_n \rightarrow \mathbb{C}$ given by $x \mapsto f(x)$ is continuous and compactly supported. We make an abuse of notation and identify this function with $f \in C_c(G_n)$. Since $f$ is arbitrary, we have $C_c(G) = \bigcup_{n=1}^\infty C_c(G_n)$. Thus
\begin{align*}
C^*(G) = \overline{\bigcup_{n=1}^\infty C_c(G_n)} \subset \overline{\bigcup_{n=1}^\infty C^*(G_n)} = \lim_\rightarrow C^*(G_n)
\end{align*}
So, by \eqref{eqn:CastGnoverline}, we have $C^*(G) = \displaystyle\lim_\rightarrow C^*(G_n)$. It follows from \cite[Theorem 6.3.2]{RLL-Ktheory} and \cite[Proposition 8.2.7]{RLL-Ktheory} that $K_\ast(C^*(G)) = \displaystyle\lim_\rightarrow K_\ast(C^*(G_n))$.
\end{proof}

\chapter{AF groupoids}
\label{appendix:AFgroupoids}

There are different definitions for AF groupoids in the literature. The first was given by Renault \cite[Definition III.1.1]{Renault} in his book, and he proved in \cite[Proposition III.1.15]{Renault} that AF groupoids determine AF algebras up to isomorphism. Because of this result, we can see AF groupoids as the groupoid analogs of AF algebras. 

Later on, other definitions for such groupoids have been introduced. In 2003 \cite{Renault-AF}, Renault defined AF equivalence relations. In 2004, Giordano, Putnam and Skau \cite[Definition 3.7]{GPS} gave a different but equivalent definition. In 2012, Matui \cite[Definition 2.2]{Matui} defined AF groupoids. His definition was less general than Renault's definition from 1980 \cite[Definition III.1.1]{Renault}, since Matui's definition required the unit space to be compact. Finally, in 2019, Farsi, Kumjian, Pask and Sims \cite[Definition 4.9]{FKPS} gave another definition of AF groupoids, along the lines of Matui, but for groupoids with non-compact unit spaces.

In a still in progress paper with Astrid an Huef, Lisa Orloff Clark and Camila Sehnem, we prove the equivalence of Renault's and Farsi et al's definition.

Here we give Farsi et al's definition. In the same way that every AF algebra is approximated by a sequence of finite dimensional C*-algebras, an AF groupoid is approximated by a sequence of more basic groupoids, called elementary.

\newc{
\begin{definition}
Let $X, Y$ be locally compact, Hausdorff, second countable, totally disconnected spaces. Let $\sigma: X \rightarrow Y$ be a surjective local homeomorphism. Define the set
\begin{align*}
R(\sigma) = \lbrace (u,v) \in X \times X: \sigma(u) = \sigma(v) \rbrace,
\end{align*}
and equip it with the operations $(u,v)(v,w) = (u,w)$ and $(u,v)^{-1} = (v,u)$. Then $R(\sigma)$ becomes a groupoid. We endow this set with the subspace topology from $X \times X$. Then $R(\sigma)$ is locally compact, Hausdorff, second countable, \'etale, and totally disconnected.
\end{definition}
}

\begin{definition}
A locally compact, Hausdorff, second countable, totally disconnected groupoid is said to be \newterm{elementary} if it is isomorphic to the groupoid $R(\sigma)$ for some surjective local homeomorphism $\sigma: X \rightarrow Y$ between locally compact, Hausdorff, second countable, totally disconnected spaces.
\end{definition}

\begin{definition}
\label{def:AFFKPS}
We say that a groupoid is \newterm{AF} if it is the increasing union of open elementary groupoids with the same unit space.
\end{definition}

It follows from Lemma \ref{lemma:inductivelimitgroupoids} that an AF groupoid is an inductive limit of a sequence of elementary groupoids. We use the example below in the commutative diagram \eqref{eqn:diagram} from page \pageref{eqn:diagram}.

\begin{example}
Let $X$ be a locally compact, Hausdorff, second countable, totally disconnected space, and let $\sigma: X \rightarrow X$ be a surjective local homeomorphism. Then
\begin{align*}
c^{-1}(0) = \lbrace (x, y) \in X^2: \sigma^n(x) = \sigma^n(y) \text{ for some } n \in \mathbb{N} \rbrace
\end{align*}
is an AF groupoid. In fact, it is the increasing union of the open subgroupoids $R(\sigma^n)$.
\end{example}

\begin{proposition}
\label{prop:AFRsigma}
Let $G$ be an AF groupoid. Write $G$ as the inductive limit of elementary groupoids $G = \displaystyle\lim_\rightarrow G_n$. Then, for each $n$, there exists a locally compact Hausdorff, second countable, totally disconnected space $Y_n$, and there is a surjective local homeomorphism $\sigma_n: G^{(0)} \rightarrow Y_n$ such that $G_n = R(\sigma_n)$.
\end{proposition}
\begin{proof}
Let $G$ be an AF groupoid written as the inductive limit of elementary groupoids $G_n$. For every $n$, there are locally compact, Hausdorff, second countable, totally disconnected spaces $X_n, Y_n$, and there is a surjective local homeomorphism $\rho_n: X_n \rightarrow Y_n$ such that $G_n \cong R(\rho_n)$.

Let $\mu_n: G_n \rightarrow R(\rho_n)$ be an isomorphism. Note that $G_n^{(0)} = G^{(0)}$. Define the map $\sigma_n: G^{(0)} \rightarrow Y_n$ by
\begin{align*}
\sigma_n(x) = \rho_n \circ \mu_n(x),
\hspace{20pt}
\text{ for $x \in G^{(0)}$.}
\end{align*}
Since $\mu_n$ is an isomorphism of \'etale groupoids, we have that $\mu_n\vert_{G^{(0)}}: G^{(0)} \rightarrow X_n$ is a homeomorphism. Since $\rho_n$ is a surjective local homeomorphism, and since $\sigma_n$ is the composition of $\mu\vert_{G^{(0)}}$ and $\rho_n$, then $\sigma_n$ is a surjective local homeomorphism.

Now we prove that $G_n = R(\sigma_n)$. The isomorphism $\mu_n$ implies that $G_n$ is principal. Moreover,
\begin{align*}
G_n \subset G^{(0)}_n \times G^{(0)}_n = G^{(0)} \times G^{(0)}.
\end{align*}
So, given $(u,v) \in G^{(0)} \times G^{(0)}$, we have
\begin{align*}
(u,v) \in G_n
&\Leftrightarrow (\mu_n(u), \mu_n(v)) \in R(\rho_n) \\
&\Leftrightarrow \rho(\mu_n(u)) = \rho(\mu_n(v)) \\
&\Leftrightarrow \sigma_n(u) = \sigma_n(v).
\end{align*}
Therefore, $G_n = R(\sigma_n)$.
\end{proof}

\chapter{Projections in $C_0(X, M_n)$}
\label{appendix:C0XMn}

We fix $X$ to be a locally compact, Hausdorff, second countable, totally disconnected space, and $n$ to be a positive integer. We show that two projections in $C_0(X, M_n)$ are equivalent if, and only if, they have the same trace.

We begin this appendix by studying equivalences of projections in $M_n$, and then we generalise the results on $M_n$ to $C_0(X, M_n)$.

\begin{definition}
Let $\mathrm{tr}: M_n \rightarrow \mathbb{C}$ be the standard trace given by
\begin{align*}
\mathrm{tr}
\begin{bmatrix}
a_{11} & a_{12} & \cdots & a_{1n} \\
a_{21} & a_{22} & \cdots & a_{2n} \\
\vdots & \vdots & \ddots & \vdots \\
a_{n1} & a_{n2} & \cdots & a_{nn}
\end{bmatrix}
= \sum_{i=1}^n a_{ii}.
\end{align*}
We call $\tr a$ the \newterm{trace} of $a \in M_n$.
\end{definition}

\begin{remark}
\label{rmk:traceproperty}
Note that $\mathrm{tr}(ab) = \mathrm{tr}(ba)$ for $a,b \in M_n$. In fact,
\begin{align*}
\mathrm{tr}(ab)
= \sum_{i=1}^n (ab)_{ii}
= \sum_{i=1}^n \sum_{j=1}^n a_{ij} b_{ji}
= \sum_{j=1}^n \sum_{i=1}^n b_{ji} a_{ij}
= \sum_{j=1}^n (ba)_{jj}
= \mathrm{tr}(ba).
\end{align*}
\end{remark}


\begin{corollary}
\label{corollary:psumv}
Let $p \in M_n$ be a projection with trace $k$. Then there are $k$ orthonormal vectors $v_1, \hdots, v_k$ such that $p = \sum_{i=1}^k v_i v_i^*$.
\end{corollary}
\begin{proof}
\newc{Let $\lbrace v_1, \dots, v_k \rbrace$ be an orthonormal basis of $p(\mathcal{C}^k$. If $k < n$, find $v_{k+1}, \dots, v_n$ such that $\lbrace v_1, \dots, v_n \rbrace$ is an orthonormal basis of $\mathbb{C}^n$. For $j = 1, \dots, k$, there exists an $x \in \mathbb{C}^n$ such that $v_j = px$. Then
\begin{align*}
pv_j = p^2 x = px = v_j
\hspace{15pt}\text{and}\hspace{15pt}
\sum_{i=1}^k v_i v_i^* v_j = v_j.
\end{align*}
For $k < j \leq n$, we have
\begin{align*}
pv_j = 0
\hspace{15pt}\text{and}\hspace{15pt}
\sum_{i=1}^k v_i v_i^* v_j = 0.
\end{align*}
Therefore $p = \sum_{i=1}^k v_i v_i^*$.
}
\end{proof}
\oldc{
By Schur's lemma, there exists a unitary matrix and an upper triangular matrix $\Lambda$ such that $p = u \Lambda u^*$. Since $p$ is self-adjoint, we have that $\Lambda$ is diagonal. Moveover, the fact that $p$ is a projection implies that the elements in the diagonal of $\Lambda$ are zeros and ones.

Let $k$ be the trace of $p$. Then $k = \tr u \Lambda u^* = \tr \Lambda$ is the number of ones of $\Lambda$. For simplicity, we assume that the first $k$ elements of $\mathrm{diag}\hphantom{.} \Lambda$ are ones (If this is not the case, we can find an elementary matrix $e$ such that $\widetilde{\Lambda} = e^{-1} \Lambda e$ has this property and $p = \widetilde{u} \widetilde{\Lambda} \widetilde{u}^*$ with $\widetilde{u} = u e$).

For each $i = 1, \hdots, n$, let $v_i$ be the $i^{\mathrm{th}}$ column of $u$. Since $u$ is unitary, then $v_1, \dots, v_n$ is an orthonormal basis of $\mathbb{C}^n$. We prove that $p$ and $\sum_{i=1}^k v_i v_i^*$ are equal by comparing the images of both matrices on the basis $v_1, \dots, v_n$.

Given $i \leq k$,
\begin{align*}
pv_i = u \Lambda u^* v_i = u \Lambda e_i = u e_i = v_i
\hspace{15pt}&\text{and}\hspace{15pt}
\sum_{l=1}^k v_l v_l^* v_i = v_i. 
\intertext{If $i > k$,}
pv_i = u \Lambda u^* v_i = u \Lambda e_i = 0
\hspace{15pt}&\text{and}\hspace{15pt}
\sum_{l=1}^k v_l v_l^* v_i = 0. 
\end{align*}
Since $v_1, \hdots, v_n$ form an orthonormal basis of $\mathbb{C}^n$, we have $p = \sum_{l=1}^k v_l v_l^*$.}
\begin{remark}
\label{rmk:dimp}
Note that this corollary implies that $\tr p$ is the dimension of $p(\mathbb{C}^k)$.
\end{remark}

\begin{remark}
\label{rmk:idsumv}
By the similar calculations of Corollary \ref{corollary:psumv}, we can show that for $v_1, \hdots, v_n$ orthonormal basis of $\mathbb{C}^n$, we have $\mathrm{id} = \sum_{i=1}^n v_i v_i^*$.
\end{remark}

Now we prove the connection between the trace and the equivalence of projections in $M_n$. The equivalence $(ii) \Leftrightarrow (i)$ of the next lemma is Exercise 2.9 of R{\o}rdan, Larsen and Laustsen's book \cite{RLL-Ktheory}.
\begin{lemma}
\label{lemma:matrixtrace}
Let $p, q$ be projections in $M_n$. Then the following are equivalent:
\begin{enumerate}[(i)]
\item $p \sim q$,
\item $\mathrm{tr}(p) = \mathrm{tr}(q)$,
\item $p \sim_u q$.
\end{enumerate}
\end{lemma}
\begin{proof}
First we show $(i) \Rightarrow (ii)$. Let $p, q \in M_n$ be two projections with $p \sim q$. Then there exists $v \in M_n$ with $p = v^*v$ and $q = vv^*$. Then
\begin{align*}
\tr p = \oldc{\tr u^*u = \tr uu^*}\newc{\tr v^*v = \tr vv^*} = \tr q.
\end{align*}

Now we prove $(ii) \Rightarrow (iii)$. Let $p,q \in M_n$ be projections with the same trace $k$. By Corollary \ref{corollary:psumv}, there are $v_1, \hdots, v_k$ orthonormal and $w_1, \hdots, w_k$ orthonormal such that
\begin{align*}
p = \sum_{i=1}^k v_i v_i^*
\hspace{15pt}\text{and}\hspace{15pt}
q = \sum_{i=1}^k w_i w_i^*
\end{align*}
Find $w_{k+1}, \hdots, w_n$ and $v_{k+1}, \hdots, v_n$ such that both $v_1, \hdots, v_n$ and $w_1, \hdots, w_n$ are orthonormal basis of $\mathbb{C}^n$. Let $u = \sum_{i=1}^n w_i v_i^*$. Then $u$ is unitary. In fact, by Remark \ref{rmk:idsumv},
\begin{align*}
u^*u = \sum_{i=1}^n \oldc{w_i w_i^*}\newc{v_i v_i^*} = \mathrm{id}
\hspace{15pt}\text{and}\hspace{15pt}
uu^* = \sum_{i=1}^n \oldc{v_i v_i^*}\newc{w_i w_i^*} = \mathrm{id}.
\end{align*}
Also, $p = u^* q u$. Indeed,
\begin{align*}
u^* q u
&= \sum_{i=1}^n v_i w_i^* q u \\
&= \sum_{i=1}^n v_i w_i^* \left( \sum_{j=1}^k w_j w_j^* \right) u \\
&= \sum_{i=1}^k v_i w_i^* w_i w_i^* u
\hspace{15pt}\text{since the $w_i$ are orthonormal}\\
&= \sum_{i=1}^k v_i w_i^* u \\
&= \sum_{i=1}^k v_i w_i^* \sum_{j=1}^n w_j v_j^* \\
&= \sum_{i=1}^k v_i w_i^* w_i v_i^* \\
&= \sum_{i=1}^k v_i v_i^* \\
&= p.
\end{align*}
Then $(iii)$ holds.

Finally, we show $(iii) \Rightarrow (i)$. If $(iii)$ is true, there exists a unitary such that $p = u^* q u$. Then $q = upu^*$. Let $v = up$. Then $v^*v = p$ and $vv^* = q$. Therefore $p \sim q$.
\end{proof}

\begin{remark}
Note that the implication $(iii) \Rightarrow (i)$ of the Lemma above is a known result for more general C*-algebras. See \cite[Proposition 2.2.7]{RLL-Ktheory}.
\end{remark}
Now that we know the connection between the trace and projections in $M_n$, we want to generalise this to $C_0(X, M_n)$.

The following proposition will be used to show that, for a projection $p \in C_0(X, M_n)$, the trace $\tr p$ is locally constant.

\begin{proposition} \cite[Propositions 2.2.4 and 2.2.7]{RLL-Ktheory}
\label{prop:closeprojections}
Let $p, q$ be projections in a C*-algebra $A$. If $\Vert p - q \Vert < 1$, then $p \sim_u q$.
\end{proposition}

\newc{It follows from this proposition that if a projection $p$ has norm $\Vert p \Vert < 1$, then $p = 0$.} Now we define the trace operator on $\mathcal{P}(C_0(X, M_n))$.

\begin{lemma}
\label{lemma:tracewelldefined}
There exists a map $\mathrm{tr}: \mathcal{P}(C_0(X, M_n)) \rightarrow C_c(X, \mathbb{N})$ given by
\begin{align}
\label{eqn:continuoustrace}
\tr p(x) = \mathrm{tr}(p(x)),
\end{align}
for $p \in \mathcal{P}(C_0(X, M_n))$ and $x \in X$. We denote $\tr p$ to be the \newterm{trace} of $p$.
\end{lemma}
\begin{proof}
It is straightforward from \eqref{eqn:continuoustrace} that the map is linear on the set $\mathcal{P}(C_0(X, M_n))$. So we only need to prove that it is well-defined.

\begin{itemize}
\item $\mathrm{tr}(p) \in C(X, \mathbb{N})$

Let $p \in \mathcal{P}(C_0(X, M_n))$. Given $x \in X$, Remark \ref{rmk:dimp} implies that
\begin{align*}
\mathrm{tr}(p)(x) = \mathrm{tr}(p(x)) = \dim (p(x)(\mathbb{C}^n)) \in \mathbb{N}.
\end{align*}
Then $\mathrm{tr}(p)$ assumes values in $\mathbb{N}$.

We still need to prove that $\mathrm{tr}(p)$ is continuous. Fix $x \in X$. Since $p$ is continuous, there exists an open neighbourhood $V \subset X$ of $x$ such that
\begin{align}
\label{eqn:tracewelldefined1}
\Vert p(x) - p(y) \Vert < 1
\text{\hspace{10pt}for }
y \in V.
\end{align}
Then, for $y \in V$, Proposition \ref{prop:closeprojections} and Lemma \ref{lemma:matrixtrace} imply that $\tr(p(x)) = \tr(p(y))$. Hence, $\mathrm{tr}(p)$ is constant on $V$. Since $x$ is arbitrary, we have that $\mathrm{tr}(p)$ is continuous, and then $\mathrm{tr}(p) \in C(X, \mathbb{N})$.

\item $\mathrm{tr}(p)$ is compactly supported

Let $p \in \mathcal{P}(C_0(X, M_n))$. Then there exists a compact subset $K \subset X$ such that
\begin{align}
\label{eqn:tracewelldefined2}
\Vert p(x) - 0 \Vert
= \Vert p(x) \Vert
< 1
\text{\hspace{10pt}for all }
x \in X \setminus K.
\end{align}
Let $x \in X \setminus K$. By \eqref{eqn:tracewelldefined2}, Proposition \ref{prop:closeprojections} and Lemma \ref{lemma:matrixtrace},
\begin{align*}
\mathrm{tr}(p(x))
= \mathrm{tr}(0)
= 0.
\end{align*}
Hence $\mathrm{tr}(p)$ is compactly supported and, therefore, it is well-defined. \qedhere
\end{itemize}
\end{proof}

\begin{remark}
\label{rmk:projectioncompactlysupported}
Note that every projection in $C_0(X, M_n)$ is an element of $C_c(X, M_n)$. In fact, let $p \in \mathcal{P}(C_0(X, M_n))$. By Lemma \ref{lemma:tracewelldefined}, its trace is compactly supported. Let $K$ be the compact support of $\mathrm{tr}(p)$. By Lemma \ref{lemma:matrixtrace}, we have that, for $x \in X \setminus K$,
\begin{align*}
\dim(p(x)\mathbb{C}^n)
= \mathrm{tr}(p(x))
= 0.
\end{align*}
Then $p(x) = 0$. Therefore $p \in C_c(X, M_n)$.
\end{remark}



Now we prove the main result of this appendix.

\begin{proposition}
\label{prop:traceCXMn}
Let $p, q \in C_c(X, M_n)$ be two projections. Then the following are equivalent:
\begin{enumerate}[(i)]
\item $p \sim q$,
\item $\tr p = \tr q$,
\item For every compact subset $K \subset X$ such that $p, q$  vanish outside $K$, \oldc{and }there is a $u \in C_c(X, M_n)$ supported on $K$ such that
\begin{align*}
u^*u = uu^* = \mathrm{id} 1_K
\hspace{15pt}\text{and}\hspace{15pt}
 p = u^* q u.
\end{align*}
\end{enumerate}
\end{proposition}
\begin{proof}
Let us prove $(i) \Rightarrow (ii)$. Suppose that $p, q \in C_0(X, M_n)$ are two projections with $p \sim q$. Then there exists $v \in C_0(X, M_n)$ such that $p = v^*v$ and $q = vv^*$. Then, by Remark \ref{rmk:traceproperty} and by the definition of the trace on $C_0(X, M_n)$, we have $\mathrm{tr}(p) = \mathrm{tr}(q)$. Thus, $(ii)$ holds.

Now we prove $(ii) \Rightarrow (iii)$. Let $p, q \in C_0(X, M_n)$ be projections with the same trace. Remark \ref{rmk:projectioncompactlysupported} implies that $p, q \in C_c(X, M_n)$. Choose an arbitrary compact subset $K \subset X$ such that both $p$ and $q$ vanish outside $K$.

Since $p$ and $q$ are continuous in the supremum norm, and since $K$ is compact open, there is a collection of elements $x_1, \hdots, x_m \in K$ and a partition $V_1, \hdots, V_m$ of $K$ by disjoint compact open subsets such that, for each $i = 1, \hdots, m$, we have
\begin{align}
\label{eqn:traceCXMn1}
x_i \in V_i,
\hspace{10pt}
\Vert p(y) - p(x_i) \Vert < \tfrac{1}{2},
\text{\hspace{10pt}and }
\Vert q(y) - q(x_i) \Vert < \tfrac{1}{2},
\text{\hspace{15pt}for }
y\in V_i.
\end{align}

Let $i = 1, \hdots, m$. By assumption, $\mathrm{tr}(p(x_i)) = \mathrm{tr}(q(x_i))$. It follows from Lemma \ref{lemma:matrixtrace} that there exists an unitary $v_i \in M_n$ such that $p(x_i) = v_i^* q(x_i) v_i$.

Let
\begin{align*}
\widetilde{p} = \displaystyle\sum_{i=1}^m p(x_i)1_{V_i},
\text{\hspace{10pt}}
\widetilde{q} = \displaystyle\sum_{i=1}^m q(x_i)1_{V_i}
\text{\hspace{10pt}and\hspace{10pt}}
v = \sum_{i=1}^m v_i 1_{V_i}.
\end{align*}
Then $\widetilde{p}$ and $\widetilde{q}$ are projections in $C_0(X, M_n)$. Also, $\widetilde{p} = v^* \widetilde{q} v$ and $v^*v = vv^* = \mathrm{id}1_K$.

Using the norm of $C_0(X, M_n)$, equation \eqref{eqn:traceCXMn1} implies that
\begin{align}
\label{eqn:traceCXMn2}
\Vert p - \widetilde{p} \Vert \leq \tfrac{1}{2}
\text{\hspace{10pt}and\hspace{10pt}}
\Vert q - \widetilde{q} \Vert \leq \tfrac{1}{2}.
\end{align}

Let $A \subset C_0(X, M_n)$ be the $\ast$-subalgebra of functions supported on $K$. Then $A$ is a unital C*-algebra with unit $\mathrm{id} 1_K$. By equation \eqref{eqn:traceCXMn2} and Proposition \ref{prop:closeprojections}, there are unitary elements $u_p, u_{q} \in A$ such that
\begin{align*}
p = u_p^* \widetilde{p} u_p
\hspace{15pt}\text{and}\hspace{15pt}
\widetilde{q} = u_{q}^* q u_{q}.
\end{align*}

Note that $u_p$ and $u_{q}$ are unitary in $A$, not necessarily in $C_0(X, M_n)$. Then
\begin{align*}
u_p^* u_p = u_p u_p^* = \mathrm{id}1_K
\hspace{15pt}\text{and}\hspace{15pt}
u_{q}^* u_{q} = u_{q} u_{q}^* = \mathrm{id} 1_K.
\end{align*}
Let $u = u_{q} v u_p$. Then $u^*u = uu^* = \mathrm{id} 1_K$, and
\begin{align*}
u^* q u
= u_p^* v^* u_q^* q u_q v u_p 
= u_p^* v^* \widetilde{q} v u_p 
= u_p^* \widetilde{p} u_p 
= p.
\end{align*}
Therefore $(iii)$ holds.

Finally, we prove $(iii) \Rightarrow (i)$. Suppose $(iii)$ holds. \oldc{Fix a compact subset $K \subset X$ containing the supports of $p$ and $q$.}  Let $v = qu$. Then, by assumption,
\begin{align*}
v^* v = u^* q u = p,
\end{align*}
and
\begin{align*}
v v^*
= q u u^* q
= q (\mathrm{id} 1_K) q
= q q 
= q.
\end{align*}
Note that $1_K q = q$ because $q$ has support in $K$.
Therefore $(i)$ is true.
\end{proof}


\chapter{Deaconu-Renault Groupoids}
\label{appendix:Deaconu-Renault}

In this appendix, we study Deaconu-Renault groupoids and we state some of their basic properties present in the literature. The construction of this class of groupoids depends on the notion of local homeomorphism. Given the topological spaces $X$ and $Y$, a map $\sigma: X \rightarrow Y$ is a \newterm{local homeomorphism} if every $x \in X$ has an open neighbourhood $\mathcal{U}$ such that $\sigma(\mathcal{U})$ is open and $\sigma\vert_{\mathcal{U}}: \mathcal{U} \rightarrow \sigma(\mathcal{U})$ is a homeomorphism.

In this appendix, $X$ always denotes a locally compact, Hausdorff, second countable space, and $\sigma: X \rightarrow X$ is a surjective local homeomorphism. We define the set
\begin{align*}
\mathcal{G} = \lbrace (x,k,y) \in X \times \mathbb{Z} \times X : \exists m, n \in \mathbb{N}, k = m - n, \sigma^m(x) = \sigma^n(y) \rbrace,
\end{align*}
and we endow it with the range and source maps $r,s: \mathcal{G} \rightarrow \mathcal{G}$ with 
\begin{align*}
r(x,k,y) = (x,0,x)
\text{\hspace{7pt} and \hspace{7pt}}
s(x,k,y) = (y,0,y);
\end{align*}
we define the set $
\mathcal{G}^{(2)} = \lbrace ((x, k_1, y), (y,k_2 ,z)): (x, k_1, y), (y, k_2, z) \in \mathcal{G} \rbrace
$, and we define the product $\mathcal{G}^{(2)} \rightarrow \mathcal{G}$ and inverse $\mathcal{G} \rightarrow \mathcal{G}$ by
\begin{align*}
(x, k_1, y)(y, k_2, z) = (x, k_1 + k_2, z),
\text{\hspace{10pt}}
(x,k,y)^{-1} = (y, -k, x).
\end{align*}
With these operations, $\mathcal{G}$ becomes a groupoid. The unit space is $\mathcal{G}^{(0)} = \lbrace (x,0,x) : x \in X \rbrace$, which is identified with $X$ by the map $(x,0,x) \mapsto x$.

Now we equip $\mathcal{G}$ with the topology that makes it a Deaconu-Renault groupoid. We use the notation $\mathcal{U}_{A,B}^{m,n}$ from\oldc{ in} \cite{ThomsenO2} to denote the open sets that generate the topology of $\mathcal{G}$.

\begin{proposition}
Let $\mathcal{G}$ be as above. Then, \oldc{$\mathcal{G}$ is groupoid. Moreover, }for all pairs of open subsets $A, B \subset X$ and for integers $m,n \geq 0$ such that $\sigma\vert_A^m$ and $\sigma\vert_B^n$ are injective with $\sigma^m(A) = \sigma^n(B)$, define the subset
\begin{align*}
\mathcal{U}_{A,B}^{m,n}
= \lbrace (x, m - n, y) \in A \times \mathbb{Z} \times B : \sigma^m(x) = \sigma^n(y) \rbrace.
\end{align*}
Then the sets $\mathcal{U}_{A,B}^{m,n}$ generate a topology on $\mathcal{G}$ that makes the groupoid a locally compact, Hausdorff, second countable and \'etale. Further, $\mathcal{G}$ is amenable. We denote the groupoid $\mathcal{G}$ endowed with this topology a \newterm{Deaconu-Renault groupoid}.
\end{proposition}
\begin{proof}
It follows from \cite[Theorem 1]{Deaconu} (see also \cite[Theorem 8]{Frausino-thesis} for a less technical proof) that $\mathcal{G}$ is locally compact, Hausdorff, second countable and \'etale. By \cite[Proposition 2.9]{Renault-Cuntzlike}, the groupoid is also amenable.
\end{proof}

Note that there are different notions of amenability for groupoids: topological amenable and measurewise amenable. By Theorem 3.3.7 of \cite{ADRenault}, these two notions coincide for groupoids that are locally compact, Hausdorff and \'etale. In this case, \cite[Proposition 6.1.8]{ADRenault} implies that the reduced and full C*-algebra of the groupoid coincide.

The first example of this groupoid was described in Renault's book \cite[Section III.2]{Renault} by a class of groupoids whose C*-algebras are Cuntz algebras, where he assumed that $X$ was the set of sequences $(x_i)_{i \in \mathbb{N}}$, \newc{with} $x_i \in \lbrace 0, 1, \dots, n \rbrace$ (or $x_i \in \mathbb{N}$ for the Cuntz algebra $\mathcal{O}_\infty$), and $\sigma$ was given by
\begin{align*}
\newc{\sigma(x_0 x_1 x_2 \dots) = x_1 x_2 \dots.}
\end{align*}
Deaconu \cite{Deaconu} generalises Renault's definition by considering compact Hausdorff spaces $X$ and assumes that $\sigma:X \rightarrow X$ is a covering map. Under these assumptions, he proved that when $\sigma$ is a homeomorphism, then $C^*(\mathcal{G})$ is isomorphic to the crossed product $C(X) \rtimes_\sigma \mathbb{Z}$.

Anantharaman-Delaroche \cite[Example 1.2 (c)]{Anantharaman-Delaroche} generalised it to the case where $\sigma$ is a surjective local homeomorphism, and $X$ is locally compact. She also gave a sufficient conditions to make the C*-algebra of this groupoid purely infinite \cite[Propositions 4.2 and 4.3]{Anantharaman-Delaroche}. Then Renault \cite{Renault-Cuntzlike} generalised this definition by considering a map $\sigma$ that is defined only on an open subset of $X$. For the purpose of our thesis, here we consider the definition by Anantharaman-Delaroche, and we impose the additional conditions on the space $X$ to be Hausdorff and second countable.

\begin{example}
Assume that $\sigma$ is a homeomorphism. Then $C^*(\mathcal{G}) \cong C_0(X) \rtimes_\sigma \mathbb{Z}$. This isomorphism was proved for $X$ compact in \cite{Deaconu}. By applying the same ideas from Chapter \ref{section:crossedproducts}, \oldc{one can set find}\newc{let $\rho: C_0(X) \rightarrow C^*(\mathcal{G})$ be} \oldc{is }the inclusion map, and $u_1 = 1_{\mathcal{U}_{X,X}^{0,1}} \in M(C^*(\mathcal{G}))$ to get the isomorphism $\rho \rtimes u: C_0(X) \rtimes_\sigma \mathbb{Z} \rightarrow C^*(\mathcal{G})$ given by
\begin{align*}
\rho \rtimes u(F)(x, k, \sigma^k(x)) = F(k)(x),
\end{align*}
for $x \in X$, $k \in \mathbb{Z}$, and $F \in C_c(\mathbb{Z}, C_0(X))$.
\end{example}

\begin{example}
Let $E$ be a row-finite directed graph with no sinks and let $E^\infty$ be the set of infinite paths of $E$, and define $\sigma: E^\infty \rightarrow E^\infty$ by $\sigma(x_1 x_2 x_3 \dots) = x_2 x_3 \dots$. Let $G_E$ be the Deaconu-Renault groupoid corresponding to $\sigma$. By \cite[Proposition 4.1]{KPRR}, the graph algebra $C^*(E)$ and $C^*(G_E)$ are isomorphic.
\end{example}

\begin{example}
\cite[Section 10.3]{KatsuraII} Assume $X$ is compact. Let $d = \sigma$ and let $r$ be the identity map on $X$. Set $E = (E^0, E^1, d, r)$. Then $E$ is a topological map. Moreover, the C*-algebra $C^*(\mathcal{G})$ of the Deaconu-Renault groupoid $\mathcal{G}$ is isomorphic to the C*-algebra $C^*(E)$ of the topological graph $E$.
\end{example}







\bibliographystyle{acm}
\bibliography{bibliography}

\end{document}